\documentclass[10pt]{article}

\usepackage{amsmath}
\usepackage{amsthm}
\usepackage{amssymb}
\usepackage{graphicx}
\usepackage{hyperref}

\newtheorem{thm}{Theorem}

\newtheorem{remark}{Remark}

\newtheorem{lemma}[thm]{Lemma}
\newtheorem{prop}[thm]{Proposition}
\newtheorem{cor}[thm]{Corollary}

\newcommand{\beq}{\begin{equation}}
\newcommand{\eeq}{\end{equation}}
\newcommand{\beqn}{\begin{eqnarray}}
\newcommand{\eeqn}{\end{eqnarray}}
\newcommand{\bal}{\begin{align}}
\newcommand{\eal}{\end{align}}
\newcommand{\cO}{c_0}

\newcommand{\cL}{{\mathcal L}}
\newcommand{\bM}{{\mathbb M}}
\newcommand{\M}{{\mathcal M}}
\newcommand{\R}{{\mathbb R}}
\newcommand{\Rp}{{\mathbb R}_+}

\renewcommand{\O}[1]{\mathrm{O}\lb #1\rb}

\def\d{{\mathrm d}}
\def\R{{\mathbb R}}
\newcommand{\RR}{\mathbb{R}}

\def\11{{\mathbbmss 1}}

\def\e{{\mathrm e}}

\newcommand{\eps}{\epsilon}
\def\eps{{\varepsilon}}

\def\l{\lambda}
\def\lam{\lambda}

\def\Ran{{\mathrm{Ran}\;}}

\newcommand{\n}{\nabla}
\newcommand{\p}{\partial}

\newcommand{\grad}{{\rm grad\,}}
\newcommand{\Tr}{\mathrm{Tr}\;}

\newcommand{\scalar}[2]{\langle{#1} \mspace{2mu}, {#2}\rangle}
\newcommand{\ip}[2]{\left\langle #1, #2 \right\rangle}

\def\({\left(}
\def\){\right)}
\newcommand{\lb}{\left(}
\newcommand{\rb}{\right)}
\newcommand{\lsb}{\left[}
\newcommand{\rsb}{\right]}

\newcommand{\norm}[1]{\lVert #1 \rVert}

\newcommand{\lra}[1]{\langle#1\rangle}

\newcommand{\DETAILS}[1]{}

\title{On Blowup in Nonlinear Heat
Equations\thanks{Supported by NSERC under Grant NA7901.}}
\date{November 22, 2011}
\begin{document}

\author{D. Egli\thanks{Department of
Mathematics, University of Toronto, Toronto, Canada.}, Z. Gang\thanks{ETH, Z$\ddot{\text{u}}$rich.}, W. Kong$^\dagger$, I.M.
Sigal$^\dagger$}

\maketitle

\begin{abstract}
  We establish the asymptotics of blowup for nonlinear heat equations with superlinear power nonlinearities in arbitrary dimensions and we estimate the remainders.
\end{abstract}

\section{Introduction} \label{SEC:Intro}
In this paper we study the blowup problem for the $n$-dimensional
nonlinear heat equation (or the reaction-diffusion equation) %with non-radial initial datum
\begin{equation}\label{NLH}
\left\{\begin{array}{lll}
\p_t u&=&\Delta u+|u|^{p-1}u\\
u(x,0)&=&u_{0}(x)
\end{array}\right.
\end{equation} with $p>1$. Here $u:\R^n\times\R^{+}\rightarrow\R$.
Eq. \eqref{NLH} arises in the problem of heat flow, or, more generally, in the problems involving diffusion, %and the theory of chemical reactions
and is a model for a large class of nonlinear parabolic equations, which are ubiquitous in mathematics and its applications. %Similar equations appear in the motion by mean curvature flow (see \cite{SS}), vortex dynamics in superconductors (see \cite{CHO}, \cite{MZ2}), surface diffusion (see \cite{BBW}) and chemotaxis (see \cite{BCKSV,BB}).

%The local well-posedness of \eqref{NLH} is well known (see, e.g.\ \cite{Ball} for the Sobolev spaces $\mathcal{H}^\alpha$, $0\le\alpha<2$).
We will deal, without mentioning it, with weak solutions of Eq. (\ref{NLH}) in the sense detailed in the next section. The local existence of such solutions is well known (see, e.g.\ \cite{Ball} for the Sobolev spaces $H^\alpha$, $0\le\alpha<2$) and is presented for readers' convenience in the next section. These solutions can be shown to be classical for $t>0.$

For some data $u_{0}(x)$, the solutions $u(x,t)$ might blow up in finite
time $t^*>0$, %. In what follows a solution $u(x,t)$ is said to blow up at time $t^*$ if it exists
i.e. they exist in $L^\infty$ for $[0,t^*)$ and
$\sup_{x}|u(x,t)|\rightarrow\infty$ as $t\rightarrow t^*$.
Thus, two key problems about \eqref{NLH} are
\begin{enumerate}
 \item Describe initial conditions for which solutions of Eq.
\eqref{NLH} blow up in finite time;
 \item Describe the blowup profile of such solutions.
\end{enumerate}

It is expected (see e.g.\ \cite{BrKu}) that the (stable) blowup profile is
universal---it is independent of lower power perturbations of the
nonlinearity and of initial conditions within certain spaces.

The following key properties of equation \eqref{NLH} %has the following properties:explain
 elucidate important features of the %expressions above:
results we discuss below:
\begin{itemize}
\item \eqref{NLH} is invariant
with respect to the scaling transformation,
\begin{equation}\label{rescale}
u(x,t)\rightarrow \lambda^{\frac{2}{p-1}} u(\lambda x,\lambda^{2} t)
\end{equation} for any constant $\lambda>0,$ i.e. if
$u(x,t)$ is a solution, so is $\lambda^{\frac{2}{p-1}}u(\lambda
x,\lambda^{2}t).$
\item \eqref{NLH} has $x-$independent (homogeneous)
solutions:
\begin{equation}\label{uhom}
u_{\rm hom}=[u_{0}^{-p+1}-(p-1)t]^{-\frac{1}{p-1}}.
\end{equation}
\end{itemize}

Solutions \eqref{uhom} blow up in finite time $t^* = \lb (p-1)u_0^{p-1}\rb^{-1}$ for $p>1$.
The linearization of \eqref{NLH} around $u_{\rm hom}$ shows that the
solution $u_{\rm hom}$ is unstable.  Moreover, it is shown in \cite{GK1}
that if either $n\le 2$ or $p\le (n+2)/(n-2)$, then the equation
\eqref{NLH} has no other self-similar solutions of the form
$(T-t)^{-\frac{1}{p-1}}\phi\lb x/\sqrt{T-t}\rb$, $\phi\in
L^{\infty}$, besides $u_{\rm hom}$.

We consider \eqref{NLH} with initial conditions in a certain neighbourhood of the homogeneous solution, which have, modulo a small perturbation, a maximum at the
origin, are slowly varying near the origin and are sufficiently
small, but not necessarily vanishing, for large $|x|$.
We show that the solutions of \eqref{NLH} for
such initial conditions blowup at a finite time $t^*$ and at some moving point $\zeta(t)$ and we characterize the blowup asymptotics.
%dynamics of these solutions.
\DETAILS{
As it turns
out, the leading term is given by the expression
\begin{equation}
\lambda(t)^\frac{2}{p-1}\lsb\frac{
c(t)}{p-1+\lambda^2(t)(x-\zeta(t)) b(t)
(x-\zeta(t))}\rsb^\frac{1}{p-1}
\label{leading}
\end{equation}
where $b(t)$  a real, symmetric $n\times n$-matrix $b(t)=(b_{ij}(t)), $  $y b y:=\sum_{i,j=1}^n y_i b_{ij}y_j$ for a $n\times n$-matrix $b:=(b_{ij})$ and  the parameters $\lambda(t)$, $b(t)$, $c(t)$ and $\zeta(t)$ obey certain
dynamical equations whose solutions give
\begin{equation}\label{eq:blowdy}
\begin{array}{lll}
\lambda(t)&=(t^*-t)^{-\frac{1}{2}}(1+o(1))\\
b(t)&=\frac{(p-1)^{2}}{4p|\ln|t^*-t||}(I+O(\frac{1}{|\ln|t^*-t||^{1/2}}))\\
c(t)&=1-\frac{p-1}{2p|\ln|t^*-t||}(1+O(\frac{1}{\ln|t^*-t|}))\\
\zeta(t)&=O(1).
\end{array}
\end{equation}
with $\lambda(0)=\sqrt{\cO+\frac{2}{p-1}\Tr b(0)}$,
$c_0>0,\ b(0)>0$ depending on the initial datum. Here $o(1)$ is in
$t^*-t$. Moreover, we estimate the remainder, the difference between
$u(x,t)$ and \eqref{leading}. %In our approach, as in \cite{DGSW}, we do not fix the time-dependent scale in the self-similarity (blowup) variables but let its behavior, as well as behavior of other parameters ($b$ and $c$) be determined by the equation.

%We will deal, without mentioning it, with weak solutions of Equation (\ref{NLH}) in the sense detailed in the next section. The local existence of such solutions is well known and is presented for readers' convenience in the next section. These solutions can be shown to be classical for $t>0.$

In what follows we use}
More precisely, with the standard notation  $y b y:=\sum_{i,j=1}^n y_i b_{ij}y_j$ for a $n\times n$-matrix $b:=(b_{ij})$, $\langle x\rangle:=(1+|x|^2)^{1/2}$ and $f\lesssim g$ for two positive functions $f$ and $g$, satisfying $f\leq Cg$ for some universal constant $C$, we have the following result.
%The main result of this paper is the following theorem.
\begin{thm}\label{maintheorem}
Let $b_0:=(b_{0ij})>0$ be a real, symmetric, positive $n\times n$-matrix with $\|b_{0}\| \ll 1$ and $1\leq \cO\leq 4$.
Suppose the initial data $u_{0}\in L^\infty(\R^n)$  for \eqref{NLH}
satisfy the conditions
\begin{equation}\label{INI2}
\left\|\langle x\rangle^{-m}\lb
u_{0}(x)-\lb\frac{\cO}{p-1+x b_{0}x}\rb^{\frac{1}{p-1}}\rb\right\|_{\infty}\le
\delta_{m},
\end{equation}
with $m= 0, 3$,  $0 \leq \delta_{0}\ll 1$ and $\delta_3 =
C\|b_0\|^2$.
Then
\begin{enumerate}
\item[(1)] There exists a time $t^*\in (0,\infty)$ such that the
solution $u(x,t)$ exists on the interval $[0, t^*)$ and blows up at $t^{*}.$
\item[(2)] For
$t< t^*$ there exist unique, $C^1$, positive, real valued functions $\lambda(t)$ and $c(t)$, a $C^1$ $n$-vector valued function $\zeta(t)$, and a $C^1$ $n\times n$-symmetric-matrix valued function $b(t)$, with $b(t)\lesssim b(0),
$ such that $u(x,t)$ can be written as
\begin{equation}\label{u-as}
u(x,t)=\lambda^{\frac{2}{p-1}}(t)\big[\big(\frac{c(t)}
{p-1+y b(t)y}\big)^{\frac{1}{p-1}}+\xi(x,t)\big],\end{equation}
where $y:=\lambda(t)(x-\zeta(t))$ %$b(t)$  a real, symmetric $n\times n$-matrix $b(t)=(b_{ij}(t)), $  $y b y:=\sum_{i,j=1}^n y_i b_{ij}y_j$ for a $n\times n$-matrix $b:=(b_{ij})$
and the fluctuation part, $\xi,$ admits the estimates %$\|\langle\lambda(t)(x-\zeta(t))\rangle^{-m}\xi(x,t)\|_{\infty}
$\|\langle y\rangle^{-m}\xi(x,t)\|_\infty\lesssim \delta_m(t),$
$m=0, 3$.  Here $\delta_0 (t)=\delta_0 \ll 1$ and $\delta_3 (t)= \|b(t)\|^{2}$.
\item[(3)] The parameters $\lambda(t)$, $b(t)$, $c(t)$ and $\zeta(t)$ obey certain
dynamical equations (with initial conditions $\lambda(0)=\sqrt{\cO+\frac{2}{p-1}\Tr b(0)}$,
$c_0>0,\ b(0)>0$, depending on the initial datum),  whose solutions give
\begin{equation}\label{eq:blowdy}
\begin{array}{lll}
\lambda(t)&=(t^*-t)^{-\frac{1}{2}}(1+o(1))\\
b(t)&=\frac{(p-1)^{2}}{4p|\ln|t^*-t||}(I+O(\frac{1}{|\ln|t^*-t||^{1/2}}))\\
c(t)&=1-\frac{p-1}{2p|\ln|t^*-t||}(1+O(\frac{1}{\ln|t^*-t|}))\\
\zeta(t)&=O(1).
\end{array}
\end{equation}
Here $o(1)$ is in $t^*-t$.
%\item[(3)] The functions $\lambda(t)$, $b(t)$, $c(t)$ and $\zeta(t)$ are of the form \eqref{eq:blowdy}.
\end{enumerate}
\end{thm}

\noindent\textbf{Remarks.} 1) %We note that
Neither smoothness of initial conditions nor decay at infinity are required. In
particular, the energy %${\cal E}(u)$
\begin{equation}\label{eqn:L2Energy}
{\cal E}(u):= \int \(\frac{1}{2}|\nabla u|^2-\frac{1}{p+1}|u|^{p+1}\)\d^nx,
\end{equation}
for such initial conditions might be infinite.

%A remark to
2) The weight in the $L^\infty$-norm for $\xi(x,t)$ %used in \eqref{INI2} is in order. As mentioned in statement (2) of the theorem, we will
 comes from the fact that we decompose a solution of (\ref{NLH}) into a leading profile and a fluctuation, with the fluctuation orthogonal to bad (positive or nearly zero) eigenvalues of the linearization around  the leading profile. The weight in question is determined by the eigenfuction of the first good (negative) eigenvalue. (Note that bad eigenvalues reflect the instabilities w.r.\ to the blowup time and center and the shape and size of the blowup profile.)
 %The weight needed in the $L^\infty$-norm is determined by the linearized problem for the fluctuation. The key point is that the number of ``bad'' eigenvalues of the linearized problem determines two things: First, the number of parameters of the leading term (must be $\geq$ number of ``bad'' eigenvalues); and second, the weight needed in the $L^\infty$-norm. It corresponds to the first ``good'' eigenfunction.

\DETAILS{\textbf{Note that the function, $v_{ab}:=\lb\frac{2a}{p-1 +  y b y}\rb^\frac{1}{p-1}$, on the r.h.s. of the expression for $u(x,t)$ is the general solution of the equation $a y\cdot\nabla_yv+\frac{2 a}{p-1} v=v^p$. Extend this to more general nonlinearities and to systems with  nonlinearities such as  $\sum_{\alpha:|\alpha|=p} c_\alpha \prod_{i=1}^m u_i^{\alpha_i},\ \alpha=(\alpha_1, \dots, \alpha_m)$.}}

There is rich literature regarding the blowup problem for equation
\eqref{NLH}. We review quickly the relevant results.  Starting with
\cite{Fu}, various criteria for blowup in finite time were derived,
see e.g.\ \cite{Fu,Ball,CHI,EVA,Le1,Le2,OHLM,Qu,Sou,FHV,
FM2}. For example, if $u_{0}\in H^{1}\cap L^{p+1}$
and $\mathcal{E}(u_{0})<0$, %where ${\cal E}(u)$ is the energy functional for \eqref{NLH} given by
then it is proved in \cite{Le1} that $\|u(t)\|_2^{2}$ blows up in finite
time $t^*$. By the observation
$$\frac{1}{2}\frac{d}{dt}\|u(t)\|_2^{2}\leq
\|u(t)\|_\infty^{p-1}\|u(t)\|_2^{2}$$ we have that $\|u(t)\|_\infty$
blows up in finite time $t^{**}\leq t^*$ also.  (In this paper, we
denote the norms in the $L^p$ spaces by $\|\cdot\|_p$.)

Blowup at a single point was studied as early as \cite{Wei1} (see also
~\cite{FM2}). The first result on asymptotics of the blowup
for arbitrary dimension $n\geq 1$ was obtained in the pioneering
paper \cite{GK1} where the authors show that under the condition
\begin{equation}\label{eq:boundedness}
|u(x,t)|(t_*-t)^\frac{1}{p-1}\ \mbox{is bounded on}\ B_1\times
(0,t_*),
\end{equation}
where $B_1$ is the unit ball in $\R^n$ centred at the origin, and
either $p\le\frac{n+2}{n-2}$ or $n\le 2$ and assuming blowup takes
place at $x=0$, one has
\begin{equation*}
\displaystyle\lim_{\lambda\rightarrow 0} \lambda^\frac{2}{p-1}
u(\lambda x,t_{*}+\lambda^2(t-t_*) )=\pm
\lb\frac{1}{p-1}\rb^\frac{1}{p-1}(t_*-t)^{-\frac{1}{p-1}}\
\mbox{or}\ 0.
\end{equation*}
This result was further improved in several papers (see e.g.\
\cite{GK2, GK3, HV92, FK, Mer,Vel,FL,
FM1,FM2,BrKu, MZ3, MZ4, MZ5}). A blowup solution
satisfying the bound \eqref{eq:boundedness} is said to be of type
I. This bound was proven under various conditions in \cite{GK2, MZ3,
MZ4, Wei2, GMS}.  Furthermore, the limits of
$H^1$-blowup solutions $u(x,t)$ as $t\uparrow T$, outside
the blowup sets were established in \cite{HV92,FK, Mer,
Vel,FL,FM1,FM2,BrKu, MZ5, FMZ}.

For $p>1$, dimension $n=1$, Herrero and Vel\'{a}zquez \cite{HV2}
(see also ~\cite{FL}) proved that if the initial condition
$u_0$ is continuous, nonnegative, bounded, even, and has only one
local maximum at $0$, and if the corresponding solution blows up,
then
\begin{equation}\label{eq:asy}
\lim_{t\uparrow t^*}(t^*-t)^{\frac{1}{p-1}}u(y((t^*-t)\ln|
t^*-t|)^{\frac{1}{2}},t)=(p-1)^{-\frac{1}{p-1}}[1+\frac{p-1}{4p}y^{2}]^{-\frac{1}{p-1}},
\end{equation}
uniformly on sets $|y|\leq R$ with $R>0$.  Further
extensions of this result are achieved in \cite{HV92, Vel,
FK, FL}.

Later for dimension $n=1$ Bricmont and Kupiainen \cite{BrKu}
constructed a co-dimension 2 submanifold of initial conditions such
that \eqref{eq:asy} is satisfied on the whole domain. More
specifically, given a small function $g$ and a small constant $b>0$,
they find constants $d_{0}$ and $d_{1}$ depending on $g$ and $b$
such that the solution to \eqref{NLH} with the datum
\begin{equation}\label{dis}
u_{0}^{*}(x)
=(p-1+bx^{2})^{-\frac{1}{p-1}}(1+\frac{d_{0}+d_{1}x}{p-1+bx^{2}})^{\frac{1}{p-1}}+g(x)\end{equation}
has the convergence \eqref{eq:asy} uniformly in
$y\in(-\infty,+\infty)$. The result of \cite{BrKu} was generalized
in \cite{MZ1} (see also \cite{GaPo1986}), where it is shown that
there exists a neighborhood $\mathcal{U}$, in the space
$%\mathcal{H}:=
L^{p+1}\cap H^{1}$, of $u_{0}^{*}$, given in
\eqref{dis}, such that if $u_{0}\in \mathcal{U}$, then the solution $u(x,t)$
blows up in a finite time $t^*$ and satisfies (\ref{eq:asy}) for
$x\in \mathbb{R}$. They conjectured that this asymptotic behavior is
generic for any blowup solution.

For  initial conditions in $L^\infty$ that lead to blowup at a prescribed location and time, $a, T$, respectively, with the blowup profile \eqref{eq:asy}, Merle et al.\ %have thoroughly analyzed the phenomenon of blowup in the nonlinear heat equation \eqref{NLH}, see
(\cite{FMZ,Mer,MZ1,MZ3,MZ4,MZ5}) %. They
established the stability of the blowup profile %given in \eqref{eq:asy}
in any dimension. %for initial conditions in $L^\infty$ (that is, if the equation with initial condition $\tilde{u}_0$ blows up at 0 (say) and at time $t^*$, with the blowup profile \eqref{eq:asy}, then there is a neighbourhood $\mathcal{U}$ of $\tilde{u}_0$ in $L^\infty$ such that equation \eqref{NLH} with initial condition in $\mathcal{U}$ blows up at $a\simeq 0$ and $T\simeq t^*$, with the same blowup profile). %Their approach differs from ours in an essential way, though. They investigate initial conditions that are required to blow up at a prescribed location and time, $a, T$, respectively. %Furthermore,  by fixing

In \cite{DGSW} precise blowup asymptotics were derived for
\eqref{NLH} in dimension $1$ for even initial conditions. 
Our results extend the results of \cite{DGSW} in two aspects. First, we address the problem of blowup in arbitrary dimensions.
Second, we consider more general, non-symmetric initial conditions, which allow the blowup center to move.
%This work developed a different approach based on dynamical rescaling, method of majorants and strong linear estimates. The results of \cite{DGSW} were extended to general initial conditions and consequently to moving blowup point in \cite{WK}. Our results and techniques extend those of \cite{DGSW}.

F. Merle (\cite{Mer2}) has informed the last author that asymptotics \eqref{u-as} - \eqref{eq:blowdy}, but without estimates of the remainders, can be derived from \cite{FMZ,FMZ2,GK2, GMS,Ma, MuW, Vel2}.  %Our results overlap with those of \cite{Mer,MZ1,MZ3,MZ4,MZ5} and of \cite{DGSW}.  However,with respect to the first group of papers, we deal with explicit (and important) open set of initial conditions and we give estimates of the remainders. As far as \cite{DGSW} is concerned,

%Note that our approach is closest to that of \cite{DGSW} (see also \cite{zouthesis}).
%(With the $L^2(\R^n)$ metric, $\grad {\cal E}$ is defined by the relation $\p{\cal E}(u)\xi=\ip{\grad{\cal E}(u)}{\xi}$, so that $\grad {\cal E}(u)=-(\Delta u+|u|^{p-1}u).) $ We immediately have that the energy ${\cal E}$ decreases under the flow of \eqref{NLH}.
%also note
Unlike the most of the works above, we do not use the fact that \eqref{NLH} is an $L^2$-gradient system $$\p_t u = -\grad {\cal E}(u),$$
with the energy defined in \eqref{eqn:L2Energy}.
%Most of the works mentioned above use this fact in an essential way. It is not used in this paper,
Instead we use %work developed a different approach based on dynamical rescaling,
method of majorants, which allow us to bootstrap our estimates, and strong linear estimates. Hence we expect our analysis can be extended to non-gradient systems.

Also, in contrast all previous works, with exception of \cite{DGSW}, which fix scaling as $\lambda (t)=(T-t)^{-\frac{1}{2}}$, where $T$ is the blowup time, we leave the scaling, $\l(t)$ (and blowup center, shape and size parameters, $b$ and $c$, and time) to be determined by the equation.
%Specific forms of initial conditions and nonlinearity enter into only two places: orthogonal decomposition and Lyapunov-Schmidt splitting (Sections \ref{Section:Reparam} and \ref{Section:Splitting}).
%
\DETAILS{Also, unlike all previous works with exception of \cite{DGSW}, which fix ``similarity variables''
\begin{align*}
y=\frac{x-a}{\sqrt{T-t}}\,,\quad s=-\log(T-t),
\end{align*}
%they restrict their attention to a certain type of blowup. Our method on the other hand
our leaves the scaling, $\l(t)$, (and blowup center, shape and size parameters, $b$ and $c$, and time) to be determined by the equation.}
%In our approach, as in \cite{DGSW}, we do not fix the time-dependent scale in the self-similarity (blowup) variables but let its behavior, as well as behavior of other parameters ($b$ and $c$) be determined by the equation.
As a result we obtain a dynamical system for the scaling parameter $\l(t)$, as well as for other parameters determining the leading profile, solving which gives the desired scaling law. Hence our approach  is well adapted to detecting %We believe this approach is more flexible and allows to obtain
the scaling dynamics in situations where scaling law is not obvious %like in the cases of critical Yang - Mills and wave map equations
(see e.g. \cite{BOS,OS, RS, RR, KST1, KST2}).  %and is therefore applicable to situations where we expect ``non-generic'' blowup behaviour, such as \begin{itemize}\item \item   \end{itemize}

We %expect also that our results  can be
believe our techniques are sufficiently simple and robust and can more or less straightforwardly be extended to  $p<0$ (collapse, see \cite{Ga}), to several blowup centers, to blowups along spheres and to more general, say polynomial, nonlinearities.
\DETAILS{Let us also remark that for simplicity we consider single-point blowup, for $p>1$. Our results can more or less straightforwardly be extended to the cases
\begin{itemize}
\item $p<0$ (collapse, see \cite{gangcollapse})
\item several blowup centers
\item blowup along spheres.
\end{itemize}}

Our proof is close to the one of \cite{DGSW} but several points are substantially revised and the exposition is simplified. Since the problem is important and our treatment is still simpler than anything presented so far in the literature, we give, for the reader's convenience a complete proof, reproducing some of the results of \cite{DGSW}.

Our approach consists of the following sequence of steps:
%\newpage
\begin{itemize}
\item Passing to blowup variables (Section \ref{Section:BV}).  Given  differentiable functions $z(t)\in \R^{n}$ and $ \lambda(t)>0,$ we pass to new variables as
$$v(y,\tau):=\lambda^{-\frac{2}{p-1}}(t) u(x,t),\ \quad \mbox{where}\  \quad y:=\lambda(t)(x-z(t))-\alpha (t)\ \mbox{and}\ \tau:=\int_{0}^{t}\lambda^{2}(s)ds.$$
Here $\alpha (t)$ satisfies the equation
$\lambda^{-2}\dot \alpha -a\alpha=-\lambda^{-1} \dot z,$
 with $a(t)=\dot{\lambda}(t)/\lambda^3(t)$.
%In the new variables our solution $u(x,t)$ is rewritten as $ v(y,\tau):=\lambda^{-\frac{2}{p-1}}(t) u(x,t).$
Now $\lambda(t), a(t), z(t)$ and $v(y,\tau)$ are unknowns we have to solve for.
\item Reparametrization of solutions (Section \ref{Section:Reparam}).  The  equation for $v(y,\tau)$, which follows from \eqref{NLH}, has the two-parameter family of approximate solutions
\begin{equation}
v_{ab}:=\lb\frac{2a}{p-1 +  y b y}\rb^\frac{1}{p-1},
\end{equation}
where $b:=(b_{ij}),\ b_{ij}\in\R$, is any real, symmetric $n\times n$-matrix and, recall, $y b y:=\sum_{i,j=1}^n y_i b_{ij}y_j$:
%Denote the map on the r.h.s. of \eqref{eqn:BVNLH} by $F_{a}$:
\begin{equation}\label{appr-eq}
%F_{a}(v):=
\lb\Delta- a y\cdot \n
-\frac{2a}{p-1}\rb v_{ab}+|v_{ab}|^{p-1}v_{ab}\approx 0.
\end{equation}
%Here we think of $a$ as some time-independent parameter (adiabatic approximation).
%Then $v_{ab}$ %:=v_{bc}|_{c=2a}$ and $b$ symmetric
% is an approximate solution of the `stationary' equation:
%\begin{equation}\label{eqn:FE}
%$F_{a}(v_{ab})\approx 0.$ %\end{equation}
In what follows we take
$b\ge 0,$ so that $v_{ab}$ is nonsingular.

It will turn out (see below in this outline) that $a$ approaches $1/2$, as $t$ approaches the blowup time, and it will be convenient to replace $v_{ab}$ by
$V_{ a b }(y):=(\frac{a+1/2}{p-1+ y b y})^{\frac{1}{p-1}}$.
We consider  %the two-parameter family  $%V_{ a b }(y)=V_{bc}:=(\frac{c}{p-1+ y b y})^{\frac{1}{p-1}}$ and
the manifold
$$\mathcal{M}_{\rm as}:=\{V_{ab}\, |\, a\in \Rp, b\in \R^{n\times n}\}$$ of almost solutions.
We %split solutions $v(y,\tau)$ into the leading term---the almost solution $%V_{ a b }(y)=V_{bc}:=(\frac{c}{p-1+ y b y})^{\frac{1}{p-1}}$%,with $c=a+\frac{1}{2}$, ---and a fluctuation $\xi$ around it.  More precisely, we would like to
 parameterize a solution by a point on the manifold
$\mathcal{M}_{\rm as}$ and a fluctuation (approximately) orthogonal
to this manifold:  %Here $a=a(b,c)$ is a twice differentiable function of $b$ and $c$.
\begin{equation}\label{deco}
v=V_{a b}+\xi,\ \qquad \xi\bot\ T_{V_{a b}}\mathcal{M}_{\rm as}, %\phi_{a}^{(ij)}, \, \, 0 \leq i, j \leq n, \label{eqn:split3}
\end{equation}
 in the sense of $L^2(\R^n,\e^{-a|y|^2/2}\d y)$ (large slow moving and small fast moving parts of the solution). %, where $V_{ab}:=\lb \frac{c}{p-1+y b y} \rb^{\frac{1}{p-1}}$, $c=a+\frac{1}{2}$ and $\phi_{a}^{(ij)}$ are defined in the beginning of Section \ref{Section:Reparam}.

\item Lyapunov-Schmidt decomposition (Section \ref{Section:Splitting}).
 Plugging the decomposition \eqref{deco} into the equation for $v(y, \tau)$ gives the equation
\begin{equation} \label{xi-eq}
\xi_\tau=-{\mathcal L}_{ab}\xi+{\mathcal N}(\xi,a,b)+{\mathcal F}(a,b)
\end{equation}
where ${\mathcal L}_{ab},\ {\mathcal N}(\xi,a,b)$ and ${\mathcal F}(a,b)$ are the linear operator, the nonlinearity and the source term respectively.

\DETAILS{For technical reasons, it is more convenient to
require the fluctuation to be almost orthogonal to the manifold
$\M_{\rm as}$. More precisely, we require $\xi$ to be orthogonal to the
vectors $\phi^{(ij)}_{a},\, 0\leq i,
j\leq n$, where
% $$\phi^{(00)}_{ a}(y):=e^{-\frac{a}{4} |y|^2},\ %and $\phi^{(i)}_{2a}:=(1-a y_i^2)e^{-\frac{a}{4} |y|^2}$ Note that $\xi$ is already orthogonal to
% \phi^{(0i)}_{a}(y)=\phi^{(i0)}_{ a}(y):=\sqrt{a}
% y_i e^{-\frac{a}{4} |y|^2},\
% \phi^{(ij)}_{a}(y):=ay_iy_je^{-\frac{a}{4} |y|^2},\, 1\leq i,
% j\leq n,$$
$$\phi^{(00)}_{ a}(y):=1,\ %and $\phi^{(i)}_{2a}:=(1-a y_i^2)e^{-\frac{a}{4} |y|^2}$ Note that $\xi$ is already orthogonal to
\phi^{(0i)}_{a}(y)=\phi^{(i0)}_{ a}(y):=\sqrt{a}
y_i ,\
\phi^{(ij)}_{a}(y):=ay_iy_j,\, 1\leq i,
j\leq n,$$
which are almost tangent vectors to the above manifold, provided $b$ is sufficiently small.}
%since our initial conditions, and therefore, the solutions are even in each $x_i$. T
%his is the only place where the condition \eqref{symm} is used.
%\paragraph{\bf Remark.}

Differentiating the equation \eqref{appr-eq} w.r.\ to  $a$, $ z$ (remember, $ y:=\lambda(t)(x-z(t))$) and $b$, %denoting the Fr\'echet derivative (linearization) of $F_{a}$ at $v_{ab}$ by  $L_{a b }$
 and using that $V_{a b}=(p-1)a \p_a V_{a b}$ and
$y\cdot \n V_{a b}=\frac{1}{p-1}\frac{2yby}{p-1+yby} V_{a b}=\frac{2ayby}{p-1+yby} \p_a V_{a b}$ and $\n V_{a b}=\lam^{-1}\n_z V_{a b}$, we obtain %(it turns out that in our analysis we should take $\frac{\partial d}{\partial \rho} \approx a$)
\begin{equation}\label{eqn:zeromodes}
\cL_{a b }(\p_a V_{a b})\approx a(1+\frac{yby}{p-1+yby}) \p_a V_{a b},\ \quad
\cL_{a b }(\n_z V_{a b})\approx a \n_z V_{a b},\ \quad
\cL_{a b }(\p_{b_{ij}}V_{a b})\approx  0 .
\end{equation}
Since for $|y|$ bounded and $\| b\|$ small, $V_{a b}\approx \lb\frac{a+1/2}{p-1}\rb^{\frac{\mu}{p-1}}$ and therefore
$$ \p_a V_{a b} \approx \frac{1}{a},\  \quad
\n_{z_j} V_{a b} \approx \lam\mu\sum_j b_{ij}y_j,\ \quad \textrm{and}\ \quad
 \p_{b_{ij}}V_{a b} \approx \mu y_iy_j,$$
where $\mu:=\frac{1}{p-1}\lb\frac{a+1/2}{p-1}\rb^{\frac{1}{p-1}}$. Hence we expect that the linearized operator $\cL_{a b }$ has approximate eigenvalues $2a,\ a$ and $0$ with the corresponding approximate eigenfunctions of $1,\ y_j$ and $y_iy_j$, which are approximate tangent vectors to $\mathcal{M}_{\rm as}$ at $V_{a b}$ spanning $T_{V_{a b}}\mathcal{M}_{\rm as}$.

The first two groups of approximate eigenfunctions are related to the scaling and translation symmetry of the original nonlinear heat equation \eqref{NLH}. The third one can be thought of as related to the symmetry w.r.\ to rotations. %the time translations. Indeed, the latter symmetry allows us to encode the time-dependence of the leading term in the parameters $a,\ z$ and $b$ only. Now, if we think of the parameters $a,\ z$ and $b$ as dependent on time $\tau$, then differentiating \eqref{eqn:FE} w.r.\ to $\tau$ and taking into account the first and second equations in \eqref{eqn:zeromodes} leads to the third equation in \eqref{eqn:zeromodes}.
The approximate eigenfunctions above give the unstable modes in our problem and they will play an important role in our analysis.

Consider now the family $\tilde V_{ bc }(y):=(\frac{c}{p-1+ y b y})^{\frac{1}{p-1}}$, with $c$ an extra parameter, and proceed as above, using the decomposition $v=\tilde V_{ bc } +\xi$, instead of \eqref{deco}. Projecting  the resulting equation for $\xi$ onto approximate
$T_{\tilde V_{ bc }}\mathcal{M}_{\rm as}$, we find the following dynamical system for the parameters $a,b,c$:
\begin{align}
\partial_\tau c&=c(c-2a)-\frac{2}{p-1}\Tr b+\mathrm{Rem}_c(\xi,a,b,c)\\
\partial_\tau b&=(c-2a)b-\frac{2b}{p-1}\Tr b+\frac{4p}{(p-1)^2}b^2+\mathrm{Rem}_b(\xi,a,b,c)\,,
\end{align}
for some remainders $\mathrm{Rem}_c(\xi,a,b,c)$ and $\mathrm{Rem}_b(\xi,a,b,c)$ which are expected to provide higher order corrections. Note that we are free to choose the (time-dependent) additional parameter $c$ at our convenience. From the above equations we read off the equilibria (zeroes of the vector field governing the evolution of the parameters $a, b, c$) as
\begin{align*}
(a,b,c)=(a^*,0,2a^*)\,,
\end{align*}
for any choice of function $a^*$. The fixed point we want the parameters to flow to is
\begin{align*}
(a,b,c)=(\frac{1}{2},0,1)\,.
\end{align*}
One way to achieve this is to fix $c$ as a convex combination of $1$ and $2a$:
$c=\rho+2(1-\rho)a\,,$
for any $\rho\in (0,1)$. Note that the extremal point $\rho=0$ is not a good choice because the equation for $c_\tau$ would lose its leading part driving $a$ and $c$ to the desired fixed point, while $\rho=1$ robs us of an equation for $a_\tau$.
The simplest choice is $\rho=\frac{1}{2}$, so that
\begin{align*}
c=\frac{1}{2}+a \quad \textnormal{and } c-2a=\frac{1}{2}-a\,.
\end{align*}
%\textit{For the remainder of this article, we fix the relation between $a$ and $c$ as $c=\frac{1}{2}+a$.}
This is exactly our reason for using $ V_{ ab }(y):=(\frac{a+1/2}{p-1+ y b y})^{\frac{1}{p-1}}$, instead of $v_{ab}:=\lb\frac{2a}{p-1 +  y b y}\rb^\frac{1}{p-1}$.

\item Linear propagator estimates (Section \ref{Section:PropEst}). Using combination of techniques we derive estimates of the propagators generated by the operator ${\mathcal L}_{ab}$ in the norms introduced above.

\item Majorants and bootstrap (Sections \ref{Section:APriori}, \ref{SEC:EstB}, \ref{SEC:EstM1}, \ref{SEC:EstM2}). To control the fluctuations $\xi(\tau),$ we introduce the  estimating functions (families of semi-norms)
\begin{equation*}%\label{majorants'}
%\begin{array}{lll}
M_{k}(T):=\max_{\tau\leq T}
\beta^{-2(2-k)}(\tau)\|\lra{y}^{-3(2-k)}\xi(\tau)\|_{\infty},\ k=1, 2,
%M_{2}(T)&:=\max_{\tau\leq T}\|\xi(\tau)\|_{\infty},\\
%A(T)&:=\max_{\tau\leq T}\beta^{-2}(\tau)\left|a(\tau)-\frac{1}{2}+\frac{2\Tr b(\tau)}{p-1}\right|,\\
%B(T)&:=\max_{\tau\leq T}\beta^{-(1+\kappa)}(\tau)\|b(\tau)-\tilde{\beta}(\tau)\|.
%\end{array}
\end{equation*}
and similarly for the parameters $b(\tau)$ and $a(\tau)$. Using \eqref{xi-eq} and the linear propagator estimates, we prove inequalities for these estimating functions, which allow us to bootstrap our estimates, starting from very rough ones provided by the local well-posedness. This allows us to propagate our estimates in time.
\end{itemize}

%%%%%%%%%%%%%%%%%%%%%%%%%%%%%%%%%%%%%%%%%%%%%%%%%%%%%%%%%%%%%%%%%%%%%%%%
%%%%%%%%%%%%%%%%%%%%%%%%%%%%%%%%%%%%%%%%%%%%%%%%%%%%%%%%%%%%%%%%%%%%%%%%
%\section{The Local Well-Posedness of \eqref{NLH}}
We conclude the introduction by stating without proof the standard result on the local well-posedness of \eqref{NLH}. %is, and we state the theorem below, where
% \label{SEC:LWP}
% Let $f$ be a locally Lipshitz continuous function, i.e. $\forall R>0$
%  there exists $C_{R}>0$ such that for all $u$, $v\in \R$ with
%  $|u|$, $|v| \leq R$,
% \begin{equation}\label{Lipshitz}
% |f(u)-f(v)| \leq C_{R}|u-v| .
% \end{equation}
% We consider the following nonlinear heat equation in $\R^n$,
% \begin{equation}\label{NLH1}
% \begin{array}{lll}
% \p_t u&=&\Delta u+f(u)\\
% u(x,0)&=&u_{0}(x).
% \end{array}
% \end{equation}
 $W^{s}:=\{u \in L^\infty, (-\Delta)^{s/2}u \in L^\infty\}$. %denotes a Sobolev space based on $L^\infty$.
 \begin{thm}\label{THM:Local}
 Let $u_0\in L^\infty$.Then there exists $t_*$ such that

 \begin{itemize}
 \item  %The nonlinear heat equation
 \eqref{NLH} %with the initial condition $u_0$
 has a unique mild solution in $
 C([0,t_*),L^\infty)$;

 \item $u$ depends continuously on the
 initial condition $u_0$;
 \item Either $t_*=\infty$ or $t_* < \infty$
 and $\|u(t)\|_{\infty}\rightarrow \infty$ as $t\rightarrow t_*$;

 \item If $u_0\in W^{s},\ s\geq 0$, then $\|\p_t u\|_{\infty} \lesssim t^{-\max(1-\frac{s}{2}, 0)}$ as $t\to 0$. In particular, $u\in C^1((0,t^*),L^\infty)$ ($C^1([0,t^*),L^\infty)$ if $s\geq 2$).
 \end{itemize}
 \end{thm}
\section{Blowup Variables and Almost Solutions} \label{Section:BV}
Let  $z(t)\in \R^{n},\ \lambda(t)>0,$ be differentiable functions and let $\alpha (t)$ satisfy the equation
\begin{equation}\label{eqn:evolutionalpha}
\lambda^{-2}\dot \alpha -a\alpha=-\lambda^{-1} \dot z,
\end{equation}
with $a(t)=\dot{\lambda}(t)/\lambda^3(t)$. We introduce the blowup variables
$$ y:=\lambda(t)(x-z(t))-\alpha (t)\ \mbox{and}\ \tau:=\int_{0}^{t}\lambda^{2}(s)ds$$ and define the new function
\begin{equation}\label{eq:definev}
  v(y,\tau):=\lambda^{-\frac{2}{p-1}}(t) u(x,t).
\end{equation}
Plugging \eqref{eq:definev} into \eqref{NLH} we obtain
\DETAILS{\begin{equation}\label{eqn:BVNLH'}
\p_\tau \tilde v=\left(\Delta_s-(as-\lambda^{-1}\dot z)\cdot\nabla_s-\frac{2a}{p-1}\right)\tilde v+|\tilde v|^{p-1}\tilde v.
\end{equation}
with the parameter
$a(\tau)=\dot{\lambda}(t(\tau))/\lambda^3(t(\tau))$.

Now observe that since $\frac{\p \tau}{\p t}=\lambda^2$, we have $\lambda^{-1}\dot z=\lambda z_\tau$.  Let $\alpha$ satisfy the equation $\alpha_\tau -a\alpha=-\lambda z_\tau$. Changing the variables in \eqref{eqn:BVNLH'} as $y=s-\alpha$ and
$\tilde v(s, \tau)= v(y, \tau)$, we conclude that $v$ satisfies the equation}
\begin{equation}\label{eqn:BVNLH}
\p_\tau v=\left(\Delta_y-a y\cdot\nabla_y-\frac{2a}{p-1}\right)v+|v|^{p-1}v,
\end{equation}
where, as above, $a(t)=\dot{\lambda}(t)/\lambda^3(t)$.  The initial
condition for this equation is obtained from the initial condition for
\eqref{NLH} as $ v(y,0)=\lambda_0^{-\frac{2}{p-1}}u_0(z_0+\frac{y+\alpha_0}{\lambda_0})$,
for some $\lambda_0, z_0$ and $\alpha_0$. %where $\lambda_0$ is the initial datum for the scaling parameter $\lambda$.

From the local well-posedness of \eqref{NLH} and using rescaling, we can conclude that there exists $T>0$ s.t.
\eqref{eqn:BVNLH} has a unique mild solution in $C([0, T), L^\infty)$ and the solution depends continuously
on the initial condition. Moreover, either $T=\infty$ or $T < \infty$ and $\|v(\tau)\|_{\infty} \rightarrow \infty$
as $\tau \rightarrow T$.

The equation \eqref{eqn:BVNLH} has the following family of homogeneous, static (i.e. $y$ and $\tau$-independent) solutions: $a$ is a constant and
\begin{equation}
v_a:=\lb \frac{2a}{p-1}\rb^\frac{1}{p-1}.
\end{equation}
This family of solutions corresponds to the homogeneous solution
\eqref{uhom} of the nonlinear heat equation with the parabolic
scaling $\lambda^{-2}=2 a(T-t)$, where the blowup time,
$T:=\left[u_0^{p-1}(p-1)\right]^{-1}$, is dependent on the initial
value, $u_0$ of the homogeneous solution $u_{\rm hom}(t)$.

If the parameter $a$ is $\tau$ dependent but $|a_\tau|$ is small,
then the above solutions are good approximations to the exact
solutions. A richer family of approximate solutions is obtained by
solving the equation $a y\cdot\nabla_yv+\frac{2 a}{p-1} v=v^p$,
obtained from \eqref{eqn:BVNLH} by neglecting the $\tau$ derivative
and second order partial derivative in $y$. This equation has the
general solution
\begin{equation}
v_{ab}(y):=\lb\frac{2a}{p-1 +  y b y}\rb^\frac{1}{p-1}
\end{equation}
for %$c=2a$ and
all $b:=(b_{ij}), b_{ij}\in\R$, real, symmetric $n\times n$-matrices. Here recall $y b y:=\sum_{i,j=1}^n y_i b_{ij}y_j$.
In what follows we take
$b\ge 0,$ so that $v_{ab}$ is nonsingular. Note
that $v_{2a, 0}=v_a$.

\section{Reparametrization of Solutions} \label{Section:Reparam}

In this section we split solutions to \eqref{eqn:BVNLH} %{eqn:w}
into the leading term---the almost solution
$V_{ a b }(y):=(\frac{a+1/2}{p-1+ y b y})^{\frac{1}{p-1}}$ %,with $c=a+\frac{1}{2}$,
---and a fluctuation $\xi$ around it. (The reason for passing from $v_{ab}(y):=\lb\frac{2a}{p-1 +  y b y}\rb^\frac{1}{p-1}$ to $V_{ a b }(y)$ was explained in the introduction.) More precisely, we would
like to parameterize a solution by a point on the manifold
$\mathcal{M}_{\rm as}:=\{V_{ab}\, |\, a\in \Rp, b\in \R^{n\times n}
\}$ of almost solutions and the fluctuation orthogonal
to this manifold (large slow moving and small fast moving parts of
the solution).  %Here $a=a(b,c)$ is a twice differentiable function of $b$ and $c$.
For technical reasons, it is more convenient to
require the fluctuation to be almost orthogonal to the manifold
$\M_{\rm as}$. More precisely, recalling the discussion at the end of the previous section, we require $\xi$ to be orthogonal to the
vectors $\phi^{(ij)}_{a},\, 0\leq i,
j\leq n$, where
% $$\phi^{(00)}_{ a}(y):=e^{-\frac{a}{4} |y|^2},\ %and $\phi^{(i)}_{2a}:=(1-a y_i^2)e^{-\frac{a}{4} |y|^2}$ Note that $\xi$ is already orthogonal to
% \phi^{(0i)}_{a}(y)=\phi^{(i0)}_{ a}(y):=\sqrt{a}
% y_i e^{-\frac{a}{4} |y|^2},\
% \phi^{(ij)}_{a}(y):=ay_iy_je^{-\frac{a}{4} |y|^2},\, 1\leq i,
% j\leq n,$$
$$\phi^{(00)}_{ a}(y):=1,\ %and $\phi^{(i)}_{2a}:=(1-a y_i^2)e^{-\frac{a}{4} |y|^2}$ Note that $\xi$ is already orthogonal to
\phi^{(0i)}_{a}(y)=\phi^{(i0)}_{ a}(y):=\sqrt{a}
y_i ,\
\phi^{(ij)}_{a}(y):=ay_iy_j,\, 1\leq i,
j\leq n,$$
which are almost tangent vectors to the above manifold, provided $b$ is sufficiently small.
Denote by $\bM_n$ the space of real, symmetric, $n\times n$ matrices
and by $\bM^+_n$, the positive cone in this space. Let $u_{\lambda, z}(y) :=
\lambda^{-\frac{2}{p-1}}u(x)$, with $x=z+\lambda^{-1}(y+\alpha)$. We define the neighborhoods
% \begin{equation*}
% U_{\epsilon}:=\{v\in L^\infty(\R^n)\ |\ \|e^{-\frac{1}{9}
% |y|^2}(v-V_{ab})\|_\infty=o(\|b\|)\ \mbox{for some}\ 1/4\le a\le 1,\
% 0< b\le \epsilon\ \}
% \end{equation*}
\begin{equation*}
U_{\epsilon}:=\{v\in L^\infty(\R^n)\ |\ \|e^{-\frac{1}{3}
|y|^2}(v-V_{ab})\|_\infty\leq C\norm{b}^2 \mbox{ for some}\ 1/4\le a\le 1,\
0< b\le \epsilon\ \}
\end{equation*}
and
\begin{equation*}
\tilde U_{\epsilon}:=\{u\in L^\infty(\R^n)\ |\ u_{\lambda, z}\in U_{\epsilon}\ \}.
\end{equation*}
The following statement will be used to reparametrize the initial conditions.
\begin{prop}\label{Prop:Splitting}
There exist an $\epsilon_{0}>0$ and a unique $C^1$ functional
$g:U_{\epsilon_0}\rightarrow \mathbb{R}^{+}\times
\bM^+_n\times \mathbb{R}^n$, such that any function $u_{\lambda, z_0}\in U_{\epsilon_0}$
can be uniquely written in the form
\begin{equation} \label{eqn:split}
u_{\lambda, z_0} =V_{ab} + \xi,
\end{equation}
with $\xi\perp  \phi_{a}^{(ij)},\ 0\leq i,
j\leq n,$ in $L^2(\R^n, \e^{-\frac{a|y|^{2}}{2}} dy)$, $(a, b, z)=g(u_{\lambda,z_0})$.
Moreover, if  $\frac{1}{4} \le a_0 \le 1, 0< b_0 \le \varepsilon_0$ %\in [\frac{1}{4},1]\times [0,\varepsilon_0]^n$
and
$\|\langle y\rangle^{-m}(u_{\lambda,z_0}-V_{a_0
b_0})\|_\infty\le \delta_{m}$ with $m=0,3$,
$\delta_{3}=O(\|b_{0}\|^2)$ and $\delta_{0}$ small, we have
\begin{equation}\label{eq:vv0}
|g_1(u_{\lambda,z_0})-(a_0, b_0)|\lesssim \|b_{0}\|^{2},
\end{equation}
\begin{equation}\label{Ineq:z_0}
|g_2(u_{\lambda,z_0})-z_0|\lesssim \|b_0\|,
\end{equation}
\begin{equation}
\|\langle y\rangle^{-3}(u_{\lambda,z_0}-V_{g(u_{\lambda,z_0})}))\|_\infty\lesssim \|b_{0}\|^{2},
\label{Ineq:IC}
\end{equation}
\begin{equation}\label{eq:vwithoutweight}
\|u_{\lambda,z_0}-V_{g(u_{\lambda,z_0})}\|_\infty\lesssim \delta_{0}+\|b_{0}\|.
\end{equation}
for $g(u_{\lambda, z_0})=(g_1(u_{\lambda, z_0}), g_2(u_{\lambda, z_0}))$, where
$g_1(u_{\lambda, z_0})=(a, b)$ and $g_2(u_{\lambda, z_0})=z$.
\end{prop}
\begin{proof} Let
$V_{\lambda a b z}(x):=\lambda^{\frac{2}{p-1}}V_{ a b }(y)$, $V_{\mu} \equiv V_{\lambda a b z}$
with $\mu=(a,b, z)$,
and $\varphi^{(ij)}_{az}(x):=\phi^{(ij)}_{a}(y)$ with $ y:=\lambda(x-z)-\alpha$.
The orthogonality conditions on the fluctuation can be written as
$G(\mu,u)=0$, where $G:\mathbb{R}^{+}\times
\bM^+_n\times \mathbb{R}^n\times
 L^\infty\lb \R^n\rb\rightarrow \bM_{n+1}$ is defined
as
\begin{equation*}
\DETAILS{G(\mu, v):=\lb \begin{array}{c} \ip{V_{\mu}-v}{e^{-\frac{a|y|^{2}}{4}}\phi_{0a}}\\
\ip{V_{\mu}-v}{e^{-\frac{a|y|^{2}}{4}}\phi^{(1)}_{2a}}\\
\cdots\\
\cdots\\
\cdots\\
\ip{V_{\mu}-v}{e^{-\frac{a|y|^{2}}{4}}\phi^{(n)}_{2a}}\\
\end{array} \rb.}
G(\mu, u):= \lb \ip{V_{\mu}-u}{\varphi^{(ij)}_{a z}}\rb .
\end{equation*}
Here and in what follows, all inner products are $L^2(\R^n, \e^{-a|y|^2/2}\d y)$ inner
products.  Whenever it is convenient we identify  $\mu$ with an $(n+1)\times(n+1)-$matrix:
$\mu_{00}:=a,\ \mu_{0i}=\mu_{i0}=z_{i},\ \mu_{ij}:=b_{ij},\, 1\leq i, j\leq n$ and let $\bM^{++}_{n+1}:=\{\mu\in \bM_{n+1}\ | a\ge 0,\  b \ge 0,\ z \in \R^n\}$ and $\bM_{n+1,\epsilon}:=\{\mu\in \bM_{n+1}\ | a\in [\frac{1}{4},1],\ 0 < b \le \epsilon,\ z \in \R^n\}$.

Let $X:= e^{\frac{1}{3} |y|^2} L^\infty(\R^n)$ with the
corresponding norm. %and $\bM_{n\epsilon}:=\{b\in \bM^+_{n}\ | b \le \epsilon\}$.
Using the implicit function theorem we will
prove that for any $\mu_0:=(a_0, b_0, z_0)\in %[\frac{1}{4},1]\times
\bM^{++}_{n+1,\epsilon_{0}}$ there exists a unique $C^1$ function
$\tilde g:X\rightarrow \mathbb \bM_{n+1} %{R}^{+}\times\bM^+_{n}\times \mathbb{R}^n
$, defined in
a neighborhood $\tilde U_{\mu_0}\subset X$ of $V_{\mu_0}$, such that
$G(\tilde g(u),u)=0$ for all $u\in \tilde U_{\mu_0}$. Let
$B_{\varepsilon}(V_{\mu_0})$ and $B_\delta(\mu_0)$ be the balls in
$X$ and $\mathbb{R}^{n+1}$ around $V_{\mu_0}$ and $\mu_0$ and of the
radii $\varepsilon$ and $\delta$, respectively.

Note first that the mapping $G$ is $C^1$ and $G(\mu_0, V_{\mu_0})=0$
for all $\mu_0$.  We claim that the linear map $\p_\mu G(\mu_0,
V_{\mu_0})$ is invertible.

\begin{lemma}\label{partialG}
$\exists\eps_0>0$ such that $\p_\mu G(\mu,u)$, for $u\in \tilde U_{\epsilon_0}$, is invertible.
\end{lemma}

\begin{proof}
Let the indices $\alpha$ and $\beta$ run over the pairs $(i, j),\ 0 \le i \le j \le n$. We compute
\begin{equation}\label{linearization}
\p_\mu G(\mu,u)= A_{1}+A_{2}
\end{equation}
where the $(\alpha,\beta)-$th entries of $A_1$ and  $A_2$ are
\begin{equation}\label{eqn:A_1}
A_1(\alpha,\beta)=\ip{\p_{\mu_\alpha} V_{\mu}}{ \varphi_{az}^{(\beta)}}
\end{equation}
and
$$A_2(\alpha,\beta)=\ip{ V_{ \mu}-u}{\p_{\mu_\alpha} \varphi_{az}^{(\beta)}},$$ respectively.
We write $A_1$ in the block form
$$A_1=\left(
  \begin{array}{ccc}
    K_{11} & K_{12} & K_{13} \\
    K_{21} & K_{22} & K_{23} \\
    K_{31} & K_{32} & K_{33} \\
  \end{array}
\right),$$
where $K_{11}=\ip{\p_{(a, b^{\text{diag}})}V_{\mu}}{\varphi_{az}^{ii}}$, with $0\leq i\leq n$, $K_{22}=\ip{\p_{b^{\text{off-diag}}}V_{\mu}}{\varphi_{az}^{ij}}$, with $1\leq i < j \leq n$,
$K_{33}=\ip{\p_z V_{\mu}}{\varphi_{az}^{0i}}$, with $1\leq i \leq n$ and similarly for the other entries.
For $b>0$ and small, we %expand the matrix $A_1$ in $b$ and
compute using change of variable $y=\lambda(x-z)-\alpha$, that (see Appendix 2 for more details)
\begin{equation}\label{K_11}
K_{11}=\frac{\lambda^{-n+\frac{2}{p-1}}}{p-1}(\frac{a+\frac{1}{2}}{p-1})^{\frac{1}{p-1}}(\frac{2\pi}{a})^{\frac{n}{2}}\left(
\begin{array}{lllll}
\frac{1}{a+\frac{1}{2}} & \frac{1}{a+\frac{1}{2}} & \frac{1}{a+\frac{1}{2}}& \cdots & \frac{1}{a+\frac{1}{2}}\\
-\frac{1}{(p-1)a} & -\frac{3}{(p-1)a} & -\frac{1}{(p-1)a}&\cdots &-\frac{1}{(p-1)a}\\
-\frac{1}{(p-1)a} & -\frac{1}{(p-1)a}& -\frac{3}{(p-1)a} &\cdots&-\frac{1}{(p-1)a}\\
\cdots&\cdots&\cdots&\ddots&\vdots\\
-\frac{1}{(p-1)a}&-\frac{1}{(p-1)a}&\cdots&-\frac{1}{(p-1)a}& -\frac{3}{(p-1)a}
\end{array}
\right)+O(\|b\|)
\end{equation}
is an $(n+1)\times(n+1)$ matrix,
\begin{equation}\label{K_22}
K_{22}=-\lambda^{-n+\frac{2}{p-1}}(\frac{a+1/2}{p-1})^{\frac{1}{p-1}}
\frac{2}{(p-1)^2a}(\frac{2\pi}{a})^{n/2}I_{\frac{n(n-1)}{2}\times\frac{n(n-1)}{2}}+O(\|b\|)
\end{equation}
and
\begin{equation}\label{K_33}
K_{33}=-\lambda^{-n+\frac{2}{p-1}}(\frac{a+1/2}{p-1})^{\frac{1}{p-1}}
(\frac{2\pi}{a})^{n/2}b+o(\|b\|)
\end{equation}
 is an $n\times n$ matrix.
Moreover,
\begin{equation}\label{K_ij}
K_{ij}=o(\|b\|)\ \text{for}\ 1\leq i\neq j \leq 3.
\end{equation}
Since $K_{11}$, $K_{22}$ and $K_{33}$ are invertible, the matrix $A_{1}$ is also invertible. Furthermore, by the Schwarz
inequality
\begin{equation}\label{eqn:A_2}
\|A_{2}\|\lesssim \| u-V_{a_{0}b_{0}}\|_X=O(\|b\|^2).
\end{equation}
Therefore there exist $\varepsilon_0$ and $\varepsilon_1$ such that the
matrix $\p_\mu G(\mu,u)$ has an inverse for
$\mu\in \bM_{n+1,\eps_0} %[\frac{1}{4},1]\times[0,\varepsilon_0]^n
$ and  $u\in%\bigcup_{\mu\in \bM_{n+1,\epsilon}} %[\frac{1}{4},1]\times[0,\varepsilon_0]^n}
B_{\varepsilon_1}(V_{\mu})$.
\end{proof}

\DETAILS{
Moreover, from \eqref{linearization}-\eqref{eqn:A_2} we know that $\p_\mu G$ can be written as
$$\p_\mu G=\left(
             \begin{array}{cc}
               A_{11} & 0 \\
               0 & A_{22} \\
             \end{array}
           \right)
           +R,
$$
where $A_{11}=O(1)$ and has an $O(1)$ inverse, $A_{22}=O(\|b\|)$ and has an $O(\|b\|^{-1})$ inverse, and $R=o(\|b\|)$.
Using the fact that $(A+B)^{-1}=[A(I+A^{-1}B)]^{-1}=(I+A^{-1}B)^{-1}A^{-1}=\sum_{k=0}^{\infty}(-A^{-1}B)^kA^{-1}$
provided that $A$ is invertible and $\|A^{-1}B\|<1$, we conclude
\begin{equation}\label{est:partialG}
(\p_\mu G)^{-1}=\left(
             \begin{array}{cc}
               B_{11} & B_{12} \\
               B_{21} & B_{22} \\
             \end{array}
           \right),
\end{equation}
where $B_{11}=O(1)$, $B_{22}=O(\|b\|^{-1})$ and $B_{12}=B_{21}=o(1)$.
}

Moreover, from \eqref{linearization}-\eqref{eqn:A_2} we know that $\p_\mu G$ can be written as
$$\p_\mu G=\left(
             \begin{array}{cc}
               A_{11} & A_{12} \\
               A_{21} & A_{22} \\
             \end{array}
           \right)
           +R,
$$
where $A_{11}=O(1)$ and has an $O(1)$ inverse, $A_{22}=O(\|b\|)$ and has an $O(\|b\|^{-1})$ inverse,
$A_{12}=o(\|b\|)$ and $A_{21}=o(\|b\|)$. Then we have
\begin{equation}\label{est:partialG}
(\p_\mu G)^{-1}=\left(
             \begin{array}{cc}
               B_{11} & B_{12} \\
               B_{21} & B_{22} \\
             \end{array}
           \right),
\end{equation}
where $B_{11}=(A_{11}-A_{12}A_{22}^{-1}A_{21})^{-1}=O(1)$, $B_{22}=(A_{22}-A_{21}A_{11}^{-1}A_{12})^{-1}=O(\|b\|^{-1})$,
$B_{12}=-A_{11}^{-1}A_{12}(A_{22}-A_{21}A_{11}^{-1}A_{12})^{-1}=o(1)$ and $B_{21}=-A_{22}^{-1}A_{21}(A_{11}-A_{12}A_{22}^{-1}A_{21})^{-1}=o(1)$.

Hence by the implicit function theorem,
the equation $G(\mu,u)=0$ has a unique solution $\mu=\tilde g(u)$ on a
neighborhood of every $V_\mu$,
$\mu\in \bM_{n+1,\epsilon} %[\frac{1}{4},1]\times[0,\varepsilon_0]^n
$, which is $C^1$ in $u$.  Our next goal is to determine these neighborhoods.

To determine a domain of the function $\mu=\tilde g(u)$, we examine closely
a proof of the implicit function theorem. Proceeding in a standard
way, we expand the function $G(\mu,u)$ in $\mu$ around $\mu_0$:
\begin{equation*}
G(\mu,u)=G(\mu_0,u)+\p_\mu G(\mu_0,u)(\mu-\mu_0)+R(\mu,u),
\end{equation*}
where $R(\mu,u)=\O{|\mu-\mu_0|^2}$ uniformly in $u\in X$. Here
$|\mu|^2=|a|^2+\|b\|^2+|z|^2$ for $\mu=(a,b, z)$.  Inserting this into the
equation $G(\mu,u)=0$ and inverting the matrix $\p_\mu G(\mu_0,u)$,
we arrive at the fixed point problem $\alpha=\Phi_u(\alpha)$, where
$\alpha:=\mu-\mu_0$ and $\Phi_u(\alpha):=- \p_\mu G(\mu_0,u)^{-1}
[G(\mu_0,u)+R(\mu,u)]$.
By the above estimates there exists an
$\varepsilon_1$ such that the matrix $\p_\mu G(\mu_0,u)^{-1}$ is
bounded in $u\in B_{\varepsilon_1}(V_{\mu_0})$.
Define
$$|\mu|_{b_0}=|a|+\|b\|+\|b_0\||z|$$
for $\mu=(a,b,z)$, then from \eqref{est:partialG}
we have $|(\p_{\mu}G)^{-1}\mu|_{b_0} \lesssim |\mu|$. It follows that
\begin{equation}\label{eqn:SplittingSharp}
\begin{array}{ll}
&|\Phi_{u}(\alpha)|_{b_0}\lesssim|G(\mu_0,u)|+|\alpha|^2.
\end{array}
\end{equation}
%Hence we obtain from \eqref{est:partialG} and the remainder estimate above that
%\begin{equation}\label{eqn:SplittingSharp}
%\begin{array}{ll}
%&|\Phi_{u1}(\alpha)|+\|b_0\||\Phi_{u2}(\alpha)|\lesssim|G(\mu_0,u)|+|\alpha|^2.
%\end{array}
%\end{equation}
%where $\Phi_u(\alpha)=(\Phi_{u1}(\alpha), \Phi_{u2}(\alpha))$,
%with $\Phi_{u1}(\alpha)$ being the first two entries and $\Phi_{u2}(\alpha)$ being the last entry.
Furthermore, using that $\p_\alpha \Phi_u(\alpha)= - \p_\mu
G(\mu_0,u)^{-1} [G(\mu,u)- G(\mu_0,u)+R(\mu,u)]$, we obtain that
there exist $\varepsilon \leq \varepsilon_1$ and $\delta$ such that
$\|\p_\alpha\Phi_u(\alpha)\|\le\frac{1}{2}$ for all $u\in
B_{\varepsilon}(V_{\mu_0})$ and $\alpha\in B_\delta(0)$.  Pick
$\varepsilon$ and $\delta$ so that $\varepsilon
%_0:=\min\{\varepsilon_1,\varepsilon_2\}
\ll\delta\ll \|b_0\|\ll 1$.  Then, for all $u\in
B_{\varepsilon}(V_{\mu_0})$, $\Phi_u$ is a contraction on the ball
$B_\delta(0)$ and consequently has a unique fixed point in this
ball.  This gives a $C^1$ function $\mu=\tilde g(u)$ on
$B_{\varepsilon}(V_{\mu_0})$ satisfying $|\mu-\mu_0|\le\delta$. An
important point here is that since $\varepsilon\ll \|b_0\|$ we have
that $b>0$ for all $V_{ab}\in B_{\varepsilon}(V_{\mu_0})$.  Now,
clearly, the balls $B_{\varepsilon}(V_{\mu_0})$ with
$\mu_0\in \bM_{n+1,\varepsilon_0} %[\frac{1}{4},1]\times[0,\varepsilon_0]^n
$ cover the
neighbourhood $\tilde U_{\varepsilon_0}$.  Hence, the map $\tilde g$ is defined on
$\tilde U_{\varepsilon_0}$ and is unique, and the same is true for the map $g$,
defined as $g(u_{\lambda, z_0})=\tilde g(u)$, which implies the first part of
the proposition.

Now we prove the second part of the proposition.  The definition of
the function $G(\mu,u)$ implies $G(\mu_0,u)=\lambda^{-n+\frac{2}{p-1}}\lb\ip{V_{a_0b_0}-u_{\lambda,z_0}}{\phi_a^{ij}(y)} \rb$, therefore
\begin{equation}
|G(\mu_0,u)|\lesssim \|e^{-\frac{1}{3} y^2}(u_{\lambda,z_0}-V_{a_0b_0})\|_\infty.
\label{eqn:28aA}
\end{equation}
This inequality together with the estimate
\eqref{eqn:SplittingSharp} and the fixed point equation
$\alpha=\Phi_u(\alpha)$, where $\alpha=\mu-\mu_0$ and $\mu=g(u_{\lambda, z_0})$,
implies
\begin{equation}\label{eqn:28a}
|g(u_{\lambda,z_0})-\mu_0|_{b_0}\lesssim \|e^{-\frac{1}{3} y^2}(u_{\lambda,z_0}-V_{a_0b_0})\|_\infty.
\end{equation}
%\begin{equation}\label{eqn:28a}
%|g_1(u_{\lambda,z_0})-(a_0,b_0)|+\|b_0\||g_2(u_{\lambda,z_0})-z_0|\lesssim \|e^{-\frac{1}{9} y^2}(u_{\lambda,z_0}-V_{a_0b_0})\|_\infty.
%\end{equation}
From one of the conditions of the
proposition, r.h.s. of \eqref{eqn:28a} $=O(\|b_0\|^2)$ if $a_0\in[\frac{1}{4},1]$. The
last estimate implies \eqref{eq:vv0} and \eqref{Ineq:z_0}.  Using Equation
\eqref{eqn:28a} we obtain
$$
\begin{array}{lll}
\|\langle y\rangle^{-3}(u_{\lambda,z_0}-V_{g(u_{\lambda,z_0})})\|_\infty&\leq &\|\langle
y\rangle^{-3}(u_{\lambda,z_0}-V_{\mu_{0}})\|_\infty+\|\langle
y\rangle^{-3}(V_{g(u_{\lambda,z_0})}-V_{\mu_{0}})\|_\infty\\
&\lesssim& \|\langle
y\rangle^{-3}(u_{\lambda,z_0}-V_{\mu_{0}})\|_\infty+|g(u_{\lambda,z_0})-\mu_{0}|\\
&\lesssim & \|\langle y\rangle^{-3}(u_{\lambda,z_0}-V_{\mu_{0}})\|_\infty,
\end{array}
$$ which leads to \eqref{Ineq:IC}.
Finally, to prove Equation \eqref{eq:vwithoutweight},
%is implied by (\ref{eq:approx}).
we write
$$\|u_{\lambda,z_0}-V_{g(u_{\lambda,z_0})}\|_\infty\leq \|u_{\lambda,z_0}-V_{a_{0},b_{0}}\|_\infty+\|V_{g(u_{\lambda,z_0})}-V_{a_{0},b_{0}}\|_\infty.$$
A straightforward computation gives $\|V_{a b}-V_{a_{0}
b_{0}}\|_\infty\lesssim |a-a_{0}|+\frac{\|b-b_{0}\|}{\|b_{0}\|}$.  Since
by \eqref{eq:vv0}, $|a-a_{0}|+\|b-b_{0}\|=O(\|b_{0}\|^2)$, we have
$\|V_{a b}-V_{a_0 b_0}\|_\infty\lesssim \|b_{0}\|$. This together with
the fact $\|u_{\lambda,z_0}-V_{a_{0},b_{0}}\|_\infty\leq \delta_{0}$ completes the
proof of \eqref{eq:vwithoutweight}.
\end{proof}

Now we establish a reparametrization of the solution $u(x,t)$ on small
time intervals. In Section \ref{SecMain} we convert this result to a
global reparametrization. In the rest of the section it is
convenient to work with the original time $t$,  instead of rescaled
time $\tau$. We let $I_{t_0, \delta}:= [t_0, t_0 +\delta]$ and
define for any time $t_{0}$ and constant $\delta>0$ three sets:
$$\mathcal{A}_{t_0, \delta}:=
C^{1}(I_{t_0, \delta},[1/4,1]),\ \mathcal{B}_{t_0,
\delta, \epsilon_0}:=C^{1}(I_{t_0, \delta},\bM^+_{n,\epsilon_0}) \ \mbox{and}\
\mathcal{C}_{t_0,
\delta}:=C^{1}(I_{t_0, \delta},[-1,1]^n),$$
where we recall the constant $\epsilon_{0}$ from Proposition
\ref{Prop:Splitting}.

Recall $u_{\lambda,z}(y, t) :=
\lambda(t)^{-\frac{2}{p-1}}u(x,t)$, with $x=z(t)+\lambda^{-1}(t)(y+\alpha(t))$. Suppose
$u(\cdot,t)$ is a function such that for some $\lambda_{0}>0$
\begin{equation}\label{eq:init2}
\sup_{t \in I_{t_0, \delta}}\|b^{-1}(t)\|\|\langle y\rangle^{-3}
%e^{-\frac{1}{9}\lambda_{0}^{2}x^2/\lambda^{2}(a,t)}
(u_{\lambda,z}(\cdot,t)-V_{a(t),b(t)})\|_{\infty}\ll 1
\end{equation}
for some  $a\in\mathcal{A}_{t_0, \delta}$, $b\in\mathcal{B}_{t_0,
\delta, \epsilon_0}$, $z\in\mathcal{C}_{t_0,\delta}$, $\lambda(t)$ satisfying
$\lambda(t_{0})=\lambda_{0}\ \mbox{and}\
\lambda^{-3}(t){\partial_t}\lambda(t)=a(t)$ and $\alpha(t)$ satisfying $\alpha(t_0)=\alpha_0$ and $\p_t\alpha(t)-\lambda^2(t)a(t)\alpha(t)+\lambda(t)\p_t z(t)=0$. We define the set
$$\mathcal{U}_{t_0,\delta, \eps_0, \lambda_0, \alpha_0}:=\{u \in C^1(\overset{\circ}{I}_{t_0, \delta},
\langle y\rangle^3L^\infty(\R^n))\ |\ \eqref{eq:init2}\ \text{holds
for some}\ a\in\mathcal{A}_{t_0,\delta},\ b\in\mathcal{B}_{t_0,\delta,\eps_0} \ \text{and} \ z\in\mathcal{C}_{t_0,\delta}\}.$$

\begin{prop}\label{Prop:Splitting2V}
Suppose $u\in \mathcal{U}_{t_0,\delta, \epsilon_0, \lambda_0, \alpha_0}$ and
$\lambda_{0}^{2}\delta \ll 1$. Then there exists a unique $C^1$ map
$g_\#:\mathcal{U}_{t_0,\delta, \epsilon_0, \lambda_0, \alpha_0}\rightarrow
\mathcal{A}_{t_0, \delta}\times \mathcal{B}_{t_0, \delta,
\epsilon_0}\times \mathcal{C}_{t_0,\delta}$, such that for $t\in I_{t_0, \delta},$ $u(\cdot,t)$ can
be uniquely represented in the form
\begin{equation} \label{eqn:splitting2}
u_{\lambda}(y, t) =V_{g_\#(u)(t)}(y) + \xi(y,t),
\end{equation}
with $(a(t), b(t), z(t))=g_\#(u)(t)$ and
\begin{equation} \label{eqn:splitting2conditions}
\begin{array}{ll}
&\xi(\cdot,t)\perp \phi_{a(t)}^{(ij)}\ \mbox{in}\
L^2(\R^n,e^{-\frac{a(t)}{2} |y|^2}dy),\\
& \lambda^{-3}(t){\partial_t}\lambda(t)=a(t)\
\mbox{and}\
\lambda(t_{0})=\lambda_{0},\\
&\p_t\alpha(t)-\lambda^2(t)a(t)\alpha(t)+\lambda(t)\p_t z(t)=0\ \text{and}\ \alpha(t_0)=\alpha_0.
\end{array}
\end{equation}
%Moreover
%\begin{equation}\label{eq:approx2}
%\max_{t_{0}\leq t\leq t_{0}+\delta}|g(u)-(a,b)|\lesssim
%\max_{t_{0}\leq t\leq t_{0}+\delta}\|e^{-\frac{1}{9}
%\lambda_{0}^{2}x^2}(\lambda^{-\frac{2}{p-1}}(a,t)u(x,t)-V_{a(t),b(t)}(\lambda(a,t)x))\|_{L^{\infty}}.
%\end{equation}
\end{prop}
\begin{proof}
%Recall the definition $X:= \langle y\rangle^3 L^\infty(\R^n)$
%with the corresponding norm.
For any function $a\in
\mathcal{A}_{t_0, \delta},$ we define a function
$$\lambda(a,t):=(\lambda_{0}^{-2}-2\int_{t_0}^{t}a(s)ds)^{-\frac{1}{2}}.$$ Let $\lambda(a)(t):=\lambda(a,t)$.
Next we define a function
$$\alpha(a,z)(t):=e^{\int_{t_0}^t\lambda^2(s)a(s)ds}\alpha_0-\int_{t_0}^t e^{\int_s^t \lambda^2(\gamma)a(\gamma) d\gamma}
\lambda(s)\p_t z(s) ds.$$
Define the $C^1$ map $G_\#$:
\begin{align*}
C^{1}(I_{t_0,
\delta},\R^+)\ \times
C^{1}(I_{t_0, \delta},\bM_n^+)\ \times C^{1}(I_{t_0,
\delta},\R^n)\ \times C^1(I_{t_0, \delta},\lra{y}^3 L^\infty(\R^n))\rightarrow C^{1}(I_{t_0, \delta},\mathbb{R}^{\frac{(n+2)(n+1)}{2}})
\end{align*}
as $$G_\#(\mu,u)(
t):=G(\mu(t),u_{\lambda (a),z}(\cdot,t)),$$ where $t\in I_{t_0,
\delta}$, $\mu=(a, b, z)$ and $G(\mu,u)$ is the same as in the proof of
Proposition \ref{Prop:Splitting}. The orthogonality conditions on
the fluctuation can be written as $G_\#(\mu,u)=0$. Using the
implicit function theorem we will first prove that for any
$\mu_0:=(a_0, b_0, z_0)\in \mathcal{A}_{t_0, \delta}\times
\mathcal{B}_{t_0, \delta, \epsilon_0}\times \mathcal{C}_{t_0, \delta}$ there exists a neighborhood
$\mathcal{U}_{\mu_0}$ of $V_{\mu_0}$ and a unique $C^1$ map
$g_\#:\mathcal{U}_{\mu_0}\rightarrow \mathcal{A}_{t_0, \delta}\times
\mathcal{B}_{t_0, \delta, \epsilon_0}\times \mathcal{C}_{t_0, \delta}$ such that $G_\#(g_\#(v),v)=0$
for all $v\in \mathcal{U}_{\mu_0}$.

We claim that $\p_\mu G_\#(\mu,u)$ is invertible, provided
$u_{\lambda (a),z}$ is close to $V_\mu$. We compute
\begin{equation}\label{linearization1}
\p_\mu G_\#(\mu,u)( t)=\p_\mu G(\mu(t), u_{\lambda (a),z}(\cdot,t))=A(
t) + B( t),
\end{equation}
where
\begin{equation} \label{linearization2}
A( t):= \p_{\mu}G(\mu, v)|_{v= u_{\lambda(a), z}},\ B(t):=
\p_vG(\mu, v)|_{v= u_{\lambda(a), z}}\p_\mu u_{\lambda (a), z}.
\end{equation}
Note that in \eqref{linearization2} $\p_v G(\mu, v)|_{v= u_{\lambda
(a), z}}$ is acting on $\p_\mu u_{\lambda (a), z}$ as an integral with respect to
$y$ and let $B(t)(y)$ be the integral kernel of this operator. We have shown in Lemma \ref{partialG}
that the first term on the r.h.s. is invertible, provided
$u_{\lambda (a), z}$ is close to $V_\mu$.

Now we show that for $\delta>0$ sufficiently small the second term
on the r.h.s. is small. Let $v:= u_{\lambda (a), z}$. Assuming for the
moment that $v$ is differentiable, we compute
$\p_av=-\p_a(\lambda^{-1})[\frac{2}{p-1}\lambda v-(y+\alpha)\nabla_yv]+
\lambda^{-1}\p_a\alpha \nabla_y v.$
Combining the last two equations together with Equation \eqref{linearization2} we obtain
\begin{equation*}
[B(t)\rho](t)=\int B(t)(y) [(-\frac{2}{p-1}\lambda v+(y+\alpha)\nabla_y v)(\p_a \lambda^{-1}) \rho+
\lambda^{-1}\nabla_y v(\p_a\alpha)\rho] dy.
\end{equation*}
Integrating by parts in the second term in the parenthesis gives
\begin{equation}\label{eqn:B(t)[y]}
[B( t) \rho](t)=-\int [(\frac{2}{p-1}\lambda v+v\nabla_y \cdot (y+\alpha))(\p_a \lambda^{-1})\rho
+\lambda^{-1}v\nabla_y \cdot (\p_a \alpha)\rho]B(t)[y] dy.
\end{equation}
Furthermore,
$\p_a(\lambda^{-1})\rho=\lambda(t)\int_{t_0}^{t}\rho(s)ds$
and
\begin{align*}
(\p_a\alpha) \rho& =e^{\int_{t_0}^t\lambda^2(s)a(s)ds}\alpha_0\int_{t_0}^t [a(s)\p_a \lambda^2(s)+\lambda^2(s)]\rho(s) ds\\
& -\int_{t_0}^t e^{\int_s^t \lambda^2(\gamma)a(\gamma) d\gamma}\p_t z(s)
[\lambda(s)\int_s^t (a(\gamma)\p_a \lambda^2(\gamma)+\lambda^2(\gamma))\rho(\gamma)d\gamma + \p_a\lambda(s)\rho(s)]ds.
\end{align*}
Now, using a density argument, we remove the
assumption  of the differentiability on $v$ and conclude that \eqref{eqn:B(t)[y]} holds without this assumption. Using this expression and
the inequality $\lambda(t) \le \sqrt{2}\lambda_0$, provided $\delta
\le (4 \sup a)^{-1}\lambda_0^{-2} \le 1/4 \lambda_0^{-2}$, we
estimate
\begin{equation}
\|B( t) \rho\|_{L^\infty([t_0,t_0+\delta])}\lesssim \delta
\lambda_0^{2}\|v
\|_{L^\infty}\|\rho\|_{L^\infty([t_0,t_0+\delta])}.
\end{equation}
So $B( t)$ is small, if $\delta \lesssim (\lambda_0^{2}\|v
\|_{L^\infty})^{-1}$, as claimed. This shows that $\p_\mu
G_\#(\mu,u)$ is invertible, provided $u_{\lambda (a),z}$ is close to
$V_\mu$. Proceeding as in the proof of Proposition
\ref{Prop:Splitting} we conclude
%that the implicit function theorem applies to the equation
%$G_\#(\mu,u)=0$ which gives a unique solution $g_\#(u)$ satisfying
%$G_\#(g_\#(u),u)=0$, which implies
the proof of Proposition \ref{Prop:Splitting2V}.
\end{proof}

%%%%%%%%%%%%%%%%%%%%%%%%%%%%%%%%%%%%%%%%%%%%%%%%%%%%%%%%%%%%%%%%%%
%%%%%%%%%%%%%%%%%%%%%%%%%%%%%%%%%%%%%%%%%%%%%%%%%%%%%%%%%%%%%%%%%%%%%
\section{A priori Estimates}\label{Section:APriori}
Let $u(x,t),$ $0\leq t\leq T$ be a solution to \eqref{NLH} with initial
condition $u_0\in U_{\epsilon_0}$ and
$v(y,\tau)=\lambda^{-\frac{2}{p-1}}(t)u(x,t)$, where $y=\lambda (x-z)-\alpha$
and $\tau(t):= \int_{0}^{t}\lambda^{2}(s)ds$. We assume that
there exist $C^{1}$ functions
$a(\tau)$, $b(\tau)$, and $c(\tau)$ such that
$v(y,\tau)$ can be represented as
\begin{align}\label{eqn:split2}
v(y,\tau)=V_{a(\tau)b(\tau)}+\xi(y,\tau),
\end{align}
where, recall,
\begin{align*}
V_{ab}:=\lb\frac{a+\frac{1}{2}}{p-1+y by}\rb^{\frac{1}{p-1}}
\end{align*}

where $\xi(\cdot,\tau)\perp \phi_{a(\tau)}^{(ij)}$
(see (\ref{eqn:split})),
$\lambda^{-3}(t)\partial_{t}\lambda(t)=a(\tau(t))$. Since $u_0\in U_{\epsilon_0}$, by condition \eqref{INI2}
\begin{equation}\label{eqn:splitB}
\|\lra{y}^{-3}\xi(y,0)\|_\infty\lesssim
\|b(0)\|^2.
\end{equation} In this
section we formulate a priori bounds on the fluctuation $\xi$ which
are proved in later sections.

Let the function $\tilde{\beta}(\tau)$ and the constant
$\kappa$ be defined as
\begin{equation}\label{FunBTau}
\tilde{\beta}(\tau):=(b(0)^{-1}+\frac{4p\tau}{(p-1)^2}I)^{-1}\
\mbox{and}\ \kappa:=\min\{\frac{1}{2},\frac{p-1}{2}\},
\end{equation}
and let $\beta(\tau)$ be the largest eigenvalue of $\tilde{\beta}(\tau)$.
%As we will see later $\beta_i$ is a good approximate of the function $b_i$.
For the functions $\xi(\tau),$ $b(\tau)$ and $a(\tau)$ we
introduce the following estimating functions (families of
semi-norms)
\begin{equation}
\label{majorants}
\begin{array}{lll}
M_{1}(T)&:=\max_{\tau\leq T}
\beta^{-2}(\tau)\|\lra{y}^{-3}\xi(\tau)\|_{\infty},\\
M_{2}(T)&:=\max_{\tau\leq
T}\|\xi(\tau)\|_{\infty},\\
A(T)&:=\max_{\tau\leq
T}\beta^{-2}(\tau)\left|a(\tau)-\frac{1}{2}+\frac{2\Tr b(\tau)}{p-1}\right|,\\
B(T)&:=\max_{\tau\leq
T}\beta^{-(1+\kappa)}(\tau)\|b(\tau)-\tilde{\beta}(\tau)\|.
\end{array}
\end{equation}

\begin{prop}\label{Prop:aprior}
Let $\xi$ be defined in \eqref{eqn:split2} and assume
$M_{1}(0),A(0), B(0) \lesssim 1$, $M_{2}(0)\ll 1$.
Assume there exists an interval $[0,T]$ such that for $\tau\in
[0,T]$
$$
M(\tau),\ A(\tau),\ B(\tau)\leq \beta^{-\kappa/2}(\tau).$$ Then in
the same time interval the parameters $a$, $b$ and the function
$\xi$ satisfy the following estimates
\begin{equation}\label{EstB}
|\frac{\p}{\p\tau}b(\tau)+\frac{4p}{(p-1)^{2}}b^{2}(\tau)|
\lesssim
\beta^3(\tau)+\beta^{3}(\tau)M_{1}(\tau)(1+A(\tau))+\beta^{4}(\tau)M_{1}^{2}(\tau)
+\beta^{2p}M_{1}^{2p}(\tau),
\end{equation} and
\begin{equation}\label{MajorE}
B(\tau)\lesssim
1+M_{1}(\tau)(1+A(\tau))+M_{1}^{2}(\tau)+M_{1}^{p}(\tau),
\end{equation}
\begin{equation}\label{EstA}
A(\tau)\lesssim
A(0)+1+\beta(0)M_{1}(\tau)(1+A(\tau))+\beta(0)M_{1}^{2}(\tau)+\beta^{2p-2}(0)M_{1}^{p}(\tau),
\end{equation}
\begin{equation}\label{M1}
\begin{array}{lll}
M_{1}(\tau)
&\lesssim &  M_{1}(0)+\beta^{\frac{\kappa}{2}}(0)[1+M_{1}(\tau)A(\tau)+M_{1}^{2}(\tau)+M_{1}^{p}(\tau)]\\
& &+[M_{2}(\tau)M_{1}(\tau)+M_{1}(\tau)M_{2}^{p-1}(\tau)],
\end{array}
\end{equation}
\begin{equation}\label{M2}
\begin{array}{lll}
M_{2}(\tau)&\lesssim& M_{2}(0)+\beta^{1/2}(0)M_{1}(0)
+\beta^{\frac{1}{3}}(0)M_1^{\frac{2}{3}}(T)M_2^{\frac{1}{3}}(T)+M_{2}^{2}(\tau)+M_{2}^{p}(\tau)\\
& &+\beta^{\frac{\kappa}{2}}(0)[1+M_{2}(\tau)+M_{1}(\tau)A(\tau)+M_{1}^{2}(\tau)+M_{1}^{p}(\tau)].
\end{array}
\end{equation}
\end{prop}

Equations \eqref{EstB}-\eqref{EstA}, \eqref{M1} and \eqref{M2} will be proved
in Sections \ref{SEC:EstB}, \ref{SEC:EstM1} and \ref{SEC:EstM2} respectively.

\begin{cor}\label{cor:aprior}
Let $\xi$ be defined in \eqref{eqn:split2} and assume
$M_{1}(0),A(0), B(0)\lesssim 1$, $M_{2}(0)\ll 1$. Assume there
exists an interval $[0,T]$ such that for $\tau\in [0,T]$,
$$
M_{1}(\tau), A(\tau),\ B(\tau)\leq \beta^{-\kappa/2}(0).$$ Then in
the same time interval the parameters $a$, $b$ and the function
$\xi$ satisfy the following estimates
\begin{equation}\label{EstABM}
M_{1}(\tau),\ A(\tau),\ B(\tau)\lesssim 1, \ M_{2}(\tau)\ll 1.
%c\ \text{and}\ M_{2}(\tau)\ll 1,\
\end{equation}
(In fact, $M_{i}(\tau) \lesssim M_{i}(0)  +
\beta^{\frac{\kappa}{2}}(0),\ i=1,2$.)
\end{cor}
\begin{proof} Since $\beta(\tau) \le \beta(0)\ll 1$,
we have
\begin{equation}\label{ApriorEST}
M_{1}(\tau), \ B(\tau),\ A(\tau)\leq
\beta^{-\frac{\kappa}{2}}(0)\leq \beta^{-\frac{\kappa}{2}}(\tau),
\end{equation}
where, recall, the definitions of $\beta(\tau)$ and $\kappa$ are
given in (\ref{FunBTau}). Thus the conditions of the proposition above are satisfied. Since
$M_{1}(\tau)\leq \beta^{-\frac{\kappa}{2}}(0)$, we can solve
\eqref{EstA} for $A(\tau)$. We substitute the result into Equations
\eqref{M1} - \eqref{M2} to obtain inequalities involving only the
estimating functions $M_{1}(\tau)$ and $M_{2}(\tau)$. Consider the
resulting inequality for $M_{2}(\tau)$. The only terms on the
r.h.s., which do not contain $\beta(0)$ to a power at least
$\kappa/2$ as a factor, are $M_{2}^{2}(\tau)$ and $M_{2}^{p}(\tau)$.
Hence for $M_{2}(0)\ll 1$ this inequality implies that $M_{2}(\tau)
\lesssim M_{2}(0)  + \beta^{\frac{\kappa}{2}}(0)$. Substituting this
result into the inequality for $M_{1}(\tau)$ we obtain that
$M_{1}(\tau) \lesssim M_{1}(0)  + \beta^{\frac{\kappa}{2}}(0)$ as
well. The last two inequalities together with \eqref{MajorE} and
\eqref{EstA} imply the desired estimates on $A(\tau)$ and $B(\tau)$.
\end{proof}

%%%%%%%%%%%%%%%%%%%%%%%%%%%%%%%%%%%%%%%%%%%%%%%%%%%%%%%%%%%%%%%%%
%%%%%%%%%%%%%%%%%%%%%%%%%%%%%%%%%%%%%%%%%%%%%%%%%%%%%%%%%%%%%%%%%%%%%%%
\section{Proof of Main Theorem \ref{maintheorem}}\label{SecMain}
%%%%%%%%%%%%%%%%%%%%%%%%%%%%%%%%%%%%%%%%%%%%%%%%%%%%%%%%%%%%%%%%%%%%
We start with an auxiliary statement which eases the induction step.
Recall the notation $I_{t_0, \delta}:= [t_0, t_0 +\delta]$. We say
that $\lambda(t)$ is \textit{admissible} on $I_{t_0, \delta}$ if
$\lambda\in C^{1}(I_{t_0, \delta},\mathbb{R}^{+})\ \mbox{and}\
\lambda^{-3}\partial_{t}\lambda \in [1/4, 1]$. Recall that $t_*$
is the maximal existence time defined in Section \ref{SEC:Intro}.

\begin{lemma} \label{induction}
Assume $u \in C^{1}((0,t_*), \langle x\rangle^3 L^\infty),\ t_0 \in
[0, t_*)$ and $u_{\lambda_0} (\cdot, t_0) \in U_{\epsilon_0/2}$ for
some $\lambda_0$ and for $\epsilon_0$ given in Proposition
\ref{Prop:Splitting}. Then there are $\delta = \delta (\lambda_0,
u)>0$ and $\lambda (t)$, admissible on $I_{t_0, \delta}$, s.t.
\eqref{eqn:splitting2} and \eqref{eqn:splitting2conditions} hold on
$I_{t_0, \delta}$.
\end{lemma}

\begin{proof}
The conditions $u \in C^{1}((0,t_*), \langle x\rangle^3L^\infty)$
and $u_{\lambda_0} (t_0) \in U_{\epsilon_0/2}$ imply that there is a
$\delta = \delta (\lambda_0,u)$ s.t. $u \in \mathcal{U}_{t_0, \delta,
\eps_0, \lambda_0, \alpha_0}$.
%(\textbf{Is this true? Remember the $b-$dependent bound in the definition of}
%$U_{\epsilon_0, \lambda_0, t_0, \delta }$!).
By Proposition \ref{Prop:Splitting2V}, the latter inclusion implies
that there is $\lambda (t)$, admissible on $I_{t_0, \delta},\
\lambda (t_0)= \lambda_0, $ s.t. \eqref{eqn:splitting2} and
\eqref{eqn:splitting2conditions} hold on $I_{t_0, \delta}$.
\end{proof}

%By Theorem \ref{THM:Local} (the local existence theorem, proven in
%Appendix \ref{Appendix:LWP}), there exists $\infty\ge t_{*}>0$ such
%that Equation (\ref{NLH}) has a unique solution $u(x,t)$ in
%$C([0,t_{*}],L^{\infty})$ and, if $t_* < \infty$, then
%$\|u(\cdot,t)\|_{\infty} \rightarrow \infty$ as $t\rightarrow t_*$.

Choose $b_0$ so that $C\|b_{0}\|^2\leq \frac{1}{2}\epsilon_{0}$ with $C$
the same as in \eqref{INI2} and with $\epsilon_{0}$ given in
Proposition \ref{Prop:Splitting}. Let
$v_{0}(y):=\lambda_{0}^{-\frac{2}{p-1}}u_{0}(z_0+\lambda_{0}^{-1}y).$
Then $v_{0}\in U_{\frac{1}{2}\epsilon_{0}}$, by condition
\eqref{INI2} with $m=3$, on the initial conditions. Hence Proposition
\ref{Prop:Splitting}  holds for $v_{0}$ and we have the splitting
\eqref{eqn:split}. Denote $g(v_0) =: (a(0), b(0),z(0))$.
%This implies that the initial condition can be
%written in the form
%\begin{equation} \label{eqn:splitic}
%u_0(x) =V_{a(0) b(0)}(y) + e^{\frac{a(0)y^{2}}{4}}\xi_0(y),
%\end{equation}
%where $(a(0), b(0))=g(v_0)$ and $\xi_0\perp e^{-\frac{a(0)}{4} y^2},
%(1-a(0) y^2)e^{-\frac{a(0)}{4} y^2}$.
%Take $b_{0}=\epsilon_{0}$ as given in Proposition \ref{Prop:Splitting}.

%Furthermore, by continuity there is a (maximal) time $t_\#\leq
%t_{*}$ such that (\ref{eq:realini}) holds for $t< t_\#$ (\textbf{Is
%this true?}). For this time interval Propositions
%\ref{Prop:SplittingIC} and \ref{Prop:SplittingIC2} hold for $u$
%and, in particular, we have the splitting \eqref{eqn:split2}.

Furthermore, by Lemma \ref{induction} there are $\delta_1 > 0$ and
$\lambda_1 (t)$, admissible on $[0, \delta_1]$, s.t. $\lambda_1 (0)=
\lambda_0$ and Equations \eqref{eqn:splitting2} and
\eqref{eqn:splitting2conditions} hold on the interval $[0, \delta_1]$.
Hence, in particular, the estimating functions $M_1(\tau),\
M_2(\tau),\ A(\tau)$ and $B(\tau)$ of Section 5 are defined on the
interval $[0, \delta_1]$. We will write these functions in the
original time $t$, i.e. we will write $M_i(t)$ for $M_i(\tau (t))$
where $\tau(t)=\int_{0}^{t}\lambda^{2}(s)ds$.

Recall the definitions of $\beta(\tau)$ and $\kappa$ given in
\eqref{FunBTau}. Since $\beta(0)$ is the largest eigenvalue of $b(0)$, by Equation
\eqref{INI2} and Proposition \ref{Prop:Splitting},   $A(0)$,
$M_{1}(0)\lesssim 1$ and $M_{2}(0) \ll 1$, while $B(0)\ll 1$, by definition. We have, by continuity,
%(or by Proposition \ref{Prop:SplittingIC}),
that
\begin{equation}\label{ApriorEST}
M_{1}(t),\ A(t), \ B(t)\leq \beta^{-\frac{\kappa}{2}}(0),
\end{equation} for a sufficiently small time interval, which we can take to be
$[0, \delta_1]$. Then by Corollary \ref{cor:aprior} we have that for
the same time interval
%there exists a constant $c$ depending on $A(0)$, $M_{1}(0)$ and
%$M_{2}(0))$ such that
\begin{equation}\label{UtimateEST}
M_{1}(t),\ A(t),\ B(t)\lesssim 1, M_{2}(t)\ll 1.
%c\ \text{and}\ M_{2}(\tau)\ll 1,\
\end{equation}

Equation \eqref{UtimateEST} implies that $u_{\lambda_1} (\cdot,
\delta_1) \in U_{\epsilon_0/2}$ (indeed, by the definitions of
$M_{1}(t)$ and $M_{2}(t)$ we have $\|\langle
y\rangle^{-3}(u_{\lambda_{1}} (\cdot,t)-V_{a(t),b(t)})\|\leq
M_{1}(t) |b(t)|^2\ \text{and}\ \|u(t)\|_{\infty} \lesssim
\lambda_{1}^{\frac{2}{p-1}}(t) [1 + M_1(t) +M_2(t)]$). Now we can
apply Lemma \ref{induction} again and find $\delta_2 > 0$ and
$\lambda_2 (t)$, admissible on $[0, \delta_1 +\delta_2]$, s.t.
$\lambda_2 (t)= \lambda_1 (t)$ for $t \in [0, \delta_1]$ and
Equations \eqref{eqn:splitting2} and
\eqref{eqn:splitting2conditions} hold on the interval $[0, \delta_1
+\delta_2]$.

We iterate the procedure above to show that there is a maximal time
$t^* \le t_*$  ($t_*$ is the maximal existence time), and a function
$\lambda(t)$, admissible on $[0, t^*)$, s.t. \eqref{eqn:splitting2}
and \eqref{eqn:splitting2conditions} and \eqref{UtimateEST} hold on
$[0, t^*)$. We claim that $t^* = t_*$ and $t^* < \infty$ and
$\lambda(t^*) = \infty$. Indeed, if $t^* < t_*$ and $\lambda(t^*) <
\infty$, then by the a priori estimate \eqref{UtimateEST}
$u_{\lambda} (t) \in U_{\epsilon_0/2}$ for any $t \le t^*.$ By Lemma
\ref{induction}, this implies that there is $\delta
>0$ and $\lambda_{\#}(t)$, admissible on $[0, t^*+\delta]$, s.t.
\eqref{eqn:splitting2} and \eqref{eqn:splitting2conditions} hold on
$[0, t^*+\delta]$ and $\lambda_{\#}(t) = \lambda(t)$ on $[0, t^*)$,
which would contradict the assumption that the time $t^*$ is
maximal. Hence
\begin{equation}\label{eq:threePossibilities}
\text{either}\ t^* = t_*\ \text{or}\ t^* < t_*\ \text{and}\
\lambda(t^*) = \infty.
\end{equation} The second case in \eqref{eq:threePossibilities} is ruled out as
follows. Using the relation between the functions $u(x,t)$ and
$v(y,\tau)$ we obtain the following a priori estimate on the
(non-rescaled) solution $u(x,t)$ of equation \eqref{NLH}:
\begin{equation}\label{junk}
\|u(t)\|_{\infty} \lesssim \lambda(t)^{\frac{2}{p-1}} [1 + M_1(t)
+M_2(t)],\end{equation} where we used the fact
$\|\xi(\cdot,\tau (t))\|_{\infty}\lesssim
M_{1}(t)+M_{2}(t)$. By the estimate \eqref{UtimateEST} above the
majorants $M_{j}(t)$ are uniformly bounded and therefore
\begin{equation}\label{upperbound}
\|u(t)\|_{\infty} \lesssim  \lambda(t)^{\frac{2}{p-1}}\  \mbox{for}\
t < t^{*}.
\end{equation}
Moreover \eqref{eqn:split2} and the fact $\|\langle
y\rangle^{-3}\xi\|_{\infty}\lesssim \|b(t)\|^2,$
implied by $M_{1}\lesssim 1$, give
\begin{equation}
|u(0,t)|\ge \lambda(t)^\frac{2}{p-1}\lsb \lb\frac{
c(t)}{p-1}\rb^\frac{1}{p-1}-C \|b(t)\|^2 \rsb\rightarrow\infty,
 \label{eqn:sqeefdsfew}
\end{equation}
as $t\uparrow t^*$, which implies that $t^* \ge t_*$ and therefore
$t_* = t^*$.

Now we consider the first case in \eqref{eq:threePossibilities}. In
this case we must have either $t^* = t_* = \infty$ or $t^* = t_* <
\infty$ and $\lambda(t^*) = \infty$, since otherwise we would have
existence of the solution on an interval greater than $[0, t_*)$.
Finally, the case $t^* = t_* = \infty$ is ruled out in the next
paragraph.
%by the estimate \eqref{eqn:sqeefdsfew} below.
This proves the claim which can reformulated as: there is a function
$\lambda(t)$, admissible on $[0, t_*)$, s.t. \eqref{eqn:splitting2}
and \eqref{eqn:splitting2conditions} and \eqref{UtimateEST} hold on
$[0, t_*)$ and $\lambda(t) \rightarrow \infty$ as $t \rightarrow
t_{*}$. This gives the statements (1) and (2) of Theorem
\ref{maintheorem}.

%Now assume we have shown that there are $0< T < t_*$ and
%$\lambda_T(t)$, admissible on $[0, T)$, s.t. (7)-(8) hold on $[0,
%T)$ and $\lambda_T(T) < \infty$. Then by the a priori estimate
%$u_{\lambda_T} (t) \in U_{\epsilon_0/2}\ \forall t \le T.$ By Lemma
%\ref{induction}, this implies that there is $\delta
%>0$ and $\lambda_{T+\delta}(t)$, admissible on $[0, T+\delta]$, s.t.
%(7)-(8) hold on $[0, T+\delta]$ and $\lambda_{T+\delta}(t) =
%\lambda_T(t)$ on $[0, T)$. This shows that there is
%$\lambda_{t_*}(t)$, admissible on $[0, t_*)$, s.t. (7)-(8) hold on
%$[0, t_*)$ and $\lambda_T(t)
%\rightarrow \infty$ as $t \rightarrow t_*$.\\

By the definitions of $A(t)$ and $B(t)$ in \eqref{majorants} and
the facts that $A(t), B(t)\lesssim 1 $ proved above, we have that
\begin{equation}\label{EstABtau}
a(t)-\frac{1}{2}=-\frac{2}{p-1}\Tr b(\tau)+\O{\beta^{2}(\tau)},\
b(t)=\tilde{\beta}(\tau)+\O{\beta^{1+\kappa/2}(\tau)},
\end{equation}
where, recall, $\tau = \tau(t)=\int_{0}^{t}\lambda^{2}(s)ds$. Hence
$a(t)-\frac{1}{2}=O(\beta(\tau))$.  Recall that
$a=\lambda^{-3}\partial_{t}\lambda$, which can be rewritten as
$\lambda^{-2}(t)=\lambda_{0}^{-2}-2 \int_{0}^{t}a(s)ds$ or
%\begin{equation}
$\lambda(t)=[\lambda_{0}^{-2}-2 \int_{0}^{t}a(s)ds]^{-\frac{1}{2}}.$
%\end{equation}
Assume $t^*=\infty.$ Since $|a(t)-\frac{1}{2}|=O(\beta(\tau))$, there
exists a time $ t^{**}< \infty$ such that $\lambda_{0}^{-2}=2
\int_{0}^{t^{**}}a(s)ds$, i.e. $\lambda(t)\rightarrow \infty$ as
$t\rightarrow t^{**}$. This contradicts the assumption that
$\lambda(t)$ is defined on $[0,t^*=\infty).$ Hence $t^*<\infty$.
This completes the proof of Statements (1) and (2) of Theorem
\ref{maintheorem}.

Now we prove statement (3) of Theorem \ref{maintheorem}.
Equation \eqref{EstABtau} implies $b(t)\rightarrow 0$ and
$a(t)\rightarrow \frac{1}{2}$ as $t\rightarrow t^*$.  By the
analysis above and the definitions of $a$, $\tau$ and $\tilde{\beta}$ (see
\eqref{FunBTau}) we have
$$\lambda(t)=(t^*-t)^{-\frac{1}{2}}(1+o(1)),\ \tau(t)=-\ln|t^*-t|(1+o(1)),$$
and $$\tilde{\beta}(\tau(t))=-\frac{(p-1)^{2}}{4p\ln|t^*-t|}(I+o(1)).$$
This gives the first equation in \eqref{eq:blowdy}. By
\eqref{EstABtau} and the relation $c=a + \frac{1}{2}$ we
have the second and third equations in \eqref{eq:blowdy}.
Finally, let $\zeta(t)=z(t)+\alpha(t)/\lambda(t)$. By \eqref{eqn:evolutionalpha}
and \eqref{Ineq:z_0} we obtain the last equation in
\eqref{eq:blowdy}. This completes the
proof of Theorem \ref{maintheorem}.
\begin{flushright}
$\square$
\end{flushright}

%%%%%%%%%%%%%%%%%%%%%%%%%%%%%%%%%%%%%%%%%%%%%%%%%%%%%%%%%%%%%%%%%%%%%%%%
%%%%%%%%%%%%%%%%%%%%%%%%%%%%%%%%%%%%%%%%%%%%%%%%%%%%%%%%%%%%%%%%%%%%%%%%
% \section{``Gauge" Transform}\label{sec:Gauge}
% Observe that the operator
% $\Delta_y-ay\cdot\nabla_y-\frac{2a}{p-1}$ in \eqref{eqn:BVNLH} is
% not self-adjoint. In order to convert it into a more tractable self-adjoint
% operator we perform a gauge transform. Let
% \begin{equation}\label{eq:definew}
% w(y,\tau):=e^{-\frac{a|y|^2}{4}}v(y,\tau).
% \end{equation}
% Then by \eqref{eqn:BVNLH}, the function $w$ satisfies the equation
% \begin{equation} \label{eqn:w}
% w_\tau=\lb \Delta_y-\frac{a^2+a_\tau}{4}|y|^2
% -\frac{2a}{p-1}+\frac{na}{2} \rb
% w+e^{\frac{a}{4}(p-1)|y|^2}|w|^{p-1}w
% \end{equation}
% The approximate solutions $v_{ab}$ to \eqref{eqn:BVNLH} transform to
% $W_{abc}$, where $W_{abc}(y):=v_{bc}(y)e^{-\frac{a|y|^2}{4}}$.
% Explicitly
% \begin{equation}
% W_{abc}(y):=\lb\frac{c}{p-1+ y b y}\rb^{\frac{1}{p-1}}e^{-\frac{a|y|^2}{4}}.
% \end{equation}

%%%%%%%%%%%%%%%%%%%%%%%%%%%%%%%%%%%%%%%%%%%%%%%%%%%%%%%%%%%%%%%%%%%%%%%%
%%%%%%%%%%%%%%%%%%%%%%%%%%%%%%%%%%%%%%%%%%%%%%%%%%%%%%%%%%%%%%%%%%%%%%%%
\section{Lyapunov-Schmidt Splitting (Effective Equations)} \label{Section:Splitting}
According to Lemma \ref{induction}, the
solution $v(y,\tau)$ of \eqref{eqn:BVNLH} can be decomposed as
\eqref{eqn:split2}, with the parameters $a$ and $b$, and the
fluctuation $\xi$ depending on time $\tau$:
\begin{equation}
v=V_{ab}+\xi,\ \xi\bot\ \phi_{a}^{(ij)}, \, \, 0 \leq i, j \leq n, \label{eqn:split3}
\end{equation}
 in the sense of $L^2(\R^n,\e^{-a|y|^2/2}\d y)$, where $V_{ab}:=\lb \frac{c}{p-1+y b y}
\rb^{\frac{1}{p-1}}$, $c=a+\frac{1}{2}$ and
$\phi_{a}^{(ij)}$ are defined in the beginning of
Section \ref{Section:Reparam}. Plugging the decomposition
\eqref{eqn:split3} into \eqref{eqn:BVNLH} gives the equation (see Appendix 2 for details)
\begin{equation}
\xi_\tau=-{\mathcal L}_{ab}\xi+{\mathcal N}(\xi,a,b)+{\mathcal
F}(a,b) \label{eqn:FluctuationGauged}
\end{equation}
where
\begin{align}
{\mathcal L}_{ab}&=-\Delta_y+ay\cdot\nabla_y+\frac{2 a}{p-1}-\frac{pc}{p-1+y by},\label{eqn:LinOp} \\
{\mathcal N}(\xi,a,b)&=|\xi+V_{ab}|^{p-1}(\xi+V_{ab})-V_{ab}^p-p V_{ab}^{p-1}\xi  \\
{\mathcal F}(a,b)&=\frac{1}{p-1}\left[\Gamma_{0}
+\sum_{j,k}\Gamma_{jk}\frac{(p-1)ay_j y_k}{p-1+y by}
+G_1\right] V_{ab},\label{eqn:Forcing}
\end{align}
with the functions $\Gamma_{jk}\, \, (1\leq j,k \leq n)$ given as
\begin{align}
\Gamma_0&:=-\frac{c_\tau}{c}+(c-2a)-\frac{2}{p-1}\Tr b,\label{eqn:Gamma0}\\
\Gamma_{jk}&:=\frac{1}{a(p-1)}\lb\p_\tau {b_{jk}}-(c-2a)b_{jk}+\frac{2b_{jk}}{p-1}\Tr b
+\frac{4p}{(p-1)^2}\sum_{i=1}^n b_{ij}b_{ik} \rb,\label{eqn:Gamma1} \\
G_1&:=-\frac{4p(y by)(\sum_{i=1}^n (\sum_{j=1}^n b_{ij}y_j)^2)}{(p-1)^2(p-1+y by)^2}.\label{eqn:G_1}
\end{align}

\begin{prop}
\label{Prop:ForcingBounds} If $A(\tau),\
B(\tau)\leq{\beta}^{-\frac{\kappa}{2}}(\tau)$
and $1/4\leq c(0)\leq 1$, then
\begin{align}
\|\lra{y}^{-3}{\cal
F}\|_\infty&=\O{|\Gamma_0|+\sum_{j,k}|\Gamma_{jk}|+{\beta}^\frac{5}{2}}\label{eqn:Festimate0}
\end{align}
and \begin{align}\|{\cal
F}\|_\infty&=\O{|\Gamma_0|+\frac{1}{\beta}\sum_{j,k}|\Gamma_{jk}|+\beta}.
\label{eqn:Festimate}
\end{align}
where, recall, $\lra{y}:=(1+y_1^2+\cdots+y_n^2)^{\frac{1}{2}}$. Furthermore
we have for ${\cal N}={\cal N}(\xi,a,b)$
\begin{equation}
|{\cal N}|\lesssim
|\xi|^{2}+|\xi|^{p}.
\label{eqn:69a}
\end{equation}
\end{prop}
\begin{proof}
\DETAILS{Rearranging the leading term of expression for $\mathcal{F}$ so that
${y_i}^2$ appears in the combination $a y_i^2-1$ gives the more
convenient expression
\begin{equation}
{\cal
F}=\frac{1}{p-1}\lsb\Gamma_0+\sum_{i=1}^n\lb\Gamma_i+\Gamma_i(ay_i^2-1)
-\Gamma_{i}\frac{ay_i^2\sum_{j=1}^nb_jy_j^2}{p-1+\sum_{j=1}^nb_jy_j^2}\rb+G_{1}
\rsb v_{ab} \label{eqn:ForcingDecomposed}
\end{equation}}
We estimate $\|\lra{y}^{-3}{\cal F}\|_\infty$
using the expression of ${\cal F}$ and the estimates
\begin{align*}
\| V_{ab}\|_\infty,\ \|\lra{y}^{-3}y_jy_k
\|_\infty \ \lesssim 1.
\end{align*}
The result is
\begin{equation}
\|\lra{y}^{-3}{\cal F}\|_\infty\lesssim
|\Gamma_0| + \sum_{j,k}|\Gamma_{jk}|
+\|b\|^\frac{5}{2}.\label{eqn:singleast}
\end{equation}

Now we estimate $\|{\cal F}\|$. Recall the
expression of $\mathcal{F}$ in Equation \eqref{eqn:Forcing}. We use
the estimates
\begin{align*}
\left\|V_{ab}\right\|_\infty,\
\left\|\frac{b_{jk} y_j y_k}{(p-1 + y by)^2}V_{ab}\right\|_\infty&\lesssim 1
\end{align*}
to obtain that
\begin{equation}
\|{\cal
F}\|_\infty\lesssim|\Gamma_0|+\sum_{j,k}\frac{1}{\|b\|}|\Gamma_{jk}|+\|b\|.
\label{eqn:doubleast}
\end{equation}

To complete the proof we estimate $b$ in terms of $\beta$ and
$B$ of the first bound. The assumption that $B\leq
\beta^{-\frac{\kappa}{2}}$ implies that $\|b\|=
\beta+\O{\beta^{1+\frac{\kappa}{2}}}$, which together with estimates \eqref{eqn:singleast} and
\eqref{eqn:doubleast}, implies estimates \eqref{eqn:Festimate0} and
\eqref{eqn:Festimate}.

For \eqref{eqn:69a} we observe that if $V_{ab}\le 2|\xi|$ then
$|{\cal N}|\le (3^p+2^p+p2^{p-1})|\xi|^p$.
If $V_{ab}\ge 2|\xi|$, then we use the formula ${\cal
N}=p \int_0^1 \lsb(V_{ab}+s\xi)^{p-1}-V_{b c}^{p-1}\rsb\xi\, ds$ and consider the cases
$1<p\le 2$ and $p>2$ separately to obtain \eqref{eqn:69a}.
\end{proof}

Recall that
$\phi_{a}^{(ij)}=(\sqrt{a}y_i)^{1-\delta_{i0}}(\sqrt{a}y_j)^{1-\delta_{j0}},\, \,
i,j=0,\cdots,n$.
\begin{prop}\label{Prop:LyapunovSchmidtReduction}
Suppose that $A(\tau), {M_1}(\tau)\leq \beta^{-\frac{\kappa}{2}},$
$B(\tau)\leq \beta^{-\frac{\kappa}{2}}$ and
$1/2\leq c(0)\leq 2$ for $0\le\tau\le T$. Let $v=V_{ab}+\xi$ be a
solution to \eqref{eqn:BVNLH} with $\xi\bot \phi_{a}^{(ij)}$
in $L^2(\R^n,\e^{-a|y|^2/2}\d y)$. Over times $0\le\tau\le T$, the parameters $a$ and
$b$ satisfy the equations
\begin{align}
\p_\tau{b}&=-\frac{4p}{(p-1)^2}{b}^2-\frac{2b}{p-1}\Tr b
+(c-2a)b+{\cal R}_{b}(\xi,a,b),\label{eqn:bParameter}\\
\frac{\p_\tau a}{a+\frac{1}{2}}&=(\frac{1}{2}-a)-\frac{2}{p-1}\Tr b+{\cal
R}_a(\xi,a,b),\label{eqn:cParameter}
\end{align}
where the remainders ${\cal R}_{a}$ and ${\cal R}_{b}$ are of the
order $\O{\beta^3+\beta^3 {M_1}(1+
A)+\beta^{4}{M_1}^{2}+\beta^{2p}{M_1}^{p}}$ and satisfy ${\cal
R}_{b}(0,a,b), {\cal R}_{c}(0,a,b)=O(\beta^3).$
\end{prop}
\begin{proof}We take the inner product of equation
(\ref{eqn:FluctuationGauged}) with $\phi_{a}^{(ij)}$ to get
\begin{align*}
\langle \xi_{\tau}, \phi_{a}^{(ij)}\rangle&=\langle -{\mathcal
L}_{ab}\xi+{\mathcal N}(\xi,a,b) +{\mathcal
F}(a,b),\phi_{a}^{(ij)}\rangle.
\end{align*}
We start with analyzing the $\mathcal{F}$
term. The inner product of $\mathcal{F}$ with
$\phi_{a}^{(ij)}$ gives the expressions
\begin{equation}\label{eqn:ProjectedEquation}
\begin{array}{ll}
(p-1)\ip{{\cal F}}{\phi_{a}^{(ij)}}&=\Gamma_0\ip{V_{ab}}{\phi_{a}^{(ij)}}+\langle G_{1}V_{ab},\phi_{a}^{(ij)} \rangle\\
&+\sum_{k,l}\left[\Gamma_{kl}\ip{V_{ab}}{a
y_k y_l\phi_{a}^{(ij)}}-\Gamma_{kl}\langle
\frac{ay_k y_l yby}{p-1+yby}V_{ab},\phi_{a}^{(ij)}\rangle\right].
\end{array}
\end{equation}
By rescaling the variable of integration so that the
exponential term does not contain the parameter $a$ and expanding $V_{ab}$ in $b$ we obtain the estimates
\begin{align*}
&\ip{V_{ab}}{\phi_{a}^{(ij)}}=\lb\frac{a+\frac{1}{2}}{p-1}\rb^{\frac{1}{p-1}}\lb\frac{2\pi}{a}\rb^{\frac{n}{2}}\delta_{ij}+\O{\|b\|},\\
&\ip{V_{ab}}{ay_ky_l\phi_{a}^{(ij)}}=\lb\delta_{kl}\delta_{ij}(1+2\delta_{ik})+
(1-\delta_{kl})(\delta_{ik}\delta_{jl}+\delta_{jk}\delta_{il})\rb
\lb\frac{2\pi}{a}\rb^{\frac{n}{2}}+\O{\|b\|},\\
&\ip{\frac{ay_ky_lyby}{p-1+yby}V_{ab}}{\phi_{a}^{(ij)}}=\O{\|b\|},\\
&\ip{G_1V_{ab}}{\phi_{a}^{(ij)}}=\O{\|b\|^3},
\end{align*}
Where we recall $G_1:=-\frac{4p(y by)(\sum_{i=1}^n (\sum_{j=1}^n b_{ij}y_j)^2)}{(p-1)^2(p-1+y by)^2}$. Substituting
these estimates into Equation \eqref{eqn:ProjectedEquation} gives
\begin{align}
(p-1)\ip{{\cal F}}{\phi_{a}^{(00)}}&=
\lb\frac{a+\frac{1}{2}}{p-1}\rb^{\frac{1}{p-1}}\lb\frac{2\pi}{a}\rb^{\frac{n}{2}}
(\Gamma_0+\sum_{k}\Gamma_{kk})+R_1,\label{est:ForcingPsi0}\\
(p-1)\ip{{\cal F}}{\phi_{a}^{(ii)}}&=
\lb\frac{a+\frac{1}{2}}{p-1}\rb^{\frac{1}{p-1}}\lb\frac{2\pi}{a}\rb^{\frac{n}{2}}
(\Gamma_0+\sum_{k}\Gamma_{kk}+2\Gamma_{ii})
+R_2 \ \text{for} \ 1\leq i\leq n,\label{est:ForcingPsi2}\\
(p-1)\ip{{\cal F}}{\phi_{a}^{(ij)}}&=
2\lb\frac{a+\frac{1}{2}}{p-1}\rb^{\frac{1}{p-1}}\lb\frac{2\pi}{a}\rb^{\frac{n}{2}}\Gamma_{ij}+R_3
\ \text{for}\ 1\leq i < j \leq n, \label{est:ForcingPsi3}
\end{align}
where the remainders $R_1$, $R_2$ and $R_3$ are bounded by
$\O{\|b\|(|\Gamma_0|+\sum_{i,j}|\Gamma_{ij}|)+\|b\|^3}$.

To estimate the remaining terms we differentiate the orthogonality condition $\scalar{\xi}{\phi_a^{ij}}=0$ to obtain
  \begin{align*}
    0=\scalar{\xi_\tau}{\phi_a^{ij}}+\scalar{\xi}{\partial_\tau \phi_a^{ij}}-\frac{a_\tau}{2}\scalar{\xi}{\phi_a^{ij}|y|^2}\,,
  \end{align*}
where the last term is due to the weight $\e^{-\frac{a}{2}|y|^2}$. Now compute
\begin{equation*}
\scalar{\xi}{\partial_\tau\phi_a^{ij}}=0\ \mbox{and}\
\left|\frac{a_\tau}{2}\scalar{\xi}{\phi_{a}^{(ij)}|y|^2}\right|\leq
\left|\frac{1}{2}a^{-1}a_\tau\ip{\lra{y}^{-3}\xi}{a^2\lra{y}^{3}y_iy_j|y|^2}\right|.
\end{equation*}
Estimating the right hand side of the second inequality by
H\"{o}lder's inequality and using the definition of $M_{1}(\tau)$
gives that over times $0\le \tau\le T$
\begin{equation*}
\left|\frac{a_\tau}{2}\scalar{\xi}{\phi_{a}^{(ij)}|y|^2}\right|=\O{|a_\tau| \beta^2 M_{1}}.
\end{equation*}
Next we estimate $a_\tau$. Since $c=\frac{1}{2}+a$, we have
$a_\tau=c_{\tau}$, and so we find from \eqref {eqn:Gamma0}
$$
c_\tau=\O{\Gamma_0+\beta^2A}
$$
for times $0\le \tau\le T$, and hence
\begin{equation}
\left|\frac{a_\tau}{2}\scalar{\xi}{\phi_{a}^{(ij)}|y|^2}\right|=\O{\beta^2 M_{1}(|\Gamma_0|+\beta^2A)}\,.\label{est:TimeDerxi}
\end{equation}
We now estimate the terms involving the linear operator ${\cal L}_{a
b c}$.  Write the operator ${\cal L}_{a b c}$ as
\begin{equation*}
{\cal L}_{ab}={\cal L}_*-\frac{pc}{p-1+yby},
\end{equation*}
where ${\cal L}_*:=-\Delta_y+ay\cdot\nabla_y+\frac{2a}{p-1}$ is self-adjoint on $L^2(\RR^n,\e^{-\frac{a}{2}|y^2|})$ and satisfies ${\cal L}_*\phi_{a}^{(00)}=\frac{2
a}{p-1}\phi_{a}^{(00)}$ and ${\cal L}_*\phi_{a}^{(ij)}=\frac{2a
p}{p-1}\phi_{a}^{(ij)}+2\delta_{ij}$ for $1\leq i, j\leq n$. Projecting ${\cal L}_{a b c}\xi$ onto the
eigenvectors $\phi_{a}^{(00)}$ and $\phi_{a}^{(ij)}$ of ${\cal L}_*$ gives
the equations
\begin{align*}
\ip{{\cal L}_{ab}\xi}{\phi_{a}^{(00)}}&=-\ip{\xi}{\frac{pc}{p-1+yby}
}=\frac{pc}{p-1}\ip{\xi}{\frac{yby}{p-1+yby}
},\\
\ip{{\cal L}_{ab}\xi}{\phi_{a}^{(ij)}}&=-\ip{\xi}{\frac{pca
y_iy_j}{p-1+yby}}=\frac{pc}{p-1}\ip{\xi}{\frac{a
y_iy_jyby}{p-1+yby}}.
\end{align*}
Estimating with H\"{o}lder's inequality gives the inequalities
\begin{align*}
|\ip{{\cal L}_{ab}\xi}{\phi_{a}^{(00)}}|&\lesssim \|b\|\|\lra{y}^{-3}\xi\|_\infty\\
|\ip{{\cal L}_{ab}\xi}{\phi_{a}^{(ij)}}|&\lesssim
\|b\|\|\lra{y}^{-3}\xi\|_\infty.
\end{align*}
In terms of the estimating functions $\beta$ and $M_{1}$, these
estimates, after using the above estimate of $a_\tau$ and
simplifying in $a$ and $c$, become
\begin{align}
\ip{{\cal L}_{ab}\xi}{\phi_{a}^{(00)}}&\lesssim \beta^{3}M_{1}\label{Est:LinearOpPsi0}\\
\ip{{\cal L}_{ab}\xi}{\phi_{a}^{(ij)}}&\lesssim \beta^3M_{1}.\label{Est:LinearOpPsi2}
\end{align}

Lastly, we estimate the inner products involving the nonlinearity.
Because of \eqref{eqn:69a}, both $\ip{{\cal N}}{\phi_{a}^{(00)}}$ and
$\ip{{\cal N}}{\phi_{a}^{(ij)}}$ are estimated by $\O{\|\lra{y}^{-3}
\xi\|_\infty^2+\|\lra{y}^{-3}
\xi\|_\infty^p}.$  Writing this in terms of
$\beta$ and $M_{1}$ and simplifying gives the estimate
\begin{equation}
|\ip{{\cal N}}{\phi_{a}^{(00)}}|,\, \, |\ip{{\cal
N}}{\phi_{a}^{(ij)}}|\lesssim \beta^4 M_{1}^2+\beta^{2p}
M_{1}^p.\label{est:NonPsi0Psi2}
\end{equation}

Estimates \eqref{est:ForcingPsi0}-\eqref{est:NonPsi0Psi2} imply that $\Gamma_0$ and $\Gamma_{ij}$ are of the order
\begin{equation*}
\O{\beta(|\Gamma_0|+\sum_{i,j}|\Gamma_{ij}|)+\beta^3+\beta^2M_{1}\lb
\beta+|\Gamma_0|+\beta^2A \rb+\beta^4M_{1}^2+\beta^{2p} M_{1}^{p}}.
\end{equation*}
By the facts that $\beta(\tau)\le \beta_{0}\ll 1$ and $A, M_{1}\leq
\beta^{-\frac{\kappa}{2}} $, we obtain the estimates
\begin{equation}
|\Gamma_0|+\sum_{i,j}|\Gamma_{ij}|\lesssim \beta^3+\beta^3M_{1}(1+A)+\beta^4
M_{1}^2+\beta^{2p}M_{1}^{p} \label{GammaEst}
\end{equation}
for the times $0\le \tau\le T$.
\end{proof}

Equations \eqref{eqn:Festimate0}, \eqref{eqn:Festimate} and \eqref{GammaEst} yield the
following corollary.
\begin{cor}
Let $k_{0}:=\min\{1,2p-1\}$ and $k_{3}:=\min\{5/2,2p\}.$
Then for $m=0$ and $3$
\begin{equation}\label{eq:Festimate}
\begin{array}{lll}
\|\lra{y}^{-m}{\cal F}\|_\infty&\lesssim
&\beta^{k_{m}}(\tau)[1+ M_1(1+ A)+ M_1^2+ M_1^{p}].
\end{array}
\end{equation}
\end{cor}

%%%%%%%%%%%%%%%%%%%%%%%%%%%%%%%%%%%%%%%%%%%%%%%%%%%%%%%%%%%%%%%%%%
%%%%%%%%%%%%%%%%%%%%%%%%%%%%%%%%%%%%%%%%%%%%%%%%%%%%%%%%%%%%%%%%%

\section{Proof of Estimates \eqref{EstB}-\eqref{EstA}}\label{SEC:EstB}
Recall that $a=c-\frac{1}{2}$. Assume $B(\tau)\le
\beta^{-\frac{\kappa}{2}}(\tau)$ for
$\tau\in[0,T]$ which implies that $\tilde{\beta} \lesssim b\lesssim \tilde{\beta}.$
We rewrite equation
\eqref{eqn:bParameter} as $\p_\tau b=-\frac{4p}{(p-1)^{2}}b^2+b\lb
\frac{1}{2}-a-\frac{2\Tr b}{p-1} \rb+{\cal R}_{b}$. By
the definition of $A$, the second term on the right hand side is
bounded by $\|b\|\beta^{2}A\lesssim \beta^{3}A$. Thus, using the
bound for ${\cal R}_{b}$ given in Proposition
\ref{Prop:LyapunovSchmidtReduction}, we obtain \eqref{EstB}.

To prove \eqref{MajorE} we use the inequality $\beta I \lesssim b$
to obtain the estimate
\begin{equation}
\left\| -\p_\tau b^{-1}+\frac{4p}{(p-1)^{2}}I \right\|\lesssim
\beta+\beta M_{1}(1+A)+\beta^{2}M_{1}^{2}+\beta^{2p-2}M_{1}^{p}.
\label{est:InversebDE}
\end{equation}
Since $\tilde{\beta}$ is a solution to $-\p_\tau
\tilde{\beta}^{-1}+4p(p-1)^{-2}I=0$, Equation \eqref{est:InversebDE}
implies that
\begin{equation*}
\left\|\p_\tau\lb b^{-1}-\tilde{\beta}^{-1}\rb\right\|\lesssim
\beta+\beta M_{1}(1+A)+\beta^{2}M_{1}^{2}+\beta^{2p-2}M_{1}^{p}.
\end{equation*}
Integrating this equation over $[0,\tau]$, multiplying the result by
$\beta^{-1-\kappa}$ and using that $\tilde{\beta}(0)=b(0)$, gives the estimate
\begin{equation*}
\beta^{-1-\kappa}\|\tilde{\beta}-b\|\lesssim \beta^{1-\kappa}
\int_0^\tau\lb \beta+\beta M_{1}(1+
A)+\beta^{2}M_{1}^{2}+\beta^{2p-2}M_{1}^{p}\, \rb ds,
\end{equation*}
where, recall, $\kappa:=\min\{\frac{1}{2}, \frac{p-1}{2}\}<1$.
Hence, by the definition of $\beta$ and $B$ and the facts that
$M_{1}$ and $A$ are increasing functions, \eqref{MajorE} follows.

Define the quantity
$\Gamma:=\frac{1}{2}-a-\frac{2}{p-1}\Tr b$.
 Differentiating $\Gamma$ with respect to $\tau$ and substituting for
$\p_\tau b$ and $a_{\tau}=c_\tau$. From equations
\eqref{eqn:bParameter} and \eqref{eqn:cParameter}, we obtain
\begin{equation*}
\p_\tau\Gamma=-c(\Gamma+{\cal R}_c)-\frac{2}{p-1}\Tr\lb
-\frac{4p}{(p-1)^2}b^2-\frac{2b}{p-1}\Tr b+(c-2a)b+{\cal
R}_{b} \rb.
\end{equation*}
Replacing $b(c-2a)$ by $b\Gamma+\frac{2b}{p-1}\Tr b$ and rearranging the resulting equation gives that
\begin{equation*}
\p_\tau \Gamma+\lsb
a+\frac{1}{2}+\frac{2}{p-1}\Tr b\rsb\Gamma=\frac{8
p}{(p-1)^3}\Tr b^2-(a+\frac{1}{2}){\cal
R}_c-\frac{2}{p-1}{\cal R}_{b}.
\end{equation*}
Let $\mu=\exp\lb \int_{0}^{\tau}\lb
a+\frac{1}{2}+\frac{2}{p-1}\Tr b\rb ds\rb$. We now
integrate the above equation over $[0,\tau]\subseteq [0,T]$. Then
the above equation implies that
\begin{equation*}
\mu(\tau)\Gamma(\tau)-\mu(0)\Gamma(0)=\int_0^\tau\p_\tau(\mu\Gamma)
=\frac{8p}{(p-1)^3}\int_0^\tau \mu \Tr b^2\, ds- \int_0^\tau
(a+\frac{1}{2})\mu{\cal R}_c\, ds- \int_0^\tau
\frac{2}{p-1}\mu{\cal R}_{b}\, ds.
\end{equation*}
Use the inequality $\|b\|\lesssim \beta$ and the estimates of ${\cal
R}_{b_i}$ and ${\cal R}_c$ in Proposition
\ref{Prop:LyapunovSchmidtReduction} to obtain
\begin{equation*}
|\Gamma|\lesssim \mu^{-1}\Gamma(0)+\mu^{-1} \int_0^\tau \mu
\beta^2\, ds+\mu^{-1} \int_0^\tau \mu\lb \beta^3+\beta^3
M_{1}(1+A)+\beta^{4}M_{1}^{2}+\beta^{2p}M_{1}^{p} \rb\, ds.
\end{equation*}
For our purpose, it is sufficient to use the less sharp inequality
\begin{equation*}
|\Gamma|\lesssim\mu^{-1}\Gamma(0)+\lb
1+\beta(0)M_{1}(1+A)+\beta^2(0)M_{1}^{2}+\beta^{2p-2}(0)M_{1}^{p}
\rb \mu^{-1} \int_0^\tau \mu \beta^2\, ds.
\end{equation*}
The assumption that $A(\tau),
B(\tau)\leq\beta^{-\frac{\kappa}{2}}(\tau),$ implies that
$a+\frac{1}{2}+\frac{2}{p-1}\Tr b=1+\O{\beta^2A}\geq
\frac{1}{2}$ and therefore $\beta^{-2}\mu^{-1}\lesssim
\beta^{-2}(0)$ and $ \int_0^\tau \mu(s) \beta^2(s)\, ds\lesssim
\mu(\tau)\beta^2(\tau)$. The last two inequalities and the relation
$\displaystyle\max_{s\leq \tau}\beta^{-2}(s)|\Gamma(s)|=A(\tau)$
lead to \eqref{EstA}.

%%%%%%%%%%%%%%%%%%%%%%%%%%%%%%%%%%%%%%%%%%%%%%%%%%%%%%%%%%%%%%%%%%%%%%%%
%%%%%%%%%%%%%%%%%%%%%%%%%%%%%%%%%%%%%%%%%%%%%%%%%%%%%%%%%%%%%%%%%%%%%%%%
\section{Rescaling of Fluctuations on a Fixed Time Interval}
\label{Section:Rescaling} The coefficient in front of $|y|^2$ in the
operator $\mathcal{L}_{ab}$, \eqref{eqn:LinOp}, is time dependent,
complicating the estimation of the semigroup generated by this
operator. In this section we introduce new time and space
variables in such a way that the coefficient of $|y|^2$ in the new
operator is constant (cf \cite{BP1, BuSu, DGSW, Per}).

Let $T$ be given and let $t(\tau)$ be the inverse of the function
$\tau(t):= \int_0^t\lambda^2(s)\, ds$.
 We approximate the scaling parameter $\lambda(t)$ over the time
interval $[0,t(T)]$ by a new parameter $\lambda_1(t)$.  We choose
$\lambda_1(t)$ to satisfy for $t\leq t(T)$
\begin{equation*}
\p_t\lb \lambda_1^{-3}\p_t\lambda_1 \rb=0\ \mbox{with} \
\lambda_1(t(T))=\lambda(t(T))\ \mbox{and}\
\p_t\lambda_1(t(T))=\p_t\lambda(t(T)).
\end{equation*}
We define $\alpha:=\lambda_1^{-3}\p_t\lambda_1=a(T)$.  This is an
analog of the parameter $a$ and it is constant.  The last two
conditions imply that $\lambda_1$ is tangent to $\lambda$ at
$t=t(T)$.  Define the new time and space variables as
\begin{equation*}
z=\frac{\lambda_1}{\lambda} y\ \mbox{and}\ \sigma=\sigma(t(\tau))\
\text{with}\ \sigma(t):= \int_0^t \lambda_1^2(s)\, ds
\end{equation*}
where $\tau\leq T$, $\sigma\leq S:=\sigma(T)$ and $\lambda$, $\lambda_{1}$ are functions of $t(\tau).$ Now we introduce the new
function $\eta(z,\sigma)$ by the equality
\begin{equation}\label{NewFun}
\lambda_1^\frac{2}{p-1} \eta(z,\sigma)=\lambda^\frac{2}{p-1}\xi(y,\tau).
\end{equation}

Denote by $t(\sigma)$ the inverse of the function $\sigma(t)$. In
the equation for $\eta(z,\sigma)$ derived below and in what follows
the symbols $\lambda$, $a$ and $b$ stand for $\lambda(t(\sigma)),$
$a(\tau(t(\sigma)))$ and $b(\tau(t(\sigma)))$, respectively.
 Substituting this change of variables into
\eqref{eqn:FluctuationGauged} gives the governing equation for
$\eta$:
\begin{equation}\label{eq:eta}
\p_\sigma\eta=-L_{\alpha}\eta+W(a,b,\alpha)\eta+F(a,b,\alpha)+N(\eta,a,
b,\alpha),
\end{equation}
where
\begin{align}
&L_{\alpha}:=L_0+V,\
L_0:=-\Delta_z+\alpha z\cdot\nabla_z-2\alpha,\
V:=\frac{2 p\alpha}{p-1}-\frac{2p\alpha}{p-1+z\tilde{\beta} z}, \label{eqn:DefnL0V}\\
&W(a,b,\alpha):=\frac{\lambda^2}{\lambda_1^2}\frac{ p
(a+\frac{1}{2})}{p-1+\frac{\lambda^2}{\lambda_1^2}z bz}
-\frac{2 p\alpha}{p-1+z\tilde{\beta} z},\nonumber\\
&F(a,b,\alpha):=\lb\frac{\lambda}{\lambda_1}\rb^\frac{2p}{p-1}
{\cal
F}(a,b,c)\nonumber
\end{align}
and
\begin{align*}
&N(\eta,a,b,\alpha):=\lb\frac{\lambda}{\lambda_1}\rb^\frac{2p}{p-1}
{\cal N}\lb
\lb\frac{\lambda_1}{\lambda}\rb^\frac{2}{p-1} \eta,b,c\rb,
\end{align*}
where, recall, $c$ and $a$ are related as $c=a+\frac{1}{2}$ and
$\beta$ is defined in \eqref{FunBTau}.

In the next statement we prove that the new parameter
$\lambda_{1}(t)$ is a good approximation of the old one,
$\lambda(t)$. The proof is an exact copy of the one in \cite{DGSW}. We reproduce it for completeness. We have
\begin{prop}\label{NewTrajectory}
If $A(\tau)\leq \beta^{-\frac{\kappa}{2}}(\tau)$ and
$\beta(0) \ll 1$, then
\begin{equation}\label{eq:appro}
|\frac{\lambda}{\lambda_{1}}(t(\tau))-1|\lesssim \beta(\tau)\leq
\beta(0).
\end{equation}
\end{prop}
\begin{proof}
Differentiating $\frac{\lambda}{\lambda_1}-1$ with respect to $\tau$
(recall that $\frac{dt}{d\tau}=\frac{1}{\lambda^2}$) gives the
expression
\begin{equation*}
\frac{d}{d\tau}\lb\frac{\lambda}{\lambda_1}-1\rb=\frac{\lambda}{\lambda_1}a-\frac{\lambda_1}{\lambda}\alpha
\end{equation*} or, after some manipulations,
\begin{equation}\label{EstLambda}
\frac{d}{d\tau}\(\frac{\lambda}{\lambda_{1}}-1\)
=2a(\frac{\lambda}{\lambda_{1}}-1)+\Gamma
\end{equation} with
$$\Gamma:=a-\alpha-a\frac{\lambda_{1}}{\lambda}(\frac{\lambda}{\lambda_{1}}-1)^{2}
+(a-\alpha)(\frac{\lambda_{1}}{\lambda}-1).$$
Observe that $\frac{\lambda}{\lambda_{1}}(t(\tau))-1=0$ when
$\tau=T.$ Thus Equation \eqref{EstLambda} can be rewritten as
\begin{equation}
\frac{\lambda}{\lambda_{1}}(t(\tau))-1=- \int_{\tau}^{T}e^{-
\int^{\sigma}_{\tau}2a(\rho)d\rho}\Gamma(\sigma)d\sigma.
\label{eqn:89}
\end{equation}
By the definition of $A(\tau)$ and the definition $\alpha=a(T)$ we
have that, if $A(\tau)\leq \beta^{-\frac{\kappa}{2}}(\tau)$, then
\begin{equation}\label{CauchA}
|a(\tau)-\alpha|,\ |a(\tau)-\frac{1}{2}|\leq 2\beta(\tau)
\end{equation} on the time interval $\tau\in [0,T]$. Thus
\begin{equation}\label{Rem}
|\Gamma|\lesssim
\beta+(1+\frac{\lambda_{1}}{\lambda})(\frac{\lambda}{\lambda_{1}}-1)^{2}+\beta|\frac{\lambda}{\lambda_{1}}-1|.
\end{equation}
which together with \eqref{eqn:89} and \eqref{CauchA} implies
\eqref{eq:appro}.
\end{proof}
%%%%%%%%%%%%%%%%%%%%%%%%%%%%%%%%%%%%%%%%%%%%%%%%%%%%%%%%%%%%%%%%%%%%%%%%%%%%
%%%%%%%%%%%%%%%%%%%%%%%%%%%%%%%%%%%%%%%%%%%%%%%%%%%%%%%%%%%%%%%%%%%%%%%%%%%%%
\section{Estimate of the Propagators}\label{Section:PropEst}
Let %$\bar{P}^{\alpha}$ be the orthogonal projection operator(in the sense of the $L^2$ scalar product) onto the space spanned by the eigenvectors of $L_{0}$ corresponding to the smallest three eigenvalues and $P^{\alpha}:=1-\bar{P}^{\alpha}$. Thus
$P^{\alpha}$ be the orthogonal projection onto the orthogonal
complement of the space spanned by the eigenvectors of $L_{0}$
corresponding to the smallest three eigenvalues. Denote by
$V_{\alpha}(\tau,\sigma)$ the propagator generated by the
operator $-P^{\alpha}L_{\alpha}P^{\alpha}$ on ${\rm Ran}\,
P^\alpha$, where, recall, the operator $L_{\alpha}$ is defined in
\eqref{eqn:DefnL0V}. The main result of this section is the
following theorem.
\begin{thm}\label{ProP}
For any function $g\in \Ran P^{\alpha}$ and for $c_0:=\alpha -
\epsilon$ with some $\epsilon>0$ small we have
$$\|\langle z\rangle^{-3}V_{\alpha}(\tau,\sigma)g\|_{\infty}\lesssim
e^{-\cO(\tau-\sigma)}\|\langle z\rangle^{-3}g\|_{\infty}.$$
\end{thm}
The proof of this theorem is given after Lemma \ref{LM:FK}. We
observe that in the $L^{2}$-norm $P^{\alpha}L_{\alpha}P^{\alpha}\geq
(-\Delta_z+\alpha z\cdot\nabla_z-2\alpha)P^{\alpha}\geq
\frac{1}{2}\alpha P^{\alpha}$. However, this does not help in
proving the weighted $L^{\infty}$ bound above. Recall the definition
of the operator
$L_{0}:=-\Delta_z+\alpha z\cdot\nabla_z-2\alpha$ in
(\ref{eqn:DefnL0V}) and define $U_0(x,y)$ as the integral kernel of
the operator $e^{-rL_{0}}$. We begin with
\begin{lemma}\label{kernelEst}
For $k=0,1,2,3,4$, any function $g$ and $r>0$ we have that
\begin{equation}
\|\langle z\rangle^{-k}e^{-L_{0}r}g\|_{\infty}\lesssim e^{2\alpha r}\|\langle
z\rangle^{-k}g\|_{\infty}
\label{est:99a}
\end{equation}
 or equivalently
\begin{equation}\label{eq:secondForm}
\int \langle
x\rangle^{-k}U_{0}(x,y)\langle
y\rangle^{k}dy\lesssim e^{2\alpha r}.
\end{equation}
\end{lemma}
\begin{proof} We only prove the case $k=2.$ The cases $k=0,4$ are
similar.  The cases $k=1,3$ follow from $k=0,2,4$ by an
interpolation result. For the case $k=2$, using that the integral
kernel of $e^{-r L_0}$ is positive and therefore $\|e^{-r L_0}
g\|_\infty\le \|f^{-1} g\|_\infty\|e^{-r L_0} f\|_\infty$ for any
$f>0$ and using that $e^{-r L_0} 1=e^{2\alpha r} 1$ and $e^{-r
L_0}(\alpha |z|^2-n)=(\alpha
|z|^2-n)$, we find that
$$
\begin{array}{lll}
\|\langle z\rangle^{-2}e^{-rL_{0}}g\|_{\infty} &\leq &\|\langle z\rangle^{-2}
e^{-rL_{0}}(|z|^{2}+1)\|_{\infty}\|\langle
z\rangle^{-2}g\|_{\infty}\\
&=& \|\langle z\rangle^{-2}[e^{2\alpha
r}(\frac{n}{\alpha}+1)+(|z|^{2}-\frac{n}{\alpha})]\|_{\infty}\|\langle
z\rangle^{-2}g\|_{\infty}\\
&\leq&2(\frac{n}{\alpha}+1)e^{2\alpha r}\|\langle
z\rangle^{-2}g\|_{\infty}.
\end{array}
$$
This implies \eqref{est:99a}.  To prove \eqref{eq:secondForm} we use
that $U_0(x,y)$ is, by definition, the integral kernel of the
operator $e^{-r L_0}$, and take $g(x)=\lra{x}^k$ in
\eqref{est:99a} to obtain \eqref{eq:secondForm}.
\end{proof}

Next we prove a more refined bound on the free evolution
$e^{-rL_0}$.

\begin{lemma}\label{EstiA}
For any function $g$ and positive constant $r$ we have
$$\left\|\langle z \rangle^{-3}e^{-rL_0}P^{\alpha}g\right\|_{\infty} \lesssim
e^{-\alpha r}\left\|\langle z \rangle^{-3}g\right\|_{\infty}.$$
\end{lemma}

\begin{proof}  First, we decompose the projection $P^{\alpha}$ in a
convenient way. We write the operator $L_0$ as
\begin{equation}\label{L0Decom}
L_0 =\sum_{k=1}^{n}L_0^{(k)}-2\alpha,\quad \text{where}\quad
L_0^{(k)}:=-\p_{z_k}^2+\alpha z_k\p_{z_k}.
\end{equation}
The spectra of the operators $L_0^{(k)}$ are:
\begin{align}
\sigma\lb L_0^{(k)}\rb&=\left\{m\alpha |\ m=0,1,2,\ldots\right\}.
\end{align}

Let $P_0^{(k)}$, $P_1^{(k)}$ and $P_2^{(k)}$ be the orthogonal projections onto the eigenspaces of the operator $L_0^{(k)}$ corresponding to the first, the second and third eigenvalues of $L_0^{(k)}$, respectively, and let
$$P_3^{(k)}:=1-P_0^{(k)}-P_1^{(k)}-P_2^{(k)},$$
$$P_{0'}^{(k)}:=1,\ P_{1'}^{(k)}:=1-P_0^{(k)},$$
$$P_{2'}^{(k)}:=1-P_0^{(k)}-P_1^{(k)}.$$
Then for any $k$, we have
\begin{equation}\label{eqn:projectdecom}
\begin{array}{l}
P_0^{(k)}+P_1^{(k)}+P_2^{(k)}+P_3^{(k)}=1,\\
P_{0'}^{(k)}=1,\ P_0^{(k)}+P_{1'}^{(k)}=1,\\
P_0^{(k)}+P_1^{(k)}+P_{2'}^{(k)}=1.
\end{array}
\end{equation}

Let $\vec{i}=(i_1,i_2,\cdots,i_n),\ i_j=0,\ 0',\ 1,\ 1',\ 2,\ 2',\ 3,$  $
 |\vec{i}|=\sum_{j=1}^n i_j,$ where the primed numbers are counted as the usual ones, and
$ P_{\vec{i}}=P_{i_1}^{(1)}P_{i_2}^{(2)}\cdots P_{i_n}^{(n)}$. For every $k\in \{1, \cdots, n\}$,
we introduce the set
 $$I_k\equiv I_k^{(n)}=\{\vec{i}=(i_1,\cdots,i_n)|\,
 \text{either}\,  |\vec{i}|=3 \  \text{and}\, i_k\neq 0',\ 1',\ 2',\    \text{or}\, |\vec{i}|<3\
 \text{and}\, i_j\neq 0',\ 1',\ 2'\  \forall 1\leq j\leq n  \}.$$
Then we have the following lemma, whose proof is given in Appendix 2: %\ref{sec:comp}:
\begin{lemma}\label{IdDecom}
For any $1\leq k\leq n$, there exists a subset $J_k$ of $I_k$ such
that $1=\sum_{\vec{i}\in J_k}P_{\vec{i}}$.
\end{lemma}
%For a proof of this lemma please see Appendix 2.

Since for any $k$ $$\sum_{\vec{i}\in I_k, |\vec{i}|=j, i_l\neq 0', 1', 2'\ \forall l}P_{\vec{i}}$$ is the eigenprojection corresponding to the $j-$th eigenvalue of $L_0,\ j=0, 1, 2$,
we have, by the definition of $P^{\alpha}$ and Lemma \ref{IdDecom}, that
\begin{equation}\label{P_alpha}
P^{\alpha}=\sum_{|\vec{i}|=3,
\vec{i}\in J_k}P_{\vec{i}}, \ \forall k\in \{1,2,\cdots,n \}.
\end{equation}
Equations \eqref{L0Decom} and  \eqref{P_alpha} give
\begin{equation}\label{Decomposition}
e^{-rL_0}P^{\alpha}=\sum_{|\vec{i}|=3, \vec{i}\in
J_k}e^{-rL_0}P_{\vec{i}} =e^{2\alpha r}\sum_{|\vec{i}|=3,
\vec{i}\in J_k}\prod_{j=1}^n\lb e^{-rL_0^{(j)}}P_{i_j}^{(j)}\rb .
\end{equation}

In the following, it is convenient to use the notation $z_0:=1$. By the inequality $\langle z \rangle^{3} \lesssim \sum_{k=0}^n|z_k|^3$, we have
%We first prove that
\begin{equation}\label{normtoA_k}
\left\|\langle z \rangle^{-3}e^{-rL_0}P^{\alpha}%\langle\langle x \rangle\rangle^{3}
\langle z \rangle^{3}\right\|_{L^{\infty}\rightarrow L^{\infty}} \lesssim %e^{-\alpha r}.
%\label{EquiEstimate} \end{equation} \begin{equation}\label{A_k}
\left\|\langle z \rangle^{-3}e^{-rL_0}P^{\alpha}\sum_{k=0}^n|z_k|^3\right\|_{L^{\infty}\rightarrow L^{\infty}}\lesssim
\sum_{k=0}^nA_k,
\end{equation}
where $A_k=\left\|\langle z \rangle^{-3}e^{-rL_0}P^{\alpha}|z_k|^3\right\|_{L^{\infty}\rightarrow L^{\infty}}$ for $0\leq
k\leq n$. Now by \eqref{Decomposition} and $\langle z
\rangle^{-3}\leq \prod_{j=1}^n\langle z_j \rangle^{-i_j}$, we obtain
that for $0\leq k\leq n$,
\begin{align*}
A_k&\leq\sum_{\vec{i}\in J_k, |\vec{i}|=3}\left\|\langle z
\rangle^{-3}
e^{-rL_0}P_{\vec{i}}|z_k|^3
\right\|_{L^{\infty}\rightarrow L^{\infty}}\\
&\lesssim \sum_{\vec{i}\in J_k, |\vec{i}|= 3}\left\|\prod_{j=1}^n\langle z_j \rangle^{-i_j}
e^{-rL_0}P_{\vec{i}}|z_k|^3
\right\|_{L^{\infty}\rightarrow L^{\infty}}\\
&=e^{2\alpha r}\sum_{\vec{i}\in J_k, |\vec{i}|= 3}\left\|
\prod_{j=1}^n\lb \langle z_j \rangle^{-i_j}
e^{-rL_0^{(j)}}P_{i_j}^{(j)}|z_k|^{3\delta_{jk}}
\rb\right\|_{L^{\infty}\rightarrow L^{\infty}}.
\end{align*}

We claim that, if $\vec{i}\in J_k$, then
\begin{equation}\label{x_jnorm}
\left\|\langle z_j \rangle^{-i_j}
e^{-rL_0^{(j)}}P_{i_j}^{(j)}|z_k|^{3\delta_{jk}}
\right\|_{L^{\infty}\rightarrow L^{\infty}}\lesssim
e^{-i_j\alpha r},
\end{equation}
Indeed if $j=k\geq
1$, then $i_k\neq 0',\ 1',\ 2'$. For $i_j=0,\ 1,\,\ \text{or}\,\ 2$
\eqref{x_jnorm} follows from the relation $e^{-rL_0^{(j)}}P_{i_j}^{(j)}=e^{-i_j\alpha
r}P_{i_j}^{(j)}$ which is due to the definition of $P_{i_j}^{(j)}$. For $i_j=3$ it is proved in \cite{DGSW} using
integration by parts (see Appendix 2).
If $j\neq k$ (which is, in particular, the
case when $k=0$), the proof is similar. Then by the above two
inequalities and the relation $\sum_{j=1}^n i_j=|\vec{i}|=3$, we
obtain
\begin{equation}\label{EstiAk}
A_k\lesssim e^{-\alpha r}\  \text{for}\ 0\leq k\leq n.
\end{equation}
Equations \eqref{normtoA_k} and \eqref{EstiAk} imply
the statement of the lemma.
\end{proof}

Next, we estimate the propagator $U_{\alpha }(\tau,\sigma),$
generated by the operator $-L_{\alpha}$.
\begin{prop}\label{propagator}
For any function $g$ and positive constants $\sigma$ and $r$ we have
\begin{equation}\label{EstiAB}
\|\langle z\rangle^{-3}
U_{\alpha}(\sigma+r,\sigma)P^{\alpha}g\|_{\infty} \lesssim
[e^{2\alpha r}r(1+r){\beta^{1/2}(\sigma)}+e^{-\alpha r}]\|\langle
z\rangle^{-3} g\|_{\infty}.
\end{equation}
\end{prop}

\begin{proof}Let $B_{\lambda}, \lambda\in \frac{R}{2}\mathbb{Z}^n$, be
a collection of semi-open, disjoint boxes centered at $\lambda$, of
sidelength $R$, whose union is $\R^n$. We take $R\leq \frac{1+r}{2}$. Let
$g_{\lambda}(x)=g(x)\chi_{\lambda}(x)$, where $\chi_{\lambda}(x)$ is
the characteristic function of $B_{\lambda}$. Then
$g(x)=\sum_{\lambda}g_{\lambda}(x)$. Let
\begin{equation}\label{FK2}
E(x,y)= \int_{\sigma}^{\sigma+r} e^{
-\int_{\sigma}^{\sigma+r}V(\sigma+s,\omega(s)+\omega_{0}(s))ds}
d\mu(\omega),
\end{equation}
where $d\mu(\omega)$ is an $n$-dimensional harmonic oscillator
(Ornstein-Uhlenbeck) probability measure on the continuous paths
$\omega: [\sigma,\sigma+r]\rightarrow \mathbb{R}$ with the boundary
condition $\omega(\sigma)=\omega(\sigma+r)=0$ and
\begin{equation}
\omega_{0}(s)=e^{\alpha (\tau-s)}\frac{e^{2\alpha\sigma}-e^{2\alpha
s}}{e^{2\alpha\sigma}-e^{2\alpha \tau}}x+e^{\alpha
(\sigma-s)}\frac{e^{2\alpha\tau}-e^{2\alpha s}}{e^{2\alpha
\tau}-e^{2\alpha\sigma}}y. \label{eqn:B0}
\end{equation} It
is shown in the Appendix that
\begin{equation}
|\partial_yE(x,y)|\lesssim r\beta^{\frac{1}{2}}.
\label{EstimatePotential}
\end{equation}
Recall that $U_\alpha(\tau,\sigma)$ is the evolution generated
by $-L_\alpha$.
Let $U(x,y)$ and $U_0(x,y)$ be the integral kernels of the operators
$U_{\alpha}(\sigma+r,\sigma)$
and $e^{-rL_{0}}$, respectively. By Feynmann-Kac formula
\eqref{eqn:BNeg1}, proved in the Appendix, we have that
$U(x,y)=U_0(x,y)E(x,y)$. Then
\begin{align}
U_{\alpha}(\sigma+r,\sigma)P^{\alpha}g(x)&=
\int U_0(x,y)E(x,y)P^{\alpha}g(y)dy \label{FK1}\\
&= \sum_{\lambda}\int
U_0(x,y)E(x,y)
P^{\alpha}g_{\lambda}(y)dy =:A(x)+B(x),
\label{Decompose}
\end{align}
where $$A(x):=\sum_{\lambda}\int
U_0(x,y)E(x,\lambda)
P^{\alpha}g_{\lambda}(y)dy$$ and
$$B(x):=\sum_{\lambda}\int
U_0(x,y)[E(x,y)-E(x,\lambda)]
P^{\alpha}g_{\lambda}(y)dy.$$

First we estimate the function $A$. We rewrite $A(x)=
\int U_0(x,y)P^{\alpha}g_x(y)dy=\lb e^{-rL_0}P^{\alpha}g_x\rb (x)$
with $g_x(y)=\sum_{\lambda}E(x,\lambda)g_{\lambda}(y)$. Now by
Lemma \ref{EstiA} we have
\begin{align*}
\left\|\langle x \rangle^{-3}
A\right\|_{\infty} &=\left\|\langle x
\rangle^{-3}
e^{-rL_0}P^{\alpha}g_x\right\|_{\infty}\\
&\lesssim e^{-\alpha r}\sup_y |\langle y \rangle^{-3}
g_x(y)|.
\end{align*}
Since $|E(x,\lambda)|\leq 1$ and ${g_{\lambda}}$'s have disjoint
supports), we obtain $|g_x(y)|\leqslant
\sum_{\lambda}|g_{\lambda}|=|g|$. The last two inequalities give
\begin{equation}
\left\|\langle x \rangle^{-3}A\right\|_{\infty}\lesssim e^{-\alpha r}\left\|\langle x
\rangle^{-3}g\right\|_{\infty}.
\label{EstimateA}
\end{equation}

Next we estimate the function $B$. Using $U_0(x,y)>0$,
\eqref{EstimatePotential}, Mean Value Theorem and the fact that the
diameters of $B_\lambda$ are not greater than $1+r$, we obtain
\begin{equation}\label{EstiB}
|B(x)|\lesssim r(1+r)\beta^{\frac{1}{2}}
\int U_0(x,y)\sum_{\lambda}|P^{\alpha}g_{\lambda}(y)|dy
=r(1+r)\beta^{\frac{1}{2}}\lb
e^{-rL_0}\sum_{\lambda}|P^{\alpha}g_{\lambda}(y)|\rb(x).
\end{equation}
Thus by \eqref{EstiB}, Lemma \ref{kernelEst} and the relation
$|g|=\sum_{\lambda}|g_{\lambda}|$,
\begin{equation}
\left\|\langle x \rangle^{-3}B\right\|_{\infty}\lesssim
r(1+r)\beta^{\frac{1}{2}}e^{2\alpha r}\left\|\langle x
\rangle^{-3}g\right\|_{\infty}.
\label{EstimateB}
\end{equation}
Combining \eqref{Decompose}, \eqref{EstimateA} and
\eqref{EstimateB}, we obtain the estimate \eqref{EstiAB}. This
proves Proposition \ref{propagator}.
\end{proof}

We will also need the following lemma
\begin{lemma}\label{LM:FK}
\begin{equation}\label{Mehler}
\|\langle z\rangle^{-k}
U_{\alpha}(\tau,\sigma)g\|_{\infty}\leq
e^{2\alpha(\tau-\sigma)}\|\langle z\rangle^{-k}
g\|_{\infty}
\end{equation} with $k=0\ \text{or}\ 3.$
\end{lemma}
\begin{proof}
By Equations (\ref{FK2}) and (\ref{eqn:BNeg1}) we have that
$|U_{\alpha}(\tau,\sigma)|(x,y)\leq e^{-L_{0}(\tau-\sigma)}(x,y).$
Thus we have
\begin{equation}\label{FeyKac}
\begin{array}{lll}
\|\langle z\rangle^{-k}
U_{\alpha}(\tau,\sigma)g\|_{\infty}& \le & \|\langle
z\rangle^{-k}
e^{-L_{0}(\tau-\sigma)}|g|\|_{\infty}.
\end{array}
\end{equation}  Now we use Lemma \ref{kernelEst} to
estimate the right hand side to complete the proof.
\end{proof}

{\textbf{Proof of Theorem \ref{ProP}}}.  Recall that
$\bar{P}_\alpha$ is the projection on the span of the three first
eigenfunctions of the operator $L_0$ and
$P^{\alpha}:=1-\bar{P}^{\alpha}$. We write
\begin{equation}\label{TranForm}
L_{\alpha}=P^\alpha L_{\alpha} P^\alpha+\bar{P}^\alpha L_{\alpha
}\bar{P}^\alpha+E_1,
\end{equation}
where the operator $E_{1}$ is defined as $E_{1}:=\bar{P}^{\alpha}
L_{\alpha }P^\alpha+P^\alpha L_{\alpha }\bar{P}^\alpha.$ Using that
$\bar{P}^\alpha P^\alpha=0$, we transform $E_1$ as
$$
\begin{array}{lll}
E_{1}&=&-\bar{P}^{\alpha}\frac{2p\alpha z\tilde{\beta}z}{(p-1)(p-1+z\tilde{\beta}z)}P^{\alpha}-
P^{\alpha}\frac{2p\alpha z\tilde{\beta}z}{(p-1)(p-1+z\tilde{\beta}z)}\bar{P}^{\alpha}.
\end{array}
$$  This relation implies that
\begin{equation}\label{EstD2}
\|\langle z\rangle^{-3}
E_{1}\eta(\sigma)\|_{\infty}\lesssim
\beta(\tau(\sigma))\|\langle z\rangle^{-3}
\eta(\sigma)\|_{\infty}.
\end{equation}
We use the Duhamel principle to rewrite the propagator
$V_{\alpha}(\sigma_{1},\sigma_{2})$ on ${\rm Ran}\, P^\alpha$
as
\begin{equation}
V_{\alpha}(\sigma_{1},\sigma_{2})P^\alpha=
U_{\alpha}(\sigma_{1},\sigma_{2})P^\alpha-
\int_{\sigma_{2}}^{\sigma_{1}}
U_{\alpha}(\sigma_{1},s)E_{1}V_{\alpha}(s,\sigma_{2}) P^\alpha
ds. \label{eqn:106}
\end{equation}
Let $r=\sigma_{1}-\sigma_{2}$, $g\in \Ran P^{\alpha}$ and
$\eta(\sigma_{1}):=V_{\alpha}(\sigma_{1},\sigma_{2})g$. We
claim that if $e^{ \alpha r}\leq \beta(\tau(\sigma_{2}))^{-1/8}$
then we have
\begin{equation}\label{firstIter}
\|\langle z\rangle^{-3}
\eta(\sigma_{1})\|_{\infty}\lesssim e^{- \alpha
r}\|\langle z\rangle^{-3}
\eta(\sigma_{2})\|_{\infty}.
\end{equation} To prove the claim we compute each term on the right
hand side of \eqref{eqn:106}.
\begin{itemize}
\item[(A)] Notice that $P^{\alpha}\eta(s)=\eta(s)$.  We use
Proposition \ref{propagator} to obtain for $e^{\alpha r}\leq
\beta(\tau(\sigma_{2}))^{-1/8}$ that
\begin{equation}\label{SecTerm} \|\langle
z\rangle^{-3}
U_{\alpha}(\sigma_{1},\sigma_{2})g\|_{\infty}\lesssim
e^{-\alpha r}\|\langle z\rangle^{-3}
g\|_{\infty}.\end{equation}

\item[(B)]
By Lemma \ref{LM:FK} and \eqref{EstD2} we obtain
\begin{equation*}
\begin{array}{lll}
\|\langle z\rangle^{-3}
\int_{\sigma_{2}}^{\sigma_{1}}U_{\alpha}(\sigma_{1},s)E_{1}\eta(s)ds\|_{\infty}
\lesssim
 \int_{\sigma_{2}}^{\sigma_{1}}e^{2\alpha(\sigma_{1}-s)}\beta(\tau(s))^{\frac{1}{2}}\|\langle
z\rangle^{-3} \eta(s)\|ds.
\end{array}
\end{equation*}

Using the condition $e^{\alpha r}\leq \beta(\sigma_{2})^{-1/8}$ and
the relation $\beta(\tau(s))\leq \beta(\tau(\sigma_{2}))$ for $s\geq
\sigma_{2}$ again, we find
\begin{equation}\label{ThirTerm}
\|\langle
z\rangle^{-3}  \int_{\sigma_{2}}^{\sigma_{1}}U_{\alpha}(\sigma_{1},s)E_{1}\eta(s)ds\|_{\infty}\\
\lesssim
 \int_{\sigma_{2}}^{\sigma_{1}}e^{-\alpha(\sigma_{1}-s)}\beta(\tau(s))^{\frac{1}{2}}\|\langle
z\rangle^{-3} \eta(s)\|ds.
\end{equation}
\end{itemize}
Equations \eqref{eqn:106}, \eqref{SecTerm} and \eqref{ThirTerm}
imply for $e^{\alpha r}\leq \beta^{-1/8}(\tau(\sigma_{2}))$ that
(remember that $\eta(\sigma_{2})=g$)
\begin{equation}\label{FinalTerm}
\begin{array}{lll}
\|\langle z\rangle^{-3}
\eta(\tau)\|_{\infty}&\lesssim & e^{-\alpha r}\|\langle
z\rangle^{-3} \eta(\sigma_{2})\|_{\infty}
+
\int_{\sigma_{2}}^{\tau}e^{-\alpha(\tau-s)}\beta(\tau(s))^{\frac{1}{2}}\|\langle
z\rangle^{-3} \eta(s)\|ds.
\end{array}
\end{equation}
Next, we define a function $K(r)$ as
\begin{equation}\label{defineKz}
K(r):=\max_{0\leq s\leq r}e^{\alpha s} \|\langle
z\rangle^{-3} \eta(\sigma_{2}+s)\|.
\end{equation}
Then \eqref{FinalTerm} implies that
$$K(\sigma_{1}-\sigma_{2})\lesssim \|\langle
z\rangle^{-3}
\eta(\sigma_{2})\|_{\infty}+
\int_{\sigma_{2}}^{\sigma_{1}}\beta(\tau(s))^{\frac{1}{2}}ds
K(\sigma_{1}-\sigma_{2}).$$ We observe that
$$ \int_{\sigma_{2}}^{\sigma_{1}}\beta(\tau(s))^{\frac{1}{2}}ds\leq
1/2$$ if $e^{\alpha r}\leq \beta(\tau(\sigma_{2}))^{-1/8}$ and if
$\beta(0)\ll 1$ and, therefore, $\beta(\tau(s))=
\frac{1}{\frac{1}{\beta(0)}+\frac{4p}{(p-1)^{2}} \tau(s)}$ are
small. Thus we have
$$K(\sigma_{1}-\sigma_{2})\lesssim \|\langle
z\rangle^{-3}
\eta(\sigma_{2})\|_{\infty},$$ which together with
Equation \eqref{defineKz} implies \eqref{firstIter}. Writing
$$V_\alpha(\tau,\sigma)=V_\alpha(\sigma_1,\sigma_2)
V_\alpha(\sigma_2,\sigma_3)\cdots
V_\alpha(\sigma_{m-1},\sigma_m)$$ with $\sigma_1=\tau$,
$\sigma_m=\sigma$ and $|\sigma_i-\sigma_{i+1}|=r$ such that
$e^{\alpha r}\leq \beta^{-1/8}(\tau(\sigma_k))\ \forall k$ and iterating
\eqref{firstIter} completes the proof of the theorem.
\begin{flushright}
$\square$
\end{flushright}

%%%%%%%%%%%%%%%%%%%%%%%%%%%%%%%%%%%%%%%%%%%%%%%%%%%%%%%%%%%%%%%%%%%%%%%%
%%%%%%%%%%%%%%%%%%%%%%%%%%%%%%%%%%%%%%%%%%%%%%%%%%%%%%%%%%%%%%%%%%%%%%%%
\section{Estimate of $M_{1}(\tau)$ (Equation \eqref{M1})}\label{SEC:EstM1}
In this subsection we derive an estimate for $M_{1}(T)$ given in
Equation \eqref{M1}.  Given any time $\tau'$, choose
$T=\tau'$ and pass from the unknown $\xi(y,\tau)$, $\tau\leq T,$
to the new unknown $\eta(z,\sigma),$ $\sigma\leq S,$ given in
\eqref{NewFun}. Now we estimate the latter function. To this end we
use Equation \eqref{eq:eta}. Observe that the function $\eta$ is
not orthogonal to the first three eigenvectors of the operator
$L_{0}$ defined in \eqref{eqn:DefnL0V}. Thus we apply the
projection $P^\alpha$ to Equation \eqref{eq:eta} to get
\begin{equation}\label{EQ:eta2}
\frac{d}{d\sigma}P^{\alpha}\eta=-P^{\alpha}L_{\alpha}P^{\alpha}\eta
+P^{\alpha}\sum_{k=1}^{4}D_{k},
\end{equation} where
we used the fact that $P^\alpha$ are $\tau$-independent and the
functions $D_{k}\equiv D_{k}(\sigma), \ k=1,2,3,4,$ are defined as
$$D_{1}:=-P^{\alpha}V\eta+P^{\alpha}VP^{\alpha}\eta,\ \ \ D_{2}:=W(a,b,\alpha)\eta,$$
$$D_{3}:=F(a,b,\alpha),\ \  D_{4}:=N(\eta,a,b,\alpha),$$ recall the definitions
of the functions $V$, $W$, $F$ and $N$ after \eqref{eqn:DefnL0V}.
\begin{lemma}\label{LM:EstDs} If $A(\tau),B_i(\tau)\leq
\beta^{-\frac{\kappa}{2}}(\tau)$ for $\tau\leq T$ and
$\|b_0\|\ll 1$, then we have
\begin{equation}\label{eq:estD1}
\|\langle z\rangle^{-3}D_{1}(\sigma)\|_{\infty}\lesssim
\beta^{5/2}(\tau(\sigma))M_{1}(T),
\end{equation}
\begin{equation}\label{eq:Remainder3}
\|\langle z\rangle^{-3}D_{2}(\sigma)\|_{\infty}\lesssim{\beta^{2+\frac{\kappa}{2}}(\tau
(\sigma))}M_{1}(T),
\end{equation}
\begin{equation}\label{eq:Festimate2}
\|\lra{z}^{-3} D_{3}(\sigma)\|_\infty
\lesssim\beta^{\min\{5/2,2p\}}(\tau(\sigma))[1+ M_1(T)(1+ A(T))+
M_1^2(T)+ M_1^{p}(T)],
\end{equation}
\begin{equation}\label{nonlinearity3}
\begin{array}{ll}
\|\langle z\rangle^{-3}D_{4}\|_{\infty}&\lesssim
\beta^{2}(\tau(\sigma))M_{1}(T)[\beta^{1/2}(\tau(\sigma))M_{1}(T)+M_{2}(T)\\
&+\beta^{\frac{p-1}{2}}(\tau(\sigma))M_{1}^{p-1}(T)+M_{2}^{p-1}(T)].
\end{array}
\end{equation}
\end{lemma}
\begin{proof}
In what follows we use the following estimates, implied by
\eqref{eq:appro},
\begin{equation}\label{eq:compare2}
\frac{\lambda_{1}}{\lambda}(t(\tau))-1=O(\beta(\tau)), \
\text{thus}\ \frac{\lambda_{1}}{\lambda}(t(\tau)),\
\frac{\lambda}{\lambda_{1}}(t(\tau))\leq 2,\ \langle
z\rangle^{-3}\lesssim \langle y\rangle^{-3}
\end{equation} where, recall that $z:=\frac{\lambda_{1}}{\lambda}y.$
We start with proving the following two estimates which will be used
frequently below
\begin{equation}\label{Compare0}
\|  \eta(\sigma)\|_{\infty}\lesssim
\beta^{1/2}(\tau(\sigma))M_{1}(\tau(\sigma))+M_{2}(\tau(\sigma))\leq
\beta^{1/2}(\tau(\sigma))M_{1}(T)+M_{2}(T),
\end{equation}
\begin{equation}\label{Compare3}
\|\langle z\rangle^{-3}
\eta(\sigma)\|_{\infty}\lesssim
\beta^{2}(\tau(\sigma))M_{1}(\tau(\sigma))\leq
\beta^{2}(\tau(\sigma))M_{1}(T).
\end{equation}
Denote by $\chi_{\geq D}$ and $\chi_{\leq D}$ the characteristic
functions of the sets $\{|x|\geq D\}$ and $\{|x|\leq D\}:$
\begin{equation}\label{cutoff} \chi_{\geq D}(x):=
\left\{\begin{array}{lll}
1,\, \,  \text{if}\ |x|\geq D\\
0,\, \,  \text{otherwise}
\end{array}\right.
\ \text{and}\ \chi_{\leq D}:=1-\chi_{\geq D}.
\end{equation} We take $D:=\frac{C}{\sqrt{\beta}}$ where
$C$ is a large constant. Writing  $1= 1 - \chi_{\geq D} + \chi_{\geq
D}$ and using the inequality $1-\chi_{\geq D}\lesssim
\beta^{-3/2}(\tau)\langle y\rangle^{-3}$, the relation between $\xi$
and $\eta$, see \eqref{NewFun}, and Estimate \eqref{eq:compare2}
we find
\begin{equation}\label{eq:keyEst}
\begin{array}{lll}
\|\eta(\sigma)\|_{\infty}\lesssim
\|\xi(\tau(\sigma))\|_{\infty}\lesssim
 \beta^{-3/2}(\tau(\sigma))\|\langle
y\rangle^{-3}\xi(\tau(\sigma))\|_{\infty}\\+
\|\chi_{\geq D}\xi(\tau)\|_{\infty}\ \leq
\beta^{1/2}(\tau(\sigma))M_{1}(\tau(\sigma))+M_{2}(\tau(\sigma))
\end{array}
\end{equation}
which is \eqref{Compare0}. Similarly recall that
$z=\frac{\lambda_{1}}{\lambda}y$ which together with \eqref{NewFun}
and \eqref{eq:compare2} yields
$$\|\langle z\rangle^{-3}\eta(\sigma)\|_{\infty}\lesssim
\|\langle
y\rangle^{-3}\xi(\tau(\sigma))\|_{\infty}\lesssim
\beta^{2}(\tau(\sigma))M_{1}(\tau(\sigma))\leq\beta^{2}(\tau(\sigma))M_{1}(T).$$
Thus we have \eqref{Compare3}.

Now we proceed directly to proving the lemma. First we rewrite
$D_{1}$ as
$$D_{1}(\sigma)=-P^{\alpha}\frac{2p\alpha z \tilde{\beta}(\tau(\sigma))z}
{(p-1)(p-1+z \tilde{\beta}(\tau(\sigma))z)}
(1-P^{\alpha})\eta(\sigma).$$ Now, using that $\langle
z\rangle^{-1}\frac{zbz}{p-1+zbz}\lesssim
\|b\|^{1/2}$ and that $\|b\|\lesssim \beta$, we obtain
$$
\begin{array}{lll}
\|\langle z\rangle^{-3}
D_{1}(\sigma)\|_{\infty}&\lesssim& \beta^{1/2}(\tau)|\|
\langle z\rangle^{-2}
(1-P^{\alpha})\eta(\sigma)\|_{\infty}.
\end{array}
$$
Next, because of the explicit form of $\bar{P}^{\alpha}:=1-P^{\alpha}$,
i.e. $\bar{P}^{\alpha}=\displaystyle|\phi_{0,\alpha} \rangle \langle
\phi_{0,\alpha}|+\sum_{i=1}^n|\phi_{1,\alpha}^{(i)} \rangle \langle
\phi_{1,\alpha}^{(i)}|+\sum_{i=1}^n|\phi_{2,\alpha}^{(i)} \rangle
\langle \phi_{2,\alpha}^{(i)}|+\sum_{1\leq i\neq j\leq
n}|\phi_{2,\alpha}^{(ij)} \rangle \langle \phi_{2,\alpha}^{(ij)}|,$
where $\phi_{m,\alpha}$ are the normalized eigenfunctions of the
operator $L_0:=-\Delta_z+\alpha z\cdot\p_z-2\alpha$, we have for any
function $g$
\begin{equation}\label{eq:estPro} \|\langle
z\rangle^{-2}\bar{P}^{\alpha}
g\|_{\infty}\lesssim \|\langle z\rangle^{-3}
g\|_{\infty}.
\end{equation} Collecting
the estimates above and using \eqref{Compare3}, we arrive at
$$\|\langle z\rangle^{-3}
D_{1}(\sigma)\|_{\infty}\lesssim
\beta^{1/2}(\tau(\sigma))\|\langle z\rangle^{-3}
\eta(\sigma)\|_{\infty} \lesssim
\beta^{5/2}(\tau(\sigma))M_{1}(T).$$

To prove \eqref{eq:Remainder3} we recall the definition of $D_{2}$
and rewrite it as
\begin{align*}
D_{2}=&\{[\frac{\lambda^{2}}{\lambda_{1}^{2}}-1]\frac{p(a+\frac{1}{2})}{p-1+yby}
+\frac{p(a-\alpha)}{p-1+yby}
+\frac{p(\frac{\lambda_{1}^{2}}{\lambda^{2}}-1)yby}
{(p-1+zbz)(p-1+yby)}  \\
&+\frac{p(\alpha-\frac{1}{2})}{p-1+yby}
-\frac{2p(\alpha-\frac{1}{2})}{p-1+z\tilde{\beta}z}+\frac{p z(\tilde{\beta}-b)z}{(p-1+zbz)(p-1+z\tilde{\beta}z)}\}\eta.
\end{align*}
Then Equations \eqref{eq:appro}, \eqref{CauchA} and the definition
of $B$ in \eqref{majorants} imply
$$\|\langle z\rangle^{-3}D_{2}(\sigma)\|_{\infty}\leq
\beta^{\frac{\kappa}{2}}(\tau(\sigma))\|\langle
z\rangle^{-3}\eta(\sigma)\|_{\infty}.$$
Using \eqref{Compare3} we obtain \eqref{eq:Remainder3} (recall
$\kappa:=\min\{\frac{1}{2},\frac{p-1}{2}\}$).

Now we prove \eqref{eq:Festimate2}. By \eqref{eq:compare2} and the
relation between $D_{3}$, $F$ and $\mathcal{F}$ we have
$$\|\langle z\rangle^{-3} D_{3}(\sigma)\|_\infty \lesssim
\|\langle y\rangle^{-3}
\mathcal{F}(a,b,c)(\tau(\sigma))\|_\infty$$ which together with
\eqref{eq:Festimate} implies \eqref{eq:Festimate2}.

Lastly we prove \eqref{nonlinearity3}. By the relation between
$D_{4}$, $N$ and $\mathcal{N}$ and the estimate in \eqref{eqn:69a}
we have
$$
\begin{array}{lll}
\|\langle z\rangle^{-3}D_{4}(\sigma)\|_{\infty}&\lesssim &\|\langle
y\rangle^{-3}\mathcal{N}(\xi(\tau(\sigma)),b(\tau(\sigma)),c(\tau(\sigma)))\|_{\infty}\\
&\lesssim &\|\langle y\rangle^{-3}\xi(\tau(\sigma))\|_{\infty}[\xi(\tau(\sigma))\|_{\infty}+\xi(\tau(\sigma))\|_{\infty}^{p-1}].
\end{array}
$$ Using \eqref{eq:keyEst} and the definition of $M_{1}$ we complete
the proof.
\end{proof}

Below we will need the following lemma. Recall that
$S:=\sigma(t(T))$.
\begin{lemma}\label{Bridge}  If $A(\tau)\leq
\beta^{-\frac{\kappa}{2}}(\tau)$, then for any $c_{1},c_{2}>0$ there
exists a constant $c(c_{1},c_{2})$ such that
\begin{equation}\label{INT}
 \int_{0}^{S}e^{-c_{1}(S-\sigma)}\beta^{c_{2}}(\tau(t(\sigma)))d\sigma\leq
c(c_{1},c_{2})\beta^{c_{2}}(T).
\end{equation}
\end{lemma}
\begin{proof}
We use the shorthand $\tau(\sigma)\equiv \tau(t(\sigma)),$ where,
recall, $t(\sigma)$ is the inverse of $\sigma(t)=
\int_{0}^{t}\lambda_{1}^{2}(k)dk$ and $\tau(t)=
\int_{0}^{t}\lambda^{2}(k)dk.$ By Proposition \ref{NewTrajectory}
we have that $\frac{1}{2}\leq \frac{\lambda}{\lambda_{1}}\leq 2$
provided that $A(\tau)\leq \beta^{-\frac{\kappa}{2}}(\tau)$. Hence
\begin{equation}\label{TauS1}
\frac{1}{4}\sigma\leq \tau(\sigma)\leq 4\sigma
\end{equation} which
implies
$\frac{1}{\frac{1}{\beta(0)}+\frac{4p}{(p-1)^{2}}\tau(\sigma)}\lesssim
\frac{1}{\frac{1}{\beta(0)}+\sigma}$.  By a direct computation we have
\begin{equation}\label{InI2}
 \int_{0}^{S}e^{-c_{1}(S-\sigma)}\beta^{c_{2}}(\tau(\sigma))d\sigma\leq
c(c_{1},c_{2})\frac{1}{(\frac{1}{\beta(0)}+\frac{4p}{p-1}S)^{c_{2}}}.
\end{equation}
Using \eqref{TauS1} again we obtain $4S\geq \tau(S)=T\geq \frac{1}{4}S$
which together with \eqref{InI2} implies \eqref{INT}.
\end{proof}

Recall that $V_{\alpha}(t,s)$ is the propagator generated by
the operator $-P^{\alpha} L_{\alpha}P^{\alpha}$. To estimate the
function $P^{\alpha}\eta$ we rewrite Equation \eqref{EQ:eta2} as
$$
P^{\alpha}\eta(S)=V_{\alpha}(S,0)P^{\alpha}\eta(0)+\displaystyle\sum_{n=1}^{4}
\int_{0}^{S}V_{\alpha}(S,\sigma)P^{\alpha}D_{n}(\sigma)d\sigma
$$ which implies
\begin{equation}\label{eq:estEsta}
\|\langle z\rangle^{-3}P^{\alpha}\eta(S)\|_{\infty}\leq K_{1}+K_{2}
\end{equation} with $$K_{1}:=\|\langle
z\rangle^{-3}V_{\alpha}(S,0)P^{\alpha}\eta(0)\|_{\infty};$$
$$K_{2}:=\|\langle z\rangle^{-3}\displaystyle\sum_{n=1}^{4}
\int_{0}^{S}V_{\alpha}(S,\sigma)P^{\alpha}D_{n}(\sigma)d\sigma\|_{\infty}.$$

Using Theorem \ref{ProP}, equation \eqref{Compare3} and the
slow decay of $\beta(\tau)$ we obtain
\begin{equation}
K_{1}\lesssim e^{-\cO S}\|\langle z\rangle^{-3}
P^{\alpha}\eta(0)\|_{\infty}\lesssim e^{-\cO S}\|\langle
z\rangle^{-3}\eta(0)\|_{\infty}\lesssim
\beta^{2}(T)M_{1}(0).
\end{equation}

By Theorem \ref{ProP}, equations \eqref{eq:estD1}-\eqref{nonlinearity3} and $
\int_{0}^{S}e^{-\cO(S-\sigma)}\beta^{2}(\tau(\sigma))d\sigma\lesssim
\beta^{2}(T)$ (see Lemma \ref{Bridge}) we have
\begin{equation}\label{M1Ge}
\begin{array}{lll}
K_{2}&\lesssim &\beta^{2}(T)
\{\beta^{\frac{\kappa}{2}}(0)[1+M_{1}(T)A(T)+M_{1}^{2}(T)+M_{1}^{p}(T)]\\
&&+[M_{2}(T)M_{1}(T)+M_{1}(T)M_{2}^{p-1}(T)]\}.
\end{array}
\end{equation}

Equation \eqref{NewFun} and the definitions of $S$ and $T$ imply
that $\lambda_{1}(t(S))=\lambda(t(T))$, $z=y$, $\eta(S)=\xi(T)$, and
$P^{\alpha}\xi=\xi$, consequently
\begin{equation}\label{eq:ini3}
\|\langle z\rangle^{-3}
P^{\alpha}\eta(S)\|_{\infty}=\|\langle
y\rangle^{-3}\xi(T)\|_{\infty}.
\end{equation}

Collecting the estimates \eqref{eq:estEsta}-\eqref{eq:ini3} and
using the definition of $M_{1}$ in \eqref{majorants} we have
$$
\begin{array}{lll}
M_{1}(T)&:=&\displaystyle\sup_{\tau\leq T}\beta^{-2}(\tau)\|\langle
y\rangle^{-3}\xi(\tau)\|_{\infty}\\
&\lesssim & M_{1}(0)+\beta^{\frac{\kappa}{2}}(0)[1+M_{1}(T)A(T)+M_{1}^{2}(T)+M_{1}^{p}(T)]\\
&&+M_{2}(T)M_{1}(T)+M_{1}(T)M_{2}^{p-1}(T)
\end{array}
$$
which together with the fact that $T$ is arbitrary implies Equation
\eqref{M1}.
\begin{flushright}
$\square$
\end{flushright}

%%%%%%%%%%%%%%%%%%%%%%%%%%%%%%%%%%%%%%%%%%%%%%%%%%%%%%%%%%%%%%%%%%%%%%%%
%%%%%%%%%%%%%%%%%%%%%%%%%%%%%%%%%%%%%%%%%%%%%%%%%%%%%%%%%%%%%%%%%%%%%%%%
\section{Estimate of $M_2$ (Equation \eqref{M2})}\label{SEC:EstM2}
The following lemma is proven similarly to the corresponding parts
of Lemma \ref{LM:EstDs} and therefore it is presented without a
proof.
\begin{lemma}\label{LM:keyest}
If $A(\tau), B(\tau)\leq \beta^{-\frac{\kappa}{2}}(\tau)$ and
$b_0\ll 1$ and $D_{k}(\sigma)$, $k=2,3,4$, are the same as in Lemma
\ref{LM:EstDs}, then
\begin{equation}\label{eq:Remainder0}
\|D_{2}(\sigma)\|_{\infty}\lesssim{\beta^{\frac{\kappa}{2}}(\tau
(\sigma))}[\beta^{1/2}(\tau(\sigma))M_{1}(T)+M_{2}(T)];
\end{equation}
\begin{equation}\label{eq:estF0}
\|D_{3}(\sigma)\|_\infty\lesssim
\beta^{\min\{1,2p-1\}}(\tau(\sigma))[1+ M_1(T)(1+ A(T))+ M_1^2(T)+
M_1^{p}(T)];
\end{equation}
\begin{equation}\label{nonlinearity0}
\|D_{4}(\sigma)\|_{\infty}\lesssim
\beta(\tau(\sigma))M_{1}^{2}(T)+M_{2}^{2}(T)+\beta^{p/2}(\tau(\sigma))M_{1}^{p}(T)+M_{2}^{p}(T).
\end{equation}
\end{lemma}

To estimate $M_2$ it is convenient to treat the $z$-dependent part
of the potential in \eqref{eqn:DefnL0V} as a perturbation.  Let the
operator $L_0$ be the same as in \eqref{eq:eta}. Rewrite
\eqref{eq:eta} to have
\begin{equation}
\eta(S)=e^{-(L_{0}+\frac{2p\alpha}{p-1}) S}\eta(0)+
\int_{0}^{S}e^{-(L_{0}+\frac{2p\alpha}{p-1})(S-\sigma)}(V_{2}\eta(\sigma)
+\displaystyle\sum_{k=2}^{4}D_{k}(\sigma))d\sigma,
\end{equation} where, recall $S:=\sigma(t(T)),$ $V_{2}$ is the operator given by
$$V_{2}:=\frac{2p\alpha}{p-1+z\tilde{\beta}(\tau(\sigma))z},$$ and the terms
$D_{n},\ n=2,3,4,$ are the same as in \eqref{EQ:eta2}. Lemma \ref{kernelEst}
implies that
$$\|e^{-(L_{0}+\frac{2p\alpha}{p-1}) s}g\|_{\infty}
=e^{-\frac{2p\alpha}{p-1}s}\|e^{-L_{0}s}g\|_{\infty} \lesssim
e^{-\frac{2\alpha}{p-1}s}\|g\|_{\infty}$$ for any function $g$ and time $s\geq 0.$
Hence we have
\begin{equation}\label{K123s}
\begin{array}{lll}
\|\eta(S)\|_{\infty}\lesssim
 K_{0}+K_{1}+K_{2}
\end{array}
\end{equation} where the functions $K_{i}$ are given by
$$K_{0}:=e^{-\frac{2\alpha}{p-1} S}\|
\eta(0)\|_{\infty};$$
$$K_{1}:= \int_{0}^{S}e^{-\frac{2\alpha}{p-1}(S-\sigma)}
\|V_{2}\eta(\sigma)\|_{\infty}d\sigma,$$
$$K_{2}:=\sum_{n=2}^{4} \int_{0}^{S}e^{-\frac{2\alpha}{p-1}(S-\sigma)}
\|D_{n}\|_{\infty}d\sigma.$$\\
We estimate the $K_{n}$'s, $n=0,1,2.$
\begin{itemize}
\item[(K0)] We start with $K_{0}$.
By \eqref{Compare0} and the decay of $e^{-\frac{2\alpha}{p-1} S}$ we have
\begin{equation}\label{eq:estK0}
K_{0}\lesssim M_{2}(0)+\beta^{1/2}(0)M_{1}(0).
\end{equation}
\item[(K1)] By the definition of $V_{2}$ we have
$$\|V_{2}\eta(\sigma)\|_{\infty}\lesssim
\|\frac{1}{\beta(\tau(\sigma))}\langle z\rangle^{-2}
\eta(\sigma)\|_{\infty}.$$ Moreover by the relation
between $\xi$ and $\eta$ in Equation \eqref{NewFun} and Proposition
\ref{NewTrajectory} we have
$$
\begin{array}{lll}
\displaystyle\max_{0\leq \sigma\leq S} \|
V_{2}\eta(\sigma)\|_{\infty} &\lesssim
\displaystyle\max_{0\leq \tau\leq T}\|\frac{1}{\beta} \langle
y\rangle^{-2}\xi(\tau)\|_{\infty}\\
&\leq \displaystyle\max_{0\leq \tau\leq T}\frac{1}{\beta}(\| \langle
y\rangle^{-3}\xi(\tau)\|_{\infty})^{\frac{2}{3}}
(\xi(\tau)\|_{\infty})^{\frac{1}{3}}\\
&\leq\beta^{\frac{1}{3}}(0)M_1^{\frac{2}{3}}(T)M_2^{\frac{1}{3}}(T).
\end{array}
$$
Therefore we obtain
\begin{equation}\label{EstK1}
\begin{array}{lll}
K_{1}&\lesssim & \displaystyle\max_{0\leq \sigma\leq
S}\|V_{2}\eta(\sigma)\|_{\infty}
\int_{0}^{S}e^{-\frac{2\alpha}{p-1}(S-\sigma)}d\sigma\\
& \lesssim &
\beta^{\frac{1}{3}}(0)M_1^{\frac{2}{3}}(T)M_2^{\frac{1}{3}}(T).
\end{array}
\end{equation}
\item[(K2)] By the definitions of $D_{k}, \ k=2,3,4,$ and
Equations \eqref{eq:Remainder0}-\eqref{nonlinearity0} we have
\begin{equation*}
\begin{array}{lll}
\sum_{k=2}^{4}\|D_{k}(\sigma)\|_\infty & \lesssim &
\beta^{\frac{\kappa}{2}}(\tau(\sigma))[1+M_{2}(T)+M_{1}(T)A(T)
+M_{1}^{2}(T)+M_{1}^{p}(T)]\\
&&+M_{2}^{2}(T)+M_{2}^{p}(T)
\end{array}
\end{equation*}
and consequently
\begin{equation}\label{EstK2}
\begin{array}{lll}
K_{2} & \lesssim &
\beta^{\frac{\kappa}{2}}(0)[1+M_{2}(T)+M_{1}(T)A(T)+M_{1}^{2}(T)
+M_{1}^{p}(T)]\\
&&+M_{2}^{2}(T)+M_{2}^{p}(T).
\end{array}
\end{equation}
\end{itemize}
Collecting the estimates \eqref{K123s}-\eqref{EstK2} we have
\begin{equation}\label{FinalStep}
\begin{array}{lll}
\|\eta(S)\|_{\infty} &\lesssim&
M_{2}(0)+\beta^{1/2}(0)M_{1}(0)
+\beta^{\frac{1}{3}}(0)M_1^{\frac{2}{3}}(T)M_2^{\frac{1}{3}}(T)\\
& &+\beta^{\frac{\kappa}{2}}(0)[1+M_{2}(T)+M_{1}(T)A(T)+M_{1}^{2}(T)
+M_{1}^{p}(T)]+M_{2}^{2}(T)+M_{2}^{p}(T).
\end{array}
\end{equation}
The relation between $\xi$ and $\eta$ in Equation \eqref{NewFun}
implies
\begin{equation*}
\|\xi(T)\|_{\infty} =\|
\eta(S)\|_{\infty}
\end{equation*}
which together with \eqref{FinalStep} gives
$$
\begin{array}{lll}
M_{2}(T)&\lesssim & M_{2}(0)+\beta^{1/2}(0)M_{1}(0)
+\beta^{\frac{1}{3}}(0)M_1^{\frac{2}{3}}(T)M_2^{\frac{1}{3}}(T)
+M_{2}^{2}(T)+M_{2}^{p}(T)\\
& &+\beta^{\frac{\kappa}{2}}(0)[1+M_{2}(T)+M_{1}(T)A(T)+M_{1}^{2}(T)+M_{1}^{p}(T)].
\end{array}
$$
Since $T$ is an arbitrary time, the proof of the estimate \eqref{M2}
for $M_2$ is complete.

%\appendix

\section*{Appendix 1: Feynmann-Kac Formula}\label{Sec:Trotter}
In this appendix we present, for the reader's convenience, a proof
of the Feynmann-Kac formula $U(x,y)=U_0(x,y)E(x,y)$ and the estimate
\eqref{EstimatePotential} used in section \ref{Section:PropEst} (cf.
\cite{BrKu, DGSW}). For stochastic calculus proofs of similar formulae see
\cite{Du, GlJa, Hida, KaSh, Simon}.
%A similar version can be found in .

Let $L_{0}:=-\Delta_y+\frac{\alpha^{2}}{4}|y|^{2}-\frac{\alpha}{2}$
and $L:=L_{0}+V$ where $V$ is a multiplication operator by a
function $V(y,\tau)$, which is bounded and Lipschitz continuous in
$\tau$. Let $U(\tau,\sigma)$ and $U_{0}(\tau,\sigma)$ be the
propagators generated by the operators $-L$ and $-L_{0},$
respectively. The integral kernels of these operators will be
denoted by $U(\tau,\sigma)(x,y)$ and $U_0(\tau,\sigma)(x,y)$.
\begin{thm}\label{THM:trotter}
The integral kernel of $U(\tau,\sigma)$ can be represented as
\begin{equation}
U(\tau,\sigma)(x,y)=U_{0}(\tau,\sigma)(x,y) \int e^{
\int_{\sigma}^{\tau}V(\omega_{0}(s)+\omega(s),s)ds}d\mu(\omega)
\label{eqn:BNeg1}\end{equation}
where $d\mu(\omega)$ is a
probability measure (more precisely, a conditional harmonic
oscillator, or Ornstein-Uhlenbeck, probability measure) on the
continuous paths $\omega: [\sigma,\tau]\rightarrow\R^n$ with
$\omega(\sigma)=\omega(\tau)=0$, and $\omega_{0}(\cdot)$ is the path
defined as
\begin{equation}
\omega_{0}(s)=e^{\alpha (\tau-s)}\frac{e^{2\alpha\sigma}-e^{2\alpha
s}}{e^{2\alpha\sigma}-e^{2\alpha \tau}}x+e^{\alpha
(\sigma-s)}\frac{e^{2\alpha\tau}-e^{2\alpha s}}{e^{2\alpha
\tau}-e^{2\alpha\sigma}}y. \label{eqn:B0}
\end{equation}
\end{thm}
\begin{remark}
\label{remark:OrnsteinUhlenbeck} $d\mu(\omega)$ is the Gaussian
measure with mean zero and covariance $(-\p_s^2+\alpha^2)^{-1}$,
normalized to 1. The path $\omega_0(s)$ solves the boundary value
problem
\begin{equation}
(-\p_s^2+\alpha^2)\omega_0=0\ \mbox{with}\ \omega(\sigma)=y\
\mbox{and}\ \omega(\tau)=x. \label{eqn:133a}
\end{equation}
Below we will also deal with the normalized Gaussian measure
$d\mu_{x y}(\omega)$ with mean $\omega_0(s)$ and covariance
$(-\p_s^2+\alpha^2)^{-1}$. This is a conditional Ornstein-Uhlenbeck
probability measure on continuous paths
$\omega:[\sigma,\tau]\rightarrow\R^n$ with $\omega(\sigma)=y$ and
$\omega(\tau)=x$ (see e.g.\ \cite{GlJa, Hida, Simon}).
\end{remark} Now, assume in addition that the function $V(y,\tau)$ satisfies the
estimates \begin{equation}\label{eq:diffeV} V\leq 0\ \text{and}\
|\partial_{y}V(y,\tau)|\lesssim
\beta^{\frac{1}{2}}(\tau)\end{equation} where $\beta(\tau)$ is a
positive function. Then Theorem \ref{THM:trotter} implies Equation
(\ref{EstimatePotential}) by the following corollary.
\begin{cor}\label{cor:diff} Under (\ref{eq:diffeV}),
$$|\partial_{y}  \int e^{  \int_{\sigma}^{\tau}V(\omega_{0}(s)+\omega(s),s)ds}d\mu(\omega)|\lesssim
|\tau-\sigma|\sup _{\sigma\leq s\leq \tau} \beta^{\frac{1}{2}}(\tau) $$
\end{cor}
\begin{proof}
By Fubini's theorem $$\partial_{y}  \int e^{
\int_{\sigma}^{\tau}V(\omega_{0}(s)+\omega(s),s)ds}d\mu(\omega)=
\int
\partial_{y}[ \int_{0}^{\tau}V(\omega_{0}(s)+\omega(s),s)ds]e^{ \int_{\sigma}^{\tau}V(\omega_{0}(s)+\omega(s),s)ds}d\mu(\omega)$$
Equation (\ref{eq:diffeV}) implies$$ |\partial_{y}
\int_{\sigma}^{\tau}V(\omega_{0}(s)+\omega(s),s)ds|\leq
|\tau-\sigma| \sup _{\sigma\leq s\leq \tau}
\beta^{\frac{1}{2}}(\tau) ,\ \text{and}\ e^{
\int_{\sigma}^{\tau}V(\omega_{0}(s)+\omega(s),s)ds}\leq 1.$$ Thus
$$|\partial_{y}  \int
e^{
\int_{\sigma}^{\tau}V(\omega_{0}(s)+\omega(s),s)ds}d\mu(\omega)|\lesssim
|\tau-\sigma| \sup _{\sigma\leq s\leq \tau}
\beta^{\frac{1}{2}}(\tau)  \int d\mu(\omega)=|\tau-\sigma| \sup
_{\sigma\leq s\leq \tau} \beta^{\frac{1}{2}}(\tau)
$$ to complete the proof.
\end{proof}

\begin{proof}[Proof of Theorem \ref{THM:trotter}]
We begin with the following extension of the Ornstein-Uhlenbeck
process-based Feynman-Kac formula to time-dependent potentials:
\begin{equation}
U(\tau,\sigma)(x,y)=U_0(\tau,\sigma)(x,y) \int e^{ \int_\sigma^\tau
V(\omega(s),s)\, ds}d\mu_{x y}(\omega). \label{eqn:B2}
\end{equation}
where $d\mu_{x y}(w)$ is the conditional Ornstein-Uhlenbeck
probability measure described in Remark
\ref{remark:OrnsteinUhlenbeck} above. This formula can be proven in
the same way as the one for time independent potentials (see
\cite{GlJa}, Equation (3.2.8)), i.e. by using the Kato-Trotter
formula and evaluation of Gaussian measures on cylindrical sets.
Since its proof contains a slight technical wrinkle, for the
reader's convenience we present it below.

Now changing the variable of integration in \eqref{eqn:B2} as
$\omega=\omega_0+\tilde{\omega}$, where $\tilde{\omega}(s)$ is a
continuous path with boundary conditions
$\tilde{\omega}(\sigma)=\tilde{\omega}(\tau)=0$, using the
translational change of variables formula $ \int f(\omega)\,
d\mu_{xy}(\omega)= \int f(\omega_0+\tilde{\omega})\,
d\mu(\tilde{\omega})$, which can be proven by taking
$f(\omega)=e^{i\langle\omega,\zeta\rangle}$ and using
\eqref{eqn:133a} (see \cite{GlJa}, Equation (9.1.27)) and omitting
the tilde over $\omega$ we arrive at \eqref{eqn:BNeg1}.
\end{proof}
There are at least three standard ways to prove \eqref{eqn:B2}: by
using the Kato-Trotter formula, by expanding both sides of the
equation in $V$ and comparing the resulting series term by term and
by using Ito's calculus (see \cite{KaSh, Simon,RSII,GlJa}). The
first two proofs are elementary but involve tedious estimates while
the third proof is based on a fair amount of stochastic calculus.
For the reader's convenience, we present the first elementary proof
of \eqref{eqn:B2}.

Before starting proving \eqref{eqn:B2} we establish an auxiliary
result. We define the operator $\mathcal{K}$ as
\begin{equation}\label{eqn:defK}
\mathcal{K}(\sigma,\delta):=
\int_{0}^{\delta}U_{0}(\sigma+\delta,\sigma+s)V(\sigma+s,\cdot)U_{0}(\sigma+s,\sigma)ds-U_{0}(\sigma+\delta,\sigma)
\int_{0}^{\delta} V(\sigma+s,\cdot)ds
\end{equation}

\begin{lemma}For any $\xi\in\mathcal{C}_{0}^{\infty}$ we
have, as $\delta\rightarrow 0^+$,
\begin{equation}\label{eqn:appro}
\sup_{0\leq \sigma\leq
\tau}\|\frac{1}{\delta}\mathcal{K}(\sigma,\delta)U(\sigma,0)\xi\|_{2}\rightarrow
0.
\end{equation}
\end{lemma}
\begin{proof}
If the potential term, $V$, is independent of $\tau$, then the proof
is standard (see, e.g.\ \cite{RSII}). We use the property that the
function $V$ is Lipschitz continuous in time $\tau$ to prove
(\ref{eqn:appro}). The operator $\mathcal{K}$ can be further
decomposed as
$$\mathcal{K}(\sigma,\delta)=\mathcal{K}_{1}(\sigma,\delta)+\mathcal{K}_{2}(\sigma,\delta)$$ with
$$\mathcal{K}_{1}(\sigma,\delta):= \int_{0}^{\delta}U_{0}(\sigma+\delta,\sigma+s)
V(\sigma,\cdot)U_{0}(\sigma+s,\sigma)ds-\delta
U_{0}(\sigma+\delta,\sigma) V(\sigma,\cdot)$$ and
$$\mathcal{K}_{2}(\sigma,
\delta):= \int_{0}^{\delta}U_{0}(\sigma+\delta,\sigma+s)
[V(\sigma+s,\cdot)-V(\sigma,\cdot)]U_{0}(\sigma+s,\sigma)ds-U_{0}(\sigma+\delta,\sigma)
 \int_{0}^{\delta} [V(\sigma+s,\cdot)-V(\sigma,\cdot)]ds.$$

Since $U_{0}(\tau,\sigma)$ are uniformly $L^{2}$-bounded and $V$ is
bounded, we have $U(\tau,\sigma)$ is uniformly $L^2$-bounded. This
together with the fact that the function $V(\tau,y)$ is Lipschitz
continuous in $\tau$ implies that
$$\|\mathcal{K}_{2}(\sigma,\delta)\|_{L^{2}\rightarrow L^{2}}\lesssim 2 \int_{0}^{\delta}sds=\delta^{2}.$$

We rewrite $\mathcal{K}_{1}(\sigma,\delta)$ as
$$\mathcal{K}_{1}(\sigma,\delta)= \int_{0}^{\delta} U_{0}
(\sigma+\delta,\sigma+s)\{V(\sigma,\cdot)[U_{0}(\sigma+s,\sigma)-1]-
[U_{0}(\sigma+s,\sigma)-1]V(\sigma,\cdot)\}ds.$$  Let
$\xi(\sigma)=U(\sigma,0)\xi$.  We claim that for a fixed $\sigma\in
[0,\tau]$,
\begin{equation}
\|\mathcal{K}_{1}(\sigma,\delta)\xi(\sigma)\|_{2}=o(\delta).
\label{eqn:142a}
\end{equation}
 Indeed, the fact
$\xi_{0}\in \mathcal{C}_{0}^{\infty}$ implies that
$L_{0}\xi(\sigma),\ L_{0}V(\sigma)\xi(\sigma)\in L^{2}.$
Consequently (see \cite{RSI})
$$\lim_{s\rightarrow
0^+}\frac{(U_{0}(\sigma+s,\sigma)-1)g}{s}\rightarrow L_{0}g, $$ for
$g=\xi(\sigma)\ \text{or}\ V(\sigma,y)\xi(\sigma)$ which implies our
claim. Since the set of functions $\{\xi(\sigma)|\sigma\in
[0,\tau]\}\subset L_{0}L^{2}$ is compact and
$\|\frac{1}{\delta}K_{1}(\sigma,\delta)\|_{L^2\rightarrow L^2}$ is
uniformly bounded, we have \eqref{eqn:142a} as $\delta\rightarrow 0$
uniformly in $\sigma\in [0,\tau]$.

Collecting the estimates on the operators $\mathcal{K}_{i},\ i=1,2$,
we arrive at (\ref{eqn:appro}).
\end{proof}
\begin{lemma}
Equation \eqref{eqn:B2} holds.
\end{lemma}
\begin{proof} In order to simplify our notation, in the proof that follows we assume,
without losing generality, that $\sigma=0$.  We divide the proof
into two parts. First we prove that for any fixed $\xi\in
\mathcal{C}_{0}^{\infty}$ the following Kato-Trotter type formula
holds
\begin{equation}\label{eqn:trotter}
U(\tau,0)\xi=\lim_{m\rightarrow \infty}\prod_{0\leq k\leq
m-1}U_{0}(\frac{k+1}{m}\tau, \frac{k}{m}\tau)e^{
\int_{\frac{k\tau}{m}}^{\frac{(k+1)\tau}{m}}V(y,s)ds}\xi
\end{equation} in the $L^{2}$ space.
We start with the formula
$$
\begin{array}{lll}
& &U(\tau,0)-\displaystyle\prod_{0\leq k\leq m-1}U_{0}(\frac{k+1}{m}\tau, \frac{k}{m}\tau)e^{ \int_{\frac{k\tau}{m}}^{\frac{(k+1)\tau}{m}}V(y,s)ds}\\
&=&\displaystyle\prod_{0\leq k\leq m-1}U(\frac{k+1}{m}\tau, \frac{k}{m}\tau)-\prod_{0\leq k\leq m-1}U_{0}(\frac{k+1}{m}\tau, \frac{k}{m}\tau)e^{ \int_{\frac{k\tau}{m}}^{\frac{(k+1)\tau}{m}}V(y,s)ds}\\
&=&\displaystyle\sum_{0\leq j\leq m}\prod_{j\leq k\leq
m-1}U_{0}(\frac{k+1}{m}\tau, \frac{k}{m}\tau)e^{
\int_{\frac{k\tau}{m}}^{\frac{(k+1)\tau}{m}}V(y,s)ds}A_{j}U(\frac{j}{m}\tau,
0)
\end{array}
$$ with the operator $$A_{j}:=U_{0}(\frac{j+1}{m}\tau, \frac{j}{m}\tau)e^{ \int_{\frac{j\tau}{m}}^{\frac{(j+1)\tau}{m}}V(y,s)ds}-U(\frac{j+1}{m}\tau, \frac{j}{m}\tau).$$

We observe that $\|U_{0}(\tau,\sigma)\|_{L^{2}\rightarrow L^{2}}\leq
1$, and moreover, by the boundness of $V,$ the operator
$U(\tau,\sigma)$ is uniformly bounded in $\tau$ and $\sigma$ in any
compact set. Consequently
\begin{equation}\label{eq:b4}
\begin{array}{lll}
& &\|[U(\tau,0)-\displaystyle\prod_{0\leq k\leq m-1}U_{0}(\frac{k+1}{m}\tau, \frac{k}{m}\tau)e^{ \int_{\frac{k\tau}{m}}^{\frac{(k+1)\tau}{m}}V(y,s)ds}]\xi\|_{2}\\
&\leq &\displaystyle\max_{j}m\|\displaystyle\prod_{j\leq k\leq m-1}U_{0}(\frac{k+1}{m}\tau, \frac{k}{m}\tau)e^{ \int_{\frac{k\tau}{m}}^{\frac{(k+1)\tau}{m}}V(y,s)ds}A_{j}U(\frac{j}{m}\tau, 0)\xi\|_{2}\\
&\lesssim
&m\displaystyle\max_{j}\|A_{j}+\mathcal{K}(\frac{k}{m}\tau,
\frac{1}{m}\tau)\|_{L^{2}\rightarrow
L^{2}}+\displaystyle\max_{j}m\|\mathcal{K}(\frac{j}{m}\tau,\frac{1}{m}\tau)U(\frac{j}{m},0)\xi\|_{2}
\end{array}
\end{equation}
where, recall the definition of $\mathcal{K}$ from
(\ref{eqn:defK}). Now we claim that
\begin{equation}\label{eqn:assum}
\|A_{j}+\mathcal{K}(\frac{k}{m}\tau,
\frac{1}{m}\tau)\|_{L^2\rightarrow L^2}\lesssim \frac{1}{m^{2}}.
\end{equation}

Indeed, by the Duhamel principle we have
$$U(\frac{j+1}{m}\tau,\frac{j}{m}\tau)=U_{0}(\frac{j+1}{m}\tau,\frac{j}{m}\tau)+ \int_{\frac{j}{m}\tau}^{\frac{j+1}{m}\tau}U_{0}(\frac{j+1}{m}\tau,s)V(y,s)U(s,\frac{j}{m}\tau)ds.$$
Iterating this equation on $U(s,\frac{k}{m}\tau)$ and using the fact
that $U(s,t)$ is uniformly bounded if $s,t$ is on a compact set, we
obtain
$$\|U(\frac{j+1}{m}\tau,\frac{j}{m}\tau)-U_{0}(\frac{j+1}{m}\tau,\frac{j}{m}\tau)- \int_{0}^{\frac{1}{m}\tau}U_{0}(\frac{j+1}{m}\tau,s)V(y,s)U_{0}(s,\frac{j}{m}\tau)ds\|_{L^{2}\rightarrow L^{2}}\lesssim \frac{1}{m^{2}}.$$
On the other hand we expand $e^{
\int_{\frac{j\tau}{m}}^{\frac{(j+1)\tau}{m}}V(y,s)ds}$ and use the
fact that $V$ is bounded to get
$$\|U_{0}(\frac{j+1}{m}\tau,\frac{j}{m}\tau)e^{ \int_{\frac{j\tau}{m}}^{\frac{(j+1)\tau}{m}}V(y,s)ds}-U_{0}(\frac{j+1}{m}\tau,\frac{j}{m}\tau)-
U_{0}(\frac{j+1}{m}\tau,\frac{j}{m}\tau)
\int_{\frac{j\tau}{m}}^{\frac{(j+1)\tau}{m}}V(y,s)ds\|_{L^{2}\rightarrow L^{2}}\lesssim \frac{1}{m^{2}}.$$
By the definition of $\mathcal{K}$ and $A_{j}$ we complete the proof
of (\ref{eqn:assum}). Equations (\ref{eqn:appro}), (\ref{eq:b4})
and (\ref{eqn:assum}) imply (\ref{eqn:trotter}). This completes
the first step.

In the second step we compute the integral kernel, $G_{m}(x,y)$, of
the operator $$G_{m}:=\displaystyle\prod_{0\leq k\leq
m-1}U_{0}(\frac{k+1}{m}\tau, \frac{k}{m}\tau)e^{
\int_{\frac{k\tau}{m}}^{\frac{(k+1)\tau}{m}}V(\cdot,s)ds}$$ in
(\ref{eqn:trotter}). By the definition, $G_{m}(x,y)$ can be written
as
\begin{equation}
G_{m}(x,y)= \int\cdot\cdot\cdot \int \prod_{0\leq k\leq
m-1}U_{\frac{\tau}{m}}(x_{k+1},x_{k})e^{
\int_{\frac{k\tau}{m}}^{\frac{(k+1)\tau}{m}}V(x_{k},s)ds}dx_{1}\cdot\cdot\cdot
dx_{m-1} \label{eqn:TrotterAsterik}
\end{equation}
 with $x_{m}:=x,\ x_{0}:=y$ and $U_\tau(x,y)\equiv U_0(0,\tau)(x,y)$
 is the integral kernel of the operator $U_0(\tau,0)=e^{-L_0\tau}$.
We rewrite \eqref{eqn:TrotterAsterik} as
\begin{equation}
 G_m(x,y)=U_\tau(x,y) \int e^{\sum_{k=0}^{m-1} \int_{\frac{k\tau}{m}}
 ^{\frac{(k+1)\tau}{m}} V(x_k,s)\, ds}\, d\mu_m(x_1,\ldots, x_m),
\label{eqn:TrotterDoubAsterik1}
\end{equation}
where
\begin{equation*}
d\mu_m(x_1,\ldots, x_m):=\frac{\prod_{0\le k\le
m-1}U_{\frac{\tau}{m}}(x_{k+1},x_k)}{U_\tau(x,y)}dx_1\ldots
dx_{k-1}.
\end{equation*}
Since $G_m(x,y)|_{V=0}=U_\tau(x,y)$ we have that $ \int
d\mu_m(x_1,\ldots,x_m)=1$.  Let
$\Delta:=\Delta_1\times\ldots\times\Delta_m$, where $\Delta_j$ is an
interval in $\R$.  Define a cylindrical set
\begin{equation*}
P^m_\Delta:=\{ \omega:[0,\tau]\rightarrow\R^n\ |\ \omega(0)=y,\
\omega(\tau)=x,\ \omega(k\tau/m)\in \Delta_k,\ 1\le k\le m-1 \}.
\end{equation*}
By the definition of the measure $d\mu_{xy}(\omega)$, we have
$\mu_{xy}(P^m_\Delta)= \int_{\Delta} d\mu_m(x_1,\ldots, x_m)$. Thus,
we can rewrite \eqref{eqn:TrotterDoubAsterik1} as
\begin{equation}
 G_m(x,y)=U_\tau(x,y) \int e^{\sum_{k=0}^{m-1}
  \int_{\frac{k\tau}{m}}^{\frac{(k+1)\tau}{m}} V(\omega(\frac{k\tau}{m}),s)\, ds}\, d\mu_{xy}(\omega),
\label{eqn:TrotterDoubAsterik2}
\end{equation}
By the dominated convergence theorem the integral on the right hand
side of \eqref{eqn:TrotterDoubAsterik2} converges in the sense of
distributions as $m\rightarrow\infty$ to the integral on the right
hand side of \eqref{eqn:B2}.  Since the left hand side of
\eqref{eqn:TrotterDoubAsterik2} converges to the left hand side of
\eqref{eqn:B2}, also in the sense of distributions (which follows
from the fact that $G_m$ converges in the operator norm on $L^2$ to
$U(\tau,\sigma)$), \eqref{eqn:B2} follows.
\end{proof}

Note that on the level of finite dimensional approximations the
change of variables formula can be derived as follows.  It is
tedious, but not hard, to prove that
$$\prod_{0\leq k\leq m-1}U_{m}(x_{k+1},x_{k})=e^{-\alpha\frac{(x-e^{-\alpha \tau}y)^{2}}{2(1-e^{-2\alpha \tau})}}\prod_{0\leq k\leq m-1}U_{m}(y_{k+1},y_{k})$$
with $y_{k}:=x_{k}-\omega_{0}(\frac{k}{m}\tau)$. By the definition
of $\omega_{0}(s)$ and the relations $x_{0}=y$ and $x_{m}=x$ we have
\begin{equation}\label{eq:gnzy}
G_{m}(x,y)=U_\tau(x,y)G^{(1)}_{m}(x,y)
\end{equation}
where
\begin{equation}\label{eq:g1xy}
G^{(1)}_{m}(x,y):=\frac{1}{4\pi \sqrt{\alpha}(1-e^{-2\alpha \tau})}
\int\cdot\cdot\cdot \int\prod_{0\leq k\leq
m-1}U_{m}(y_{k+1},y_{k})e^{
\int_{\frac{k\tau}{m}}^{\frac{(k+1)\tau}{m}}V(y_{k}+\omega_{0}(\frac{k\tau}{m}),s)ds}dy_{1}\cdot\cdot\cdot
dy_{k-1}.\end{equation} Since $\displaystyle\lim_{m\rightarrow
\infty}G_{m}\xi$ exists by (\ref{eqn:appro}), we have
$\displaystyle\lim_{m\rightarrow\infty}G^{(1)}_{m}\xi$ (in the weak
limit) exists also. As shown in \cite{GlJa},
$\displaystyle\lim_{m\rightarrow\infty}G^{(1)}_{m}= \int e^{
\int_{0}^{\tau}V(\omega_{0}(s)+\omega(s),s)ds}d\mu(\omega)$ with
$d\mu$ being the (conditional) Ornstein-Uhlenbeck measure on the set
of path from $0$ to $0.$ This completes the derivation of the change
of variables formula.
\begin{remark}
In fact, Equations \eqref{eqn:trotter}), \eqref{eq:gnzy} and
\eqref{eq:g1xy}) suffice to prove the estimate in Corollary
\ref{cor:diff}.
\end{remark}

\section*{Appendix 2: Computations and Proofs}\label{sec:comp}

\paragraph{Equation \eqref{eqn:A_1}: Computation of $A_1$.}
Here through some examples we show how to compute the matrix $A_1$.
We have
\begin{equation*}
\begin{array}{lll}
\ip{\p_a V_{\mu}}{\varphi_{az}^{ij}}&=&\lambda^{-n+\frac{2}{p-1}}\ip{\p_a V_{ab}}{\phi_a^{(ij)}}\\
&=&\frac{\lambda^{-n+\frac{2}{p-1}}}{p-1}(\frac{a+\frac{1}{2}}{p-1})^{\frac{1}{p-1}}\frac{1}{a+\frac{1}{2}}
\int e^{-\frac{a}{4}|y|^2}\phi_a^{(ij)}dy + O(\|b\|)\\
&=&\frac{\lambda^{-n+\frac{2}{p-1}}}{p-1}(\frac{a+\frac{1}{2}}{p-1})^{\frac{1}{p-1}}
\frac{1}{a+\frac{1}{2}}(\frac{2\pi}{a})^{n/2}\delta_{ij}
+O(\|b\|),
\end{array}
\end{equation*}
\begin{equation*}
\begin{array}{lll}
\ip{\p_{b_{ii}} V_{\mu}}{\varphi_{az}^{jj}}&=&\lambda^{-n+\frac{2}{p-1}}\ip{\p_{b_{ii}} V_{ab}}{e^{-\frac{a}{4}|y|^2}\phi_a^{(jj)}}\\
&=&-\frac{\lambda^{-n+\frac{2}{p-1}}}{(p-1)^2}(\frac{a+\frac{1}{2}}{p-1})^{\frac{1}{p-1}}
\int y_i^2e^{-\frac{a}{4}|y|^2}\phi_a^{(jj)}dy + O(\|b\|)\\
&=&\left\{
\begin{array}{l}
-\frac{\lambda^{-n+\frac{2}{p-1}}}{(p-1)^2}(\frac{a+\frac{1}{2}}{p-1})^{\frac{1}{p-1}}
\frac{3}{a}(\frac{2\pi}{a})^{n/2}
+O(\|b\|),\ \text{if}\ i=j,\\
-\frac{\lambda^{-n+\frac{2}{p-1}}}{(p-1)^2}(\frac{a+\frac{1}{2}}{p-1})^{\frac{1}{p-1}}
\frac{1}{a}(\frac{2\pi}{a})^{n/2}
+O(\|b\|),\ \text{if}\ i\neq j.
\end{array}
\right.
\end{array}
\end{equation*}
Similarly we can compute all the other entries.

\paragraph{Derivation of Equation \eqref{eqn:FluctuationGauged}-\eqref{eqn:G_1}.}
Let $v=V_{ab}+\xi$, then we have
\begin{equation}\label{eqn:v_derivative}
\begin{array}{lll}
\p_\tau v&=&\frac{1}{p-1}(\frac{c}{p-1+yby})^{\frac{1}{p-1}-1}
\frac{c_\tau(p-1+yby)-c y b_\tau y}{(p-1+yby)^2}+\xi_\tau,\\
\p_{y_i}v&=&-\frac{1}{p-1}(\frac{c}{p-1+yby})^{\frac{1}{p-1}}\frac{2\sum_j b_{ij}y_j}{p-1+yby}
+\p_{y_i}\xi,\\
\p_{y_i}^2 v&=&\frac{1}{(p-1)^2}(\frac{c}{p-1+yby})^{\frac{1}{p-1}}(\frac{2\sum_j b_{ij}y_j}{p-1+yby})^2
-\frac{1}{p-1}(\frac{c}{p-1+yby})^{\frac{1}{p-1}}\frac{2b_{ii}(p-1+yby)-(2\sum_j b_{ij}y_j)^2}{(p-1+yby)^2}\\
&&+\p_{y_i}^2\xi.
\end{array}
\end{equation}
Plugging \eqref{eqn:v_derivative} into \eqref{eqn:BVNLH} we obtain
\begin{equation*}
\begin{array}{ll}
&\p_\tau \xi+\frac{c_\tau(p-1+yby)-c yb_\tau y}{(p-1)(p-1+yby)c}V_{ab}\\
=&\frac{4\sum_i(\sum_j b_{ij}y_j)^2}{(p-1)^2(p-1+yby)^2}V_{ab}-\sum_i \frac{2b_{ii}(p-1+yby)-4(\sum_j b_{ij}y_j)^2}{(p-1)(p-1+yby)^2}V_{ab}\\
&+\Delta\xi+\frac{2a yby}{(p-1)(p-1+yby)}V_{ab}- a\sum_i y_i\p_{y_i}\xi-\frac{2a}{p-1}V_{ab}-\frac{2a}{p-1}\xi+|V_{ab}+\xi|^{p-1}(V_{ab}+\xi).
\end{array}
\end{equation*}
It follows that
\begin{equation*}
\begin{array}{lll}
\p_\tau \xi &=& (\Delta-ay\cdot\p_y-\frac{2a}{p-1}+\frac{pc}{p-1+yby})\xi+
|V_{ab}+\xi|^{p-1}(V_{ab}+\xi)-V_{ab}^p-pV_{ab}^{p-1}\xi\\
&&+(\frac{c}{p-1+yby}-\frac{2a}{p-1}+\frac{2a yby}{(p-1)(p-1+yby)}+\frac{4p\sum_i (\sum_j b_{ij}y_j)^2}{(p-1)^2(p-1+yby)^2}
-\frac{2\sum_i b_{ii}}{(p-1)(p-1+yby)}-\frac{c_\tau/c}{p-1}+\frac{yb_\tau y}{(p-1)(p-1+yby)})V_{ab}.
\end{array}
\end{equation*}
Rearranging the terms on the r.h.s.
we obtain the equations \eqref{eqn:FluctuationGauged}-\eqref{eqn:G_1} for $\xi$.

\paragraph{Proof of Lemma \ref{IdDecom}.}
%\begin{proof}
We prove this result by induction in the dimension $n$. For $n=1$,
the result is straightforward since
$1=P_0^{(1)}+P_1^{(1)}+P_2^{(1)}+P_3^{(1)}$.

Assume the statement of the lemma is true for all dimensions $m\leq n-1$ and we will
prove it for dimension $n$. By symmetry we only need to prove it for the
case $k$. We have by assumption
\begin{equation} \label{Decom1}
1=\sum_{\vec{i'}\in J_1^{'}}P_{\vec{i'}}
\end{equation}
where $J_1^{'}\subset I_1^{(n-1)}.$
We claim the following relations
\begin{equation}
\begin{array}{lll}
P_0^{(n)}&=&\sum_{\vec{i'}\in J_1^{'}}P_{\vec{i'}}P_0^{(n)}, \label{P0}
\end{array}
\end{equation}
\begin{equation}\label{P1}
\begin{array}{lll}
P_1^{(n)} & = &P_{0}^{(1)}\cdots
P_{0}^{(n-1)}P_1^{(n)}+\sum_{j=1}^{n-1}P_{0'}^{(1)}\cdots P_{0'}^{(j-1)}P_{1'}^{(j)}
P_0^{(j+1)}\cdots P_0^{(n-1)}P_1^{(n)}\\
&=& P_0^{(1)}\cdots P_0^{(n-1)}P_1^{(n)} +
\sum_{k<l} P_{0'}^{(1)}\cdots P_{0'}^{k-1}P_{1'}^{(k)}P_0^{(k+1)}\cdots P_0^{(l-1)}
P_{1'}^{(l)}P_0^{(l+1)}\cdots P_0^{(n-1)}P_1^{(n)} \\
&&+ \sum_{k=1}^{n-1}\sum_{l=1,2'} P_0^{(1)}\cdots P_0^{(k-1)}P_{l}^{(k)}P_0^{(k+1)}\cdots P_0^{(n-1)}P_1^{(n)},
\end{array}
\end{equation}
\begin{equation}We
P_2^{(n)}
=P_{0}^{(1)}\cdots
P_{0}^{(n-1)}P_2^{(n)}+\sum_{j=1}^{n-1}P_{0'}^{(1)}\cdots P_{0'}^{(j-1)}P_{1'}^{(j)}
P_0^{(j+1)}\cdots P_0^{(n-1)}P_2^{(n)},\label{P2}
\end{equation}
\begin{equation}
\begin{array}{lll}
P_3^{(n)}&=&P_{0'}^{(1)}\cdots P_{0'}^{(n-1)}P_3^{(n)}.
\label{P3}
\end{array}
\end{equation}
%By the definitions above we have
%\begin{equation}\label{Decom2}
%1=P_0^{(n)}+P_1^{(n)}+P_2^{(n)}+P_3^{(n)}.
%\end{equation}
%Using \eqref{Decom1}, the identities
%\eqref{eqn:projectdecom} and
In fact, \eqref{P0} follows directly from \eqref{Decom1}, and \eqref{P3} is trivial. Moreover, using the second relation in \eqref{eqn:projectdecom} we obtain
\begin{equation}\label{eqn:decompose1}
\begin{array}{lll}
&&1=P_{0'}^{(1)}\cdots P_{0'}^{(n-1)}=P_{0'}^{(1)}\cdots P_{0'}^{(n-2)}(P_{0}^{(n-1)}+P_{1'}^{(n-1)})\\
&=&P_{0'}^{(1)}\cdots P_{0'}^{(n-2)}P_{0}^{(n-1)} + P_{0'}^{(1)}\cdots P_{0'}^{(n-2)}P_{1'}^{(n-1)}\\
&=&P_{0'}^{(1)}\cdots P_{0'}^{(n-3)}(P_{0}^{(n-2)}+P_{1'}^{(n-2)})P_{0}^{(n-1)} +P_{0'}^{(1)}\cdots P_{0'}^{(n-2)}P_{1'}^{(n-1)}\\
&=& \cdots \cdots\\
&=& P_0^{(1)}\cdots P_0^{(n-1)}+\sum_{j=1}^{n-1}P_{0'}^{(1)}\cdots P_{0'}^{(j-1)}P_{1'}^{(j)}P_{0}^{(j+1)}\cdots P_{0}^{(n-1)},
\end{array}
\end{equation}
and
\begin{equation}\label{eqn:decompose2}
\begin{array}{lll}
&& P_{0'}^{(1)}\cdots P_{0'}^{(j-1)}P_{1'}^{(j)}=P_{0'}^{(1)}\cdots P_{0'}^{(j-2)}(P_{0}^{(j-1)}+P_{1'}^{(j-1)})P_{1'}^{(j)}\\
&=&P_{0'}^{(1)}\cdots P_{0'}^{(j-2)}P_{0}^{(j-1)}P_{1'}^{(j)}+ P_{0'}^{(1)}\cdots P_{0'}^{(j-2)}P_{1'}^{(j-1)}P_{1'}^{(j)}\\
&=&P_{0'}^{(1)}\cdots P_{0'}^{(j-3)}(P_{0}^{(j-2)}+P_{1'}^{(j-2)})P_{0}^{(j-1)}P_{1'}^{(j)}+
P_{0'}^{(1)}\cdots P_{0'}^{(j-2)}P_{1'}^{(j-1)}P_{1'}^{(j)}\\
&=& \cdots \cdots \\
&=& P_{0}^{(1)}\cdots P_{0}^{(j-1)}P_{1'}^{(j)}+
\sum_{k<j}P_{0'}^{(1)}\cdots P_{0'}^{(k-1)}P_{1'}^{(k)}P_{0}^{(k+1)}\cdots P_{0}^{(j-1)}P_{1'}^{(j)}.
\end{array}
\end{equation}
\eqref{P2} follows readily from \eqref{eqn:decompose1}. Finally, using \eqref{eqn:decompose1} and
\eqref{eqn:decompose2} we arrive at \eqref{P1}.
Thus by \eqref{P0}-\eqref{P3} and the relation $1=P_0^{(n)}+P_1^{(n)}+P_2^{(n)}+P_3^{(n)}$ we find
\begin{equation*}
\begin{array}{lll}
1 &=& \sum_{\vec{i'}\in
J_1^{'}}P_{\vec{i'}}P_0^{(n)}+P_0^{(1)}\cdots P_0^{(n-1)}P_1^{(n)} +
\sum_{k<l} P_{0'}^{(1)}\cdots P_{0'}^{k-1}P_{1'}^{(k)}P_0^{(k+1)}\cdots P_0^{(l-1)}
P_{1'}^{(l)}P_0^{(l+1)}\cdots P_0^{(n-1)}P_1^{(n)} \\
&&+ \sum_{k=1}^{n-1}\sum_{l=1,2'} P_0^{(1)}\cdots P_0^{(k-1)}P_{l}^{(k)}P_0^{(k+1)}\cdots P_0^{(n-1)}P_1^{(n)}+P_{0}^{(1)}P_{0}^{(2)}\cdots
P_{0}^{(n-1)}P_2^{(n)}\\
&&+\sum_{j=1}^{n-1}P_{0'}^{(1)}\cdots P_{0'}^{(j-1)}P_{1'}^{(j)}
P_0^{(j+1)}\cdots P_0^{(n-1)}P_2^{(n)}+P_{0'}^{(1)}\cdots
P_{0'}^{(n-1)}P_3^{(n)}.
\end{array}
\end{equation*}
Therefore we obtain $1=\sum_{\vec{i}\in J_n}P_{\vec{i}}$, where
\begin{equation*}
\begin{array}{ll}
J_n= & \{\vec{i}=(i_1,\cdots,i_{n-1},0)| (i_1,\cdots,i_{n-1})\in
J_1^{'}\}\bigcup \{(0,\cdots,0,k), (0',\cdots,0',3): k=1, 2\}\\
& %\bigcup(
\bigcup_{1 \le k<l \le n-1}\{(i_1, \cdots, i_{n-1}, 1): i_k=i_l=1', i_j=0'\ \text{if}\ j<k,
i_j=0\ \text{if}\ k<j<l\ \text{or}\ l<j<n\}\\ %)\\
& %\bigcup(
\bigcup_{k=1}^{n-1} \{(i_1, \cdots, i_{n-1}, 1): i_j=l\delta_{jk} \forall 1\le j\le n-1, l=1,2'\}\\ %)\\
& %\bigcup(
\bigcup_{k=1}^{n-1} \{(i_1, \cdots, i_{n-1}, 2): i_k=1', i_j=0'\ \text{if}\ j<k,
i_j=0\ \text{if}\ k<j<n\}. %).
\end{array}
\end{equation*}
Obviously this $J_n$ is a subset of $I_n$. This proves Lemma
\ref{IdDecom}. $\Box$
%\end{proof}

\paragraph{Proof of \eqref{x_jnorm} in the case $i_j=3$.}
Let $L_0=-\Delta+\alpha x\p_x$. We want to show
\begin{align*}
\norm{\langle x\rangle^{-3}\e^{-rL_0}P_3|x|^3}_{L_\infty\to L_\infty}\lesssim \e^{-3\alpha r}\,.
\end{align*}
Let $U_0(x,y)$ be the integral kernel
of $U_\alpha :=\e^{\frac{\alpha x^2}{4}}\e^{-rL_0}\e^{-\frac{\alpha x^2}{4}}$.
By a standard formula (see \cite{Simon, GlJa}) we have
\begin{equation*}
U_{0}(x,y)=4\pi (1-e^{-2\alpha
r})^{-\frac{1}{2}}\sqrt{\alpha}e^{2\alpha
r}e^{-\alpha\frac{(x-e^{-\alpha r}y)^{2}}{2(1-e^{-2\alpha r})}}.
\label{eqn:96a}
\end{equation*}
Define a new function $f:=e^{-\frac{\alpha y^{2}}{2}}P_3 g$.
The definitions above imply
\begin{equation}\label{FK1}
e^{\frac{\alpha x^{2}}{2}}U_{\alpha}(\sigma+r,\sigma)e^{-\frac{\alpha x^{2}}{2}}P_3 g= e^{-\frac{\alpha x^{2}}{2}}\int U_{0}(x,y)f(y)dy.
\end{equation}
Integrate by parts on the right hand side of \eqref{FK1} to obtain
\begin{equation}\label{Estimateta}
\begin{array}{lll}
e^{\frac{\alpha x^{2}}{2}}U_{\alpha}(\sigma+r,\sigma)e^{-\frac{\alpha x^{2}}{2}}P_3 g
&=&e^{\frac{\alpha x^{2}}{2}}\int \partial_{y}^{3}U_{0}(x,y)f^{(-3)}(y)dy
\end{array}
\end{equation}
where
$f^{(-m-1)}(x):= \int_{-\infty}^{x}f^{(-m)}(y)dy$ and $f^{(-0)}:=f.$
Because $P_3 g\perp y^{m}e^{-\frac{\alpha
y^{2}}{2}},\ m=0,1,2,$ we have that $f\perp 1,\ y,\ y^{2}.$ Therefore by integration by parts we have
$$f^{(-m)}(y)= \int_{-\infty}^{y}f^{(-m+1)}(x)dx=- \int_{y}^{\infty}f^{(-m+1)}(x)dx,\ m=1,2,3.$$
Moreover, by the definition of $f^{(-m)}$ and the equation above we
have
$$|f^{(-m)}(y)|\lesssim
e^{-\frac{\alpha y^{2}}{2}}\langle y\rangle^{3-m}\|\langle
y\rangle^{-3}P_3 g\|_{\infty}.$$
Using the explicit formula for $U_0(x,y)$ given above we find $$|\partial^{k}_{y}U_{0}(x,y)|\lesssim \frac{e^{-\alpha
kr}}{(1-e^{-2\alpha r})^{k}}(|x|+|y|+1)^{k}U_{0}(x,y).$$We

Collecting the estimates above and using Equation
\eqref{Estimateta}, we have the following result
$$
\begin{array}{lll}
& &\langle x\rangle^{-3}|e^{\frac{\alpha x^{2}}{2}} U_{\alpha}(\sigma+r,\sigma)e^{-\frac{\alpha x^{2}}{2}}P_3g(x)|\\
&\lesssim & \frac{1}{(1-e^{-2\alpha r})^{3}}\langle
x\rangle^{-3}e^{\frac{\alpha x^{2}}{2}} \int (|x|+|y|+1)^{3}e^{-3\alpha
r}U_{0}(x,y)|f^{(-3)}(y)|dy\\
&\lesssim &\frac{e^{-3\alpha
r}}{(1-e^{-2\alpha r})^{3}}
e^{\frac{\alpha x^{2}}{2}}\int \langle
x\rangle^{-3}U_{0}(x,y)e^{-\frac{\alpha}{2}y^{2}}\langle
y\rangle^{3}dy \|\langle
y\rangle^{-3}P_3g\|_{\infty}.
\end{array}
$$
$\Box$

%%%%%%%%%%%%%%%%%%%%%%%%%%%%%%%%%%%%%%%%%%%%%%%%%%%%%%%%%%%%%%%%%%%%%%%%
%%%%%%%%%%%%%%%%%%%%%%%%%%%%%%%%%%%%%%%%%%%%%%%%%%%%%%%%%%%%%%%%%%%%%%%%
\def\cprime{$'$} \def\cprime{$'$}

%\bibliographystyle{abbrv}
%\input{biblio.bbl} %bbl}
%\bibliography{biblio}
%\begin{thebibliography}{10}
%\end{thebibliography}
\end{document}